\documentclass[12pt]{amsart}
\usepackage{amsfonts}
\usepackage{amsthm}
\usepackage{amsmath}
\usepackage{amssymb}
\usepackage{amscd}
\usepackage[latin2]{inputenc}
\usepackage{t1enc}
\usepackage[mathscr]{eucal}
\usepackage{indentfirst}
\usepackage{graphicx}
\usepackage{graphics}
\usepackage{pict2e}
\usepackage{epic}
\usepackage{hyperref}
\usepackage{tikz-cd}
\numberwithin{equation}{section}
\usepackage[margin=2.54cm]{geometry}
\usepackage{epstopdf} 
\usepackage{enumerate, enumitem}
\usepackage{pgfplots}
\usepackage{multicol}
\usepackage{wrapfig}
\usepackage{amsmath}
\usepackage{thmtools} 
\usepackage[font=small]{caption}

\makeatletter
\@ifpackageloaded{babel}%
{%
	\addto\extrasenglish{%
	}%
}{\relax}
\makeatother

\newtheorem{theorem}{Theorem}[section]
\newtheorem{lemma}[theorem]{Lemma}
\newtheorem{conjecture}[theorem]{Conjecture}
\newtheorem{corollary}[theorem]{Corollary}
\newtheorem{claim}{Claim}[theorem]
\newtheorem{proposition}[theorem]{Proposition}

\newtheorem*{claim*}{Claim}

\usepackage{cleveref}
\crefname{section}{§\hspace{-3pt}}{§§}
\Crefname{section}{Section}{§§}

\theoremstyle{definition}
\newtheorem{definition}[theorem]{Definition}

\theoremstyle{remark}
\newtheorem{remark}[theorem]{Remark}
\newtheorem*{ack*}{Acknowledgements}

\newtheorem*{main*}{Main Theorem}

\newcommand{\diag}{\operatorname{diag}}
\newcommand{\ev}{\operatorname{ev}}

\newcommand{\inte}{\operatorname{int}}
\newcommand{\supp}{\operatorname{supp}}
\newcommand{\spec}{\bigsqcup_{i=1}^{\ell}(X_{i}\setminus Y_{i})}
\newcommand{\NN}{\mathbb{N}}

\pgfplotsset{every tick label/.append style={font=\Large}}
\pgfplotsset{xlabel/.append style={font=\Large}}

\usepackage{tikz-cd}
\usepackage{pgfplots}
\usepackage{float}
\usepackage{mathtools}
\renewcommand\qedsymbol{/\hspace{-2pt}/}
\usetikzlibrary{arrows,positioning} 
\tikzset{
	>=stealth',
	punkt/.style={
		rectangle,
		rounded corners,
		draw=black, very thick,
		text width=6.5em,
		minimum height=2em,
		text centered,
	},
}

\begin{document}

\title[The Stable Rank of Diagonal ASH Algebras]{The Stable Rank of Diagonal ASH Algebras and Crossed Products by Minimal Homeomorphisms}

\author{Mihai Alboiu}
\thanks{The first author was supported by an NSERC CGS D grant.}
\thanks{Dept. of Mathematics, Yale University, 
	442 Dunham Lab, 10 Hillhouse Ave, New Haven, CT, 06511, USA} 
\thanks{e-mails: mihai.alboiu@yale.edu (M. Alboiu), jlutley@gmail.com (J. Lutley)}
\address{Department of Mathematics, Yale University, 
442 Dunham Lab, 10 Hillhouse Ave, New Haven, CT, 06511, USA} 
\email{mihai.alboiu@yale.edu}

\author{James Lutley}
\email{jlutley@gmail.com}

\begin{abstract} We introduce a subclass of recursive subhomogeneous algebras, in which each of the pullback maps is diagonal in a suitable sense. We define the notion of a diagonal map between two such algebras and show that every simple inductive limit of these algebras with diagonal bonding maps has stable rank one. As an application, we prove that for any infinite compact metric space $T$ and minimal homeomorphism $h\colon T\to T$, the associated dynamical crossed product $\mathrm{C^*}(\mathbb{Z},T,h)$ has stable rank one. This affirms a conjecture of Archey, Niu, and Phillips. We also show that the Toms--Winter Conjecture holds for such crossed products.
\end{abstract}

\maketitle



\section{Introduction} 

With the aim of formulating a notion of dimension for a $\mathrm{C}^{*}$-algebra, in \cite{rie}, Rieffel  introduced the concept of stable rank. The \textit{stable rank} of a unital $\mathrm{C}^{*}$-algebra $A$ is the least natural number $n$ for which the set of all $n$-tuples of $A$ that generate $A$ as a left ideal is dense in $A^{n}$; if no such integer exists, the stable rank is said to be $\infty$. Of particular note is the instance when the stable rank is one. In \cite{rie}, Proposition 3.1, it is shown that a unital $\mathrm{C}^{*}$-algebra has stable rank one if and only if the set of invertible elements is dense within the algebra. An important problem in the field has been to determine when a $\mathrm{C}^{*}$-algebra has stable rank one. 

In \cite{Ror}, R\o{}rdam supplied one of the first major results concerning stable rank. He showed that the tensor product of a simple unital stably finite $\mathrm{C}^{*}$-algebra and a UHF algebra has stable rank one. This was followed by a result of D\u{a}d\u{a}rlat, Nagy,  N{\'{e}}methi, and Pasnicu, who proved in \cite{4Rom} that a simple unital inductive limit of full matrix algebras (those of the form $C(X,M_{n}(\mathbb{C}))$ for a compact Hausdorff space $X$) always has stable rank one assuming there is a uniform upper bound on the dimensions of the base spaces in the finite stage algebras. Later, in \cite{Ror2}, R\o{}rdam also showed that every simple unital finite $\mathrm{C}^{*}$-algebra that absorbs the Jiang--Su algebra, $\mathcal{Z}$, tensorially has stable rank one. 

Villadsen proved in \cite{vil1} that the converse to the result in \cite{4Rom} does not hold by constructing a unital simple limit of full matrix algebras, whose base space dimensions were not uniformly bounded above, yet which nonetheless still had stable rank one. He went on, in \cite{vil2}, to construct a class of simple unital AH algebras---inductive limits of homogeneous $\mathrm{C}^{*}$-algebras (those whose irreducible representations all have the same dimension)---of arbitrary stable rank, thereby affirming the subtleties present in the problem of stable rank. 

Given compact metric spaces $X$ and $Z$, together with continuous functions $\lambda_{1},\ldots,\lambda_{k}\colon Z\to X$, there is a naturally induced $^*$-homomorphism from $C(X,M_{n}(\mathbb{C}))$ to $C(Z,M_{nk}(\mathbb{C}))$ given by
$$
f\mapsto\diag(f\circ\lambda_{1},\ldots,f\circ\lambda_{k}).
$$
These induced maps between full matrix algebras are referred to as \textit{diagonal}. They have been used to construct a rich class of examples in the field, including those of Goodearl in \cite{Good} and Villadsen in \cite{vil1}.

Just over a decade ago, another stable rank one result was obtained by Elliott, Ho, and Toms in \cite{EHT}. Their paper, which stemmed from Ho's work in \cite{Ho}, showed that the condition of bounded dimension in \cite{4Rom} could be replaced with the assumption that all of the bonding maps in the inductive limit are diagonal.

In this present paper, we extend the AH stable rank one result of Elliott, Ho, and Toms in \cite{EHT} to a suitable class of approximately subhomogeneous (ASH) algebras---inductive limits of subhomogeneous $\mathrm{C}^{*}$-algebras (those whose irreducible representations all have dimension at most some fixed integer). 

The building-block algebras in the AH setting are full matrix algebras, whose primitive quotients are intrinsically matrix unit compatible. This internal compatibility is crucial to obtaining the stable rank one result in \cite{EHT}. To achieve this for the subhomogeneous building blocks in the ASH setting, it is necessary, therefore, to consider only subhomogeneous algebras whose primitive quotients fit together in a compatible (i.e. matrix unit compatible) way. We restrict our attention to a subclass of recursive subhomogeneous algebras.

Recursive subhomogeneous algebras are a particularly tractable class of unital subhomogeneous algebras introduced by Phillips in \cite{phil}, which are iterative pullbacks of full matrix algebras. In order to ensure the aforementioned compatibility, it is necessary that all the pullback maps be diagonal in a suitable sense and we call such algebras diagonal subhomogeneous (DSH) algebras. We are then able to define the notion of a diagonal map between two DSH algebras, which sends each point in the spectrum of the range algebra to an ordered list of eigenvalues in the domain algebra. It turns out that this setup is enough to extend the results in \cite{EHT}; more specifically, every simple inductive limit of DSH algebra with diagonal bonding maps has stable rank one (see \autoref{main}). 

DSH algebras arise naturally in the study of dynamical crossed products. The orbit-breaking subalgebras of crossed products introduced by Q. Lin in \cite{Lin} (see also \cite{LP} and \cite{LP2}) following the work of Putnam in \cite{Put} are examples of DSH algebras. Using our stable rank one theorem for inductive limits and results of Archey and Phillips developed in \cite{arcphil}, we are able to prove a conjecture of Archey, Niu, and Phillips stated in the same paper (Conjecture 7.2); namely, that for an infinite compact metric space $T$ and a minimal homeomorphism $h\colon T\to T$, the dynamical crossed product $\mathrm{C^*}(\mathbb{Z},T,h)$ has stable rank one (see \autoref{main2}). Using a result of Thiel in \cite{thi}, we are also able to show that, for such crossed products, classifiability is determined solely by strict comparison, thereby affirming the Toms--Winter Conjecture for simple dynamical crossed products (see \autoref{main3}).  

This paper is organized as follows. \Cref{ch: DSH} is dedicated to structure and basic properties of DSH algebras. In \cref{sub: intro}, we formally introduce the class of DSH algebras, the notion of a diagonal map between two such algebras, and we prove some basic lemmas concerning this class, which are used throughout the remainder of the paper. The aim of \cref{quotients} is to show that quotients of DSH algebras remain DSH, and that diagonal maps between two DSH algebras remain diagonal when passing to quotients; this allows one to assume that the bonding maps in \autoref{main} are injective. Finally, in \cref{sub: hom}, we show that every homogeneous DSH algebra is a full matrix algebra. 

\Cref{ch: stable rank} contains the main results of the paper. The proof of \autoref{main} is quite technical and relies on several lemmas, which are established in \cref{sc: prelim lemmas} and \cref{sc: main lemmas}. In \cref{sc: outline of proof}, we outline the significance of these lemmas and illustrate how they come together to prove \autoref{main} in \cref{sc: main proof}. Lastly, in \cref{sc: cp}, we discuss the significance of \autoref{main} in the setting of minimal dynamical crossed products, and we establish \autoref{main2} and \autoref{main3}. 

Throughout the paper, we use $\mathbb{N}$ to denote the set of strictly positive integers and the symbol $\subset$ to denote non-strict set inclusion. Given a $\mathrm{C}^*$-algebra $A$, we let $\hat{A}$ denote the set of equivalence classes of non-zero irreducible representations of $A$ equipped with the hull-kernel topology. If $A$ is unital, we use $1_{A}$ to denote the unit of $A$. For $n\in\NN$, we use the shorthand $M_{n}$ to refer to the matrix algebra $M_{n}(\mathbb{C})$. When speaking about a matrix $D\in M_{n}$, we denote the $(i,j)$-entry of $D$ by $D_{i,j}$, and we let $1_{n}$ denote the identity matrix in $M_{n}$.

\section{Diagonal Subhomogeneous (DSH) Algebras}
\label{ch: DSH}

In this section we introduce the class of diagonal subhomogeneous (DSH) algebras that we deal with in this paper and examine their basic properties and structure. In \cref{sub: intro}, we define what a DSH algebra is and the notion of a diagonal map between two such algebras. We discuss some basic properties and notions concerning these algebras that are used throughout the remainder of the section and beyond. The chief purpose of \cref{quotients} is to prove that given any inductive limit of DSH algebras with diagonal bonding maps, one may always assume the maps in the sequence are injective. 

In Corollary 1.8 of \cite{phil}, Phillips shows that every unital homogeneous $\mathrm{C}^{*}$-algebra (regardless of its Dixmier--Douady class) has a recursive subhomogeneous decomposition. This follows by using the pullback maps to appropriately adjoin various pieces of the spectrum over which the algebra is locally trivial. In the DSH setting, where the pullback maps preserve the matrix units of the primitive quotients in a very strong sense, such a bonding is possible only if the homogeneous algebra is in fact a full matrix algebra, as we show in \cref{sub: hom}.

\subsection{Introductory Definitions and Basic Properties}
\label{sub: intro}
Let us start off by recalling the definition of a recursive subhomogeneous algebra.

\begin{definition}[\cite{phil}, Definition 1.1]
	A \textit{recursive subhomogeneous algebra} is a $\mathrm{C}^*$-algebra given by the following recursive definition.
	\begin{enumerate}
		\item If $X$ is a compact metric space and $n\geq 1$, then $C(X,M_{n})$ is a recursive subhomogeneous algebra. 
		\item If $A$ is a recursive subhomogeneous algebra, $X$ is a compact metric space, $Y\subset X$ is closed, $\varphi\colon A\to C(Y, M_{n})$ is any unital $^*$-homomorphism, and $\rho\colon C(X,M_{n})\to C(Y, M_{n})$ is the restriction $^*$-homomorphism, then the pullback
		$$
		A\oplus_{C(Y, M_{n})}C(X,M_{n}):=\{(a,f)\in A\oplus C(X,M_{n}):\varphi(a)=\rho(f)\}
		$$ 
		is a recursive subhomogeneous algebra. 
	\end{enumerate}
\end{definition}

Therefore, if $A$ is a recursive subhomogeneous algebra, there are compact metric spaces $X_{1},\ldots, X_{ l}$ (the \textit{base spaces} of $A$), closed subspaces $Y_{1}:=\varnothing$, $Y_{2}\subset X_{2},\ldots, Y_{ l}\subset X_{ l}$, positive integers $n_{1},\ldots, n_{ l}$, $\mathrm{C}^*$-algebras  $A^{(i)}\subset\bigoplus_{j=1}^{i}C(X_{j},M_{n_{j}})$ for $1\leq i\leq  l$, and unital $^*$-homomorphisms $\varphi_{i}\colon A^{(i)}\to C(Y_{i+1},M_{n_{i+1}})$ for $1\leq i\leq  l-1$ such that:
\begin{enumerate}
	\item $A^{(1)}=C(X_{1},M_{n_{1}})$;
	\item for all $1\leq i\leq  l-1$
	$$
	A^{(i+1)}=\{(a,f)\in A^{(i)}\oplus C(X_{i+1},M_{n_{i+1}}):\varphi_{i}(a)=f|_{Y_{i+1}}\};
	$$
	\item $A=A^{( l)}$.
\end{enumerate}	
Simply put, 
$$
A=\bigg[\cdots\Big[\left[C_{1}\oplus_{C_{2}'}C_{2}\right]\oplus_{C_{3}'}C_{3}\Big]\cdots\bigg]\oplus_{C_{ l}'}C_{ l},
$$
where $C_{i}:=C(X_{i},M_{n_{i}})$, $C_{i}':=C(Y_{i},M_{n_{i}})$, and the maps $\varphi_{1},\ldots,\varphi_{ l-1}$ are used in the pullback. In this case, we say the length of the composition sequence is $ l$. As shown in \cite{phil}, the decomposition of a recursive subhomogeneous is highly non-unique. We make the same tacit assumption adopted in that paper: unless otherwise specified, every recursive subhomogeneous algebra comes equipped with a decomposition of the form given above. In particular, we refer to the number $ l$ above as the \textit{length of $A$}.

Since for all $1\leq i\leq l$, we have $A^{(i)}\subset\bigoplus_{j=1}^{i}C(X_{j},M_{n_{j}})$, we can view each element $f\in A^{(i)}$ as $(f_{1},\ldots,f_{i})$, where $f_{j}\in C(X_{j},M_{n_{j}})$ for all $1\leq j\leq i$. For $1\leq i\leq  l$ and $x\in X_{i}$, we have the usual \textit{evaluation map} $\ev_{x}\colon A\to M_{n_{i}}$ given by $\ev_{x}(f):=f_{i}(x)$ for all $f\in A$. We let $\mathfrak{s}(A):=\min\{n_{1},\ldots,n_{ l}\}$ and $\mathfrak{S}(A):=\max\{n_{1},\ldots,n_{ l}\}$.

The chief reasons for working with recursive subhomogeneous algebras are that they are very convenient computationally and they also allow us to carry forward much of the structure intrinsic to a full matrix algebra. There is, however, no restriction on the pullback maps used to join together the full matrix algebras in the recursive decomposition. In particular, the pullback maps need not piece together the matrix units of the various primitive quotients in a compatible way. Therefore, in order to harness the internal matrix unit compatible structure of a full matrix algebra, one must ensure that the pullback maps used in the recursive decomposition preserve the matrix units of each full matrix algebra. An effective way to achieve this is to require the pullback maps to be diagonal in an appropriate sense, which we now make clear. 

\begin{definition}[DSH algebras]
	\label{DSH def}
	A $\mathrm{C}^*$-algebra $A$ is a \textit{diagonal subhomogeneous (DSH) algebra} (of length $ l$) provided that it is a recursive subhomogeneous algebra (of length $ l$) (with a decomposition as described above), and for all $1\leq i\leq  l-1$ and $y\in Y_{i+1}$, there is a list of points $x_{1}\in X_{i_{1}}\setminus Y_{i_{1}},\ldots, x_{t}\in X_{i_{t}}\setminus Y_{i_{t}}$ such that for all $f\in A^{(i)}$, 
	$$
	\varphi_{i}(f)(y)=\diag(f_{i_{1}}(x_{1}),\ldots,f_{i_{t}}(x_{t})).
	$$
	We say $y$ \textit{decomposes into} $x_{1},\ldots,x_{t}$, that each $x_{j}$ is a point \textit{in the decomposition of $y$}, and that $x_{j}$ \textit{begins at index $1+n_{i_{1}}+\cdots+n_{i_{j-1}}$ down the diagonal of $y$}. 	Given $1\leq j\leq i$ and $y'\in Y_{j}$, we say that $y'$ is \textit{in the decomposition of $y$} if there exists a $1\leq k\leq n_{i}$ with the property that for all $f\in A^{(i)}$ there are matrices $P\in M_{k-1}$ and $Q\in M_{n_{i}-n_{j}-(k-1)}$ such that $f_{i}(y)=\diag(P,f_{j}(y'),Q)$. 
\end{definition}

Whenever we work with a DSH algebra of length $ l$ we adopt, unless otherwise specified, the same notation for the decomposition used above. Thus, if $A$ is a DSH algebra of length $ l$, we can view $A$ as the set of all $f:=(f_{1},\ldots,f_{ l})\in\bigoplus_{i=1}^{ l}C(X_{i},M_{n_{i}})$ such that for all $1\leq i< l$ and $y\in Y_{i+1}$, 
$$
f_{i+1}(y)=\diag\left(f_{i_{1}}(x_{1}),\ldots,f_{i_{t}}(x_{t})\right).
$$
As is shown in \autoref{unique} below, the decomposition of $y$ is unique up to the re-indexing of identical points; that is, if $y$ decomposes into $x_{1},\ldots,x_{t}$ and $z_{1},\ldots,z_{s}$, then $s=t$ and, for $1\leq j\leq s$, $x_{j}=z_{j}$.  

\begin{definition}[Diagonal maps between DSH algebras]
	\label{diagonal map}
	Given two DSH algebras $A_{1}$ and $A_{2}$ of lengths $ l_{1}$ and $ l_{2}$ and with base spaces $X_{1}^{1},\ldots,X_{ l_{1}}^{1}$ and $X_{1}^{2},\ldots,X_{ l_{2}}^{2}$, respectively, we say that a $^*$-homomorphism $\psi\colon A_{1}\to A_{2}$ is \textit{diagonal} provided that for all $1\leq i\leq  l_{2}$ and $x\in X_{i}^{2}$, there are points $x_{1},\ldots,x_{t}$ with $x_{j}\in X_{i_{j}}^{1}$ such that for all $f\in A_{1}$, 
	$$
	\psi(f)_{i}(x)=\diag(f_{i_{1}}(x_{1}),\ldots,f_{i_{t}}(x_{t})).
	$$
	We say that $x$ \textit{decomposes into} $x_{1},\ldots,x_{t}$.
\end{definition}

Note that if $Y_{1}^{1}\subset X_{1}^{1},\ldots ,Y_{ l_{1}}^{1}\subset X_{ l_{1}}^{1}$ and $Y_{1}^{2}\subset X_{1}^{2},\ldots ,Y_{ l_{2}}^{2}\subset X_{ l_{2}}^{2}$ are the corresponding closed subsets of the base spaces in \autoref{diagonal map}, then, owing to the decomposition structure of $A_{1}$ and $A_{2}$, we get an equivalent definition by replacing $X_{i}^{2}$ and $X_{i_{j}}^{1}$ above with $X_{i}^{2}\setminus Y_{i}^{2}$ and $X_{i_{j}}^{1}\setminus Y_{i_{j}}^{1}$, respectively. Note, too, that by definition diagonal maps are automatically unital.

For the remainder of \cref{sub: intro}, let us assume that $A$ is a DSH algebra of length $ l$. The following lemma provides us with a description of the spectrum of $A$.

\begin{lemma}[\cite{phil}, Lemma 2.1]
	\label{DSH spectrum}
	The map $x\mapsto [\ev_{x}]$ defines a continuous bijection
	$$
	\spec\to \hat{A},
	$$
	(where, recall,  $Y_{1}:=\varnothing$) whose restriction to each $X_{i}\setminus Y_{i}$ is a homeomorphism onto its image. In particular every irreducible representation of $A$ is unitarily equivalent to $\ev_{x}$ for some $x\in \spec$.
\end{lemma}

We often tacitly refer to a given irreducible representation $\ev_{x}$ simply as $x$, since we view such an element both as an irreducible representation and as a point in $X_{i}$. 

\begin{remark}
	A subset $D\subset X_{i}\setminus Y_{i}$ can be viewed as a subset both of $X_{i}$ and of $\hat{A}$. We denote by $\overline{D}^{X_{i}}$ the closure of $D$ with respect to the topology on $X_{i}$. With one or two exceptions, when speaking about open and closed subsets of $X_{i}$ in this paper, we mean with respect to the topology on $X_{i}$; such subsets could, in general, include points in $Y_{i}$, in which case they would not even be a subset of the spectrum. In any case, for subsets of $X_{i}\setminus Y_{i}$, the topology is always made explicit.
\end{remark}

\begin{lemma}
	\label{unique}
	Suppose $2\leq i\leq  l$ and $y\in Y_{i}$. If $y$ decomposes into $x_{1},\ldots,x_{t}$ and $z_{1},\ldots,z_{s}$, then $s=t$ and for $1\leq j\leq s$, $x_{j}=z_{j}$.   
\end{lemma}

\begin{proof}
	By \autoref{DSH spectrum}, $A$ is liminary and $x_{1}=z_{1}$ if and only if $\ev_{x_{1}}=\ev_{z_{1}}$. Hence, if $x_{1}\not=z_{1}$, Proposition 4.2.5 of \cite{dix} furnishes a function $f\in A$ such that $\ev_{x_{1}}(f)$ and $\ev_{z_{1}}(f)$ are the $0$ and identity matrix of appropriate sizes, respectively. This contradicts the assumption that
	\begin{equation}
	\label{2 decomps}
	\diag(\ev_{x_{1}}(f),\ldots,\ev_{x_{t}}(f))=f_{i}(y)=\diag(\ev_{z_{1}}(f),\ldots,\ev_{z_{s}}(f)).
	\end{equation}
	Therefore, $x_{1}=z_{1}$. Continuing inductively, we see that if $s<t$, $s>t$, or $x_{j}\not=z_{j}$ for some $j$, then \autoref{2 decomps} is violated. Hence, \autoref{unique} follows.
\end{proof}

By \autoref{DSH spectrum} and \autoref{DSH def}, given $y\in \bigsqcup_{i=1}^{ l}X_{i}$, either $\ev_{y}$ is an irreducible representation of $A$ or, if $y$ is in some $Y_{i}$, $\ev_{y}$ splits up into irreducible representations of $A$. The following definition categorizes the elements in the base spaces depending on the indices at which these irreducible representations occur.

\begin{definition}
	\label{Bik def}
	Given $1\leq i\leq  l$ and $1\leq k\leq n_{i}$, we define $B_{i,k}$ to be the set of points in $X_{i}$ that have an irreducible representation beginning at index $k$ down their diagonal. For $k\leq 0$, we set $B_{i,k}:=\varnothing$.
\end{definition}


The following rudimentary observations about the $B_{i,k}$'s defined above will be very helpful in the proofs of the lemmas in \cref{sc: prelim lemmas} and \cref{sc: main lemmas}.

\begin{lemma}
	\label{Bik}
	\mbox{}
	\begin{enumerate}
		\item $B_{i,1}=X_{i}$ for all $1\leq i\leq  l$.
		\item If $1\leq i\leq  l$ and $k>1$, then $B_{i,k}\subset Y_{i}$.	In particular, $B_{1,k}=\varnothing$ for $k>1$.
		\item If $2\leq i\leq  l$ and $y\in Y_{i}$ decomposes into $x_{1}\in X_{i_{1}}\setminus Y_{i_{1}},\ldots,x_{t}\in X_{i_{t}}\setminus Y_{i_{t}}$, then $y\in B_{i,k}$ if and only if $k=1+n_{i_{1}}+\cdots+n_{i_{j-1}}$ for some $1\leq j\leq t$. In particular, $B_{i,k}=\varnothing$ for all $n_{i}-(\mathfrak{s}(A)-1)< k\leq n_{i}$, where, recall, $\mathfrak{s}(A):=\min\{n_{j}:1\leq j\leq  l\}$. 
	\end{enumerate}
\end{lemma}

\begin{proof}
	Fix $1\leq i\leq l$ and suppose $y\in X_{i}$. If $y$ belongs to $Y_{i}$ and decomposes into $x_{1},\ldots,x_{t}$, then by \autoref{DSH def}, $x_{1}$ begins at index $1$ down the diagonal of $y$, and so $y\in B_{i,1}$. If $y\notin Y_{i}$, then by \autoref{DSH spectrum}, $y$ is irreducible and trivially begins at index $1$ down its diagonal, and, moreover, cannot have any irreducible representation beginning at any index $k\geq 2$. This proves \textit{(1)} and \textit{(2)}. To prove \textit{(3)}, suppose $y$ belongs to $Y_{i}$ and decomposes into $x_{1}\in X_{i_{1}}\setminus Y_{i_{1}},\ldots,x_{t}\in X_{i_{t}}\setminus Y_{i_{t}}$. By \autoref{DSH def}, $y\in B_{i,k}$ if $k=1+n_{i_{1}}+\cdots+n_{i_{j-1}}$ for some $1\leq j\leq t$. Since \autoref{unique} shows that the decomposition of $y$ is unique, it follows that $y$ cannot belong to any other $B_{i,k}$, which establishes \textit{(3)}.
\end{proof}

The following lemma allows us to approximate any point in some $Y_{i}$ by irreducible representations in $X_{i}\setminus Y_{i}$ (with respect to the topology on $X_{i}$). 

\begin{lemma}
	\label{empty interior}
	For each $2\leq i\leq  l$, we may assume $\inte(Y_{i})=\varnothing$.
\end{lemma}

\begin{proof}
	Fix $1\leq i\leq l-1$. Let $Y_{i+1}':=Y_{i+1}\setminus\inte(Y_{i+1})$ and $X_{i+1}':=X_{i+1}\setminus\inte(Y_{i+1})$. We have the following commutative diagram of restriction $^*$-homomorphisms:
	$$
	\begin{tikzcd}
	C(X_{i+1},M_{n_{i+1}}) \arrow{r}{\rho} \arrow[swap]{d}{\lambda} & C(Y_{i+1},M_{n_{i+1}}) \arrow{d}{\tau} \\
	C(X_{i+1}',M_{n_{i+1}}) \arrow{r}{\rho'} & C(Y_{i+1}',M_{n_{i+1}})
	\end{tikzcd}
	$$ 
	Let $B^{(i+1)}:= A^{(i)}\oplus_{C(Y_{i+1}',M_{n_{i+1}})}C(X_{i+1}',M_{n_{i+1}})$, where the connecting $^*$-homomorphism is $\varphi_{i}':=\tau\circ\varphi_{i}\colon A^{(i)}\to C(Y_{i+1}',M_{n_{i+1}})$. Let us show that $A^{(i+1)}$ is isomorphic to $B^{(i+1)}$. Given $a\in A^{(i)}$ and $f\in C(X_{i+1},M_{n_{i+1}})$ with $(a,f)\in A^{(i+1)}$, define $\Gamma\colon A^{(i+1)}\to B^{(i+1)}$ by $\Gamma((a,f)):=(a,\lambda(f))$. Note that $\varphi_{i}'(a)=\tau(\varphi_{i}(a))=\tau(\rho(f))=\rho'(\lambda(f))$, so that $\Gamma$ is well defined. It is easy to see that $\Gamma$ is a $^*$-homomorphism. To see that $\Gamma$ is injective, suppose $(a,f)\in A^{(i+1)}$ with $(a,\lambda(f))=\Gamma((a,f))=(0,0)$. Then $a=0$ and so $f|_{Y_{i+1}}=\varphi_{i}(a)=0$, which, together with the fact that $\lambda(f)=0$, yields that $f=0$. For surjectivity suppose $a\in A^{(i)}$ and $g\in C(X_{i+1}',M_{n_{i+1}})$ with $(a,g)\in B^{(i+1)}$. Then $\varphi_{i}(a)|_{Y_{i+1}'}=g|_{Y_{i+1}'}$, so that the function $h\colon X_{i+1}\to M_{n_{i+1}}$ defined to be $\varphi_{i}(a)(x)$ for $x\in Y_{i+1}$ and $g(x)$ for $x\in X_{i+1}'$ is well defined and continuous. Moreover, $\varphi_{i}(a)=h|_{Y_{i+1}}$, which implies $(a,h)\in A^{(i+1)}$ and $\Gamma((a,h))=(a,\lambda(h))=(a,g)$, proving surjectivity. 
\end{proof}

The lemma following guarantees that a function in $A$ will be invertible provided it is an invertible matrix in every primitive quotient of $A$.

\begin{lemma}
	\label{non-invertible at point}
	Suppose $f\in A$ and that for all $1\leq i\leq l$ and $x\in X_{i}\setminus Y_{i}$, the matrix $f_{i}(x)$ is invertible in $M_{n_{i}}$. Then $f$ is invertible in $A$.
\end{lemma}
\begin{proof}
	Owing to the diagonal decomposition at points in $Y_{i}$, we may assume that $f_{i}(x)$ is an invertible matrix for all $1\leq i\leq l$ and $x\in X_{i}$. Define $g\in\bigoplus_{i=1}^{ l }C(X_{i},M_{n_{i}})$ to be $(g_{1},\ldots,g_{ l })$, where for $1\leq i\leq  l $ and $z\in X_{i}$, $g_{i}(z):=f_{i}(z)^{-1}$. 
	Since $g$ is the inverse of $f$ in $\bigoplus_{i=1}^{ l }C(X_{i},M_{n_{i}})$, to prove the lemma, we need only to verify that $g\in A$. Suppose $1\leq i\leq  l -1$ and that $y\in Y_{i+1}$ decomposes into $x_{1}\in X_{i_{1}}\setminus Y_{i_{1}},\ldots,x_{t}\in X_{i_{t}}\setminus Y_{i_{t}}$. Then,
	$$
	\begin{aligned}
		g_{i+1}(y)&=f_{i+1}(y)^{-1}\\
		&=\diag((f_{i_{1}}(x_{1}),\ldots,f_{i_{t}}(x_{t}))^{-1}\\
		&=\diag((f_{i_{1}}(x_{1})^{-1},\ldots,f_{i_{t}}(x_{t})^{-1})\\
		&=\diag(g_{i_{1}}(x_{1}),\ldots,g_{i_{t}}(x_{t})),
	\end{aligned}
	$$ 
	so that $g\in A$. 
\end{proof}	

This next lemma illustrates a particular circumstance in which a set which is open in one of the base spaces of $A$ is open when viewed as a subset of the spectrum.

\begin{lemma}
	\label{open in spectrum}
	Suppose $1\leq i\leq \ell$. If $U\subset X_{i}\setminus Y_{i}$ is open with respect to the topology on $X_{i}$ and has the property that no point in $U$ appears in the decomposition of any point in $Y_{j}$ for any $j>i$, then $U$ is open with respect to the hull-kernel topology on $\hat{A}$. 
\end{lemma}

\begin{proof}
	Let $x\in U$ be arbitrary. Put $g_{j}\equiv0$ for $j<i$ and define $g_{i}\in C(X_{i},M_{n_{i}})$ to be any function such that $g_{i}(x)\not=0$ and $g_{i}|_{X_{i}\setminus U}\equiv 0$. Since $Y_{i}\subset X_{i}\setminus U$, $g_{i}$ vanishes on $Y_{i}$, so that $\varphi_{i-1}((g_{1},\ldots,g_{i-1}))=0=g_{i}|_{Y_{i}}$; thus,  $(g_{1},\ldots,g_{i})\in A^{(i)}$. For $j>i$, set $g_{j}\equiv 0$. Since no point in $U$ is in the decomposition of any point in $Y_{j}$ for any $j>i$, it follows inductively that $g:=(g_{1},\ldots,g_{\ell})\in A$. This proves that $U$ is open in $\hat{A}$. 
\end{proof}	

The final lemma in this subsection shows that if a point $x\in X_{i}$ is not in the decomposition of any point in some $Y_{j}$, then there must be an open neighbourhood of $x$ in $X_{i}$ consisting only of points which also do not appear in the decomposition of any point in $Y_{j}$. 

\begin{lemma}
	\label{not decomposition}
	Suppose $1\leq i < j \leq\ell$ and let $F\subset X_{i}$ denote the set of points that are in the decomposition of a point in $Y_{j}$. Then, $F$ is closed in $X_{i}$. 
\end{lemma}

\begin{proof}
	Suppose $(z_{n})_{n}$ is a sequence of points in $F$ converging to $x\in X_{i}$. For each $n\in\NN$, there is a $y_{n}\in Y_{j}$ with the property that $z_{n}$ is in the decomposition of $y_{n}$. Since $Y_{j}$ is compact, we may pass to a subsequence to conclude that there is a $y\in Y_{j}$ such that $y_{n}\to y$. Passing to a further subsequence, we may assume that there is a $1\leq k\leq n_{j}$ such that for all $n\in \NN$, the representation $\ev_{z_{n}}$ begins at index $k$ down the diagonal of $\ev_{y_{n}}$. Suppose $y$ decomposes into $x_{1}\in X_{i_{1}}\setminus Y_{i_{1}},\ldots,x_{t}\in X_{i_{t}}\setminus Y_{i_{t}}$. Let us show that $x\in F$ by proving it is in the decomposition of $y$. Let $f\in A$ be arbitrary. For each $n\in\NN$, there are matrices $P_{n}\in M_{k-1}$ and $Q_{n}\in M_{n_{j}-n_{i}-(k-1)}$ such that $f_{j}(y_{n})=\diag(P_{n},f_{i}(z_{n}),Q_{n})$. Since $\lim_{n\to\infty}f_{j}(y_{n})=f_{j}(y)$, there are matrices $P\in M_{k-1}$ and $Q\in M_{n_{j}-n_{i}-(k-1)}$ such that 
	\begin{equation}
	\label{decomp equation}
	f_{j}(y)=\diag(P,f_{i}(x),Q).
	\end{equation}
	If $x\in Y_{i}$, it follows by definition that $x$ is in the decomposition of $y$. If $x\in X_{i}\setminus Y_{i}$ and $x$ is not in the decomposition of $y$, then we may use Proposition 4.2.5 of \cite{dix} to find a function $g\in A$ that is non-zero at $x$, but vanishes at all points in the decomposition of $y$, which implies that $g_{j}(y)=0$ and contradicts \autoref{decomp equation}. Thus, $x$ is in the decomposition of $y$ and, hence, $x\in F$.
\end{proof}

\subsection{Quotients of DSH Algebras}
\label{quotients}

In Proposition 3.1 of \cite{phil}, Phillips shows that the class of (separable) recursive subhomogeneous algebras is closed under the taking of quotients. The recursive decomposition of the quotient is not explicitly constructed from that of the original algebra, but rather is furnished using a characterization of (separable) recursive subhomogeneous algebras (see \cite{phil}, Theorem 2.16).

We show in this subsection that associated to any quotient $B$ of a DSH algebra $A$, there is a DSH algebra (see \autoref{isomorphic quotient}) whose decomposition is canonically obtained from the decomposition of $A$, and which is isomorphic to $B$ (see \autoref{quotient DSH}). We are then able to prove (see \autoref{injective}), that the diagonality of maps between two DSH algebras is preserved when passing to quotients, thus allowing us to assume that the bonding maps in \autoref{main} are injective.

Let $A$ be a DSH algebra of length $ l$. Suppose we have a non-zero $\mathrm{C}^*$-algebra $B$ and a surjective $^*$-homomorphism $\rho\colon A\to B$. This yields an injective single-valued map $\hat{\rho}\colon\hat{B}\to\hat{A}$ given by $\hat{\rho}([\pi]):=[\pi\circ \rho]$. For $1\leq i\leq  l$, define $X_{i}':=\overline{X_{i}\cap \hat{\rho}(\hat{B})}^{X_{i}}$ and $Y_{i}':=X_{i}'\cap Y_{i}$. Recall that these definitions make sense by \autoref{DSH spectrum}.

\begin{lemma}
	\label{image closed}
	$\hat{\rho}(\hat{B})$ is closed in $\hat{A}$. 
\end{lemma}

\begin{proof}
	Suppose $[\pi]\in\overline{\hat{\rho}(\hat{B})}$. Then 
	$$
	\ker\pi\supset\bigcap_{[\sigma]\in\hat{\rho}(\hat{B})}\ker\sigma=\bigcap_{[\tau]\in \hat{B}}
	\ker\hat{\rho}([\tau])=\bigcap_{[\tau]\in\hat{B}}\ker(\tau\circ\rho).
	$$
	Note that $a\in\bigcap_{[\tau]\in\hat{B}}\ker(\tau\circ\rho)$ if and only if $\rho(a)\in\bigcap_{[\tau]\in\hat{B}}\ker\tau$ if and only if $\rho(a)=0$. Hence, $\ker\pi\supset\ker\rho$. Thus, the irreducible representation $\tau$ of $B$ given by $\tau(b):=\pi(a)$, where $a$ is any lift of $b$ under $\rho$ is well defined. Therefore, $[\pi]=[\tau\circ\rho]=\hat{\rho}([\tau])\in\hat{\rho}(\hat{B})$, so that $\overline{\hat{\rho}(\hat{B})}\subset\hat{\rho}(\hat{B})$. 
\end{proof}

\begin{lemma}
	\label{quotient decomp}
	Suppose $1\leq i\leq  l$ and $y\in Y_{i}'$. If $1\leq j<i$ and $x\in X_{j}\setminus Y_{j}$ is in the decomposition of $y$, then $x\in X_{j}\cap\hat{\rho}(\hat{B})\subset X_{j}'$.  
\end{lemma}

\begin{proof}
	Since $y\in Y_{i}'$, we have $y\in X_{i}'=\overline{X_{i}\cap\hat{\rho}(\hat{B})}^{X_{i}}$. Choose a sequence $(z_{n})_{n}$ in $X_{i}\cap\hat{\rho}(\hat{B})$ such that $z_{n}\to y$ with respect to the topology on $X_{i}$. Let us show that $(\ev_{z_{n}})_{n}\to \ev_{x}$, with respect to the hull-kernel topology on $\hat{A}$. Suppose $U$ is an open set in $\hat{A}$ containing $\ev_{x}$. Then there is a function $f\in A$ that is non-zero at $x$, but vanishes at each point in $\hat{A}\setminus U$. Since $x$ is in the decomposition of $y$, this implies that $f_{i}(y)\not=0$. Since $z_{n}\to y$ in $X_{i}$ and since $f_{i}$ is continuous, there is an $n_{0}$ such that for all $n\geq n_{0}$, $f_{i}(z_{n})\not=0$. In particular, this means that for all $n\geq n_{0}$, $\ev_{z_{n}}\in U$. Therefore, $\ev_{z_{n}}\to \ev_{x}$ in $\hat{A}$. Now, by \autoref{image closed}, $\hat{\rho}(\hat{B})$ is closed and, hence, what we've shown implies that $\ev_{x}\in\hat{\rho}(\hat{B})$. Therefore, $x\in X_{j}\cap\hat{\rho}(\hat{B})\subset X_{j}'$. 
\end{proof}

In the following lemma, we construct a DSH algebra from $A$ over the base spaces $X_{i}'$, where the pullback maps are just restrictions of the pullback maps in the definition of $A$ (the $\varphi_{i}$'s). We show afterwards (see \autoref{quotient DSH}) that this new DSH algebra is isomorphic to the quotient $B$.

\begin{proposition}
	\label{isomorphic quotient}
	There is a DSH algebra $D$ of length $ l$ with the following properties:
	\begin{enumerate}
		\item $D^{(1)}=C(X_{1}',M_{n_{1}})$; 
		\item for all $1\leq i\leq  l$, $D^{(i)}\subset\bigoplus_{j=1}^{i}C(X_{j}',M_{n_{j}})$;
		\item for all $1\leq i < l$, 
		the pullback map $\tau_{i}\colon D^{(i)}\to C(Y_{i+1}',M_{n_{i+1}})$ is given by $\tau_{i}(f)(y):=\diag(f_{i_{1}}(x_{1}),\ldots,f_{i_{t}}(x_{t}))$, where $x_{1},\ldots,x_{t}$ are the points in the decomposition of $y$ coming from the definition of $A$;
		\item for $1\leq i<  l$, $D^{(i+1)}=D^{(i)}\oplus_{C(Y_{i+1}',M_{n_{i+1}})}C(X_{i+1}',M_{n_{i+1}})$ with pullback map $\tau_{i}$;
		\item for all $1\leq i\leq  l$, if $(f_{1},\ldots,f_{i})\in D^{(i)}$, there is a $(g_{1},\ldots,g_{i})\in A^{(i)}$ such that for all $1\leq j\leq i$, $g_{j}|_{X_{j}'}=f_{j}$. 
	\end{enumerate}
\end{proposition}

\begin{proof}
	Let us proceed by induction on $i$. Define $D^{(1)}:=C(X_{1}',M_{n_{1}})$ so that \textit{(1)} holds. Since $X_{1}'$ is closed in $X_{1}$, we may extend a function in $D^{(1)}$ to a function in $A^{(1)}=C(X_{1},M_{n_{1}})$, so that \textit{(5)} holds when $i=1$. Now, fix $1\leq i\leq  l-1$ and assume that we have defined $D^{(1)},\ldots,D^{(i)}$ and $\tau_{1},\ldots,\tau_{i-1}$ satisfying conditions \textit{(1)} to \textit{(5)}. Let us show how to define $\tau_{i}$ and $D^{(i+1)}$. Given $(f_{1},\ldots, f_{i})\in D^{(i)}$, use \textit{(5)} to get $(g_{1},\ldots,g_{i})\in A^{(i)}$ such that $g_{j}|_{X_{j}'}=f_{j}$ for $1\leq j\leq i$. Define $\tau_{i}\colon D^{(i)}\to C(Y_{i+1}',M_{n_{i+1}})$ by $\tau_{i}((f_{1},\ldots,f_{i})):=\varphi_{i}((g_{1},\ldots,g_{i}))|_{Y_{i+1}'}$. 
	
	To see that $\tau_{i}$ is a well-defined $^*$-homomorphism satisfying \textit{(3)}, suppose $(h_{1},\ldots,h_{i})\in A^{(i)}$ also restricts coordinate-wise to $(f_{1},\ldots,f_{i})$. If $y\in Y_{i+1}'$ decomposes into $x_{1}\in X_{i_{1}}\setminus Y_{i_{1}},\ldots,x_{t}\in X_{i_{t}}\setminus Y_{i_{t}}$, then by \autoref{quotient decomp}, we have $x_{1}\in X_{i_{1}}',\ldots,x_{t}\in X_{i_{t}}'$. Hence,
	$$
	\begin{aligned}
	\varphi_{i}((g_{1},\ldots,g_{i}))(y)&=\diag(g_{i_{1}}(x_{1}),\ldots,g_{i_{t}}(x_{t}))\\
	&=\diag(f_{i_{1}}(x_{1}),\ldots,f_{i_{t}}(x_{t}))\\
	&=\diag(h_{i_{1}}(x_{1}),\ldots,h_{i_{t}}(x_{t}))\\
	&=\varphi_{i}((h_{1},\ldots,h_{i}))(y).
	\end{aligned}
	$$
	Therefore, $\tau_{i}$ satisfies \textit{(3)} and is independent of the choice of extension used. Moreover, $\tau_{i}((f_{1},\ldots,f_{i}))$ is continuous, being the restriction of a continuous function. Thus, $\tau_{i}$ is well defined and it is clearly a $^*$-homomorphism since $\varphi_{i}$ is.  
	
	Next, define $D^{(i+1)}:=D^{(i)}\oplus_{C(Y_{i+1}',M_{n_{i+1}})}C(X_{i+1}',M_{n_{i+1}})$, using $\tau_{i}$ as the pullback map. This ensures that \textit{(2)} and \textit{(4)} hold, and so we just need to verify \textit{(5)}. Suppose $(d,f)\in D^{(i+1)}$, where $d\in D^{(i)}$ and $f\in C(X'_{i+1},M_{n_{i+1}})$. By the inductive hypothesis, we may apply \textit{(5)} to $d$ to obtain a $b\in A^{(i)}$ such that $b_{j}|_{X_{j}'}=d_{j}$ for all $1\leq j\leq i$. Let $g:=\varphi_{i}(b)\in C(Y_{i+1},M_{n_{i+1}})$. If $y\in X_{i+1}'\cap Y_{i+1}=Y_{i+1}'$, then $g(y)=\varphi_{i}(b)(y)=\tau_{i}(d)(y)=f(y)$. Thus, since $X_{i+1}'$ and $Y_{i+1}$ are both closed in $X_{i+1}$ and since $f$ and $g$ agree on their intersection, they have a common extension $h\in C(X_{i+1},M_{n_{i+1}})$. Since $\varphi_{i}(b)=g=h|_{Y_{i+1}}$, we have $(b,h)\in A^{(i+1)}$, and since $h|_{X_{j+1}'}=f$, it follows that \textit{(5)} holds. 
\end{proof}

\begin{proposition}
	\label{quotient DSH}
	Let $D=D^{( l)}$ be the DSH algebra constructed in \autoref{isomorphic quotient}. There is a $^*$-isomorphism $\Gamma\colon B\to D$ given coordinate-wise by $\Gamma(b)_{i}:=a_{i}|_{X_{i}'}$ for $1\leq i\leq  l$, where $a\in A$ is any lift of $b$ under $\rho$. In particular, the quotient $B$ is a DSH algebra.
\end{proposition}

\begin{proof}
	Let us first show that $\Gamma(b)$ is independent of the choice of lift. Fix $1\leq i\leq  l$ and suppose $g,h\in A$ satisfy $\rho(g)=\rho(h)$. We must show that $g_{i}|_{X_{i}'}=h_{i}|_{X_{i}'}$. Note that $\overline{X_{i}'\setminus Y_{i}'}^{X_{i}}=X_{i}'$. Indeed, $X_{i}'$ is closed with respect to the topology on $X_{i}$, and so the fact that $\overline{X_{i}'\setminus Y_{i}'}^{X_{i}}$ is a subset of $X_{i}'$ is clear; for the reverse inclusion, if $z\in X_{i}'$, there is a sequence $(z_{n})_{n}\subset \hat{\rho}(\hat{B})\cap X_{i}\subset X_{i}'\setminus Y_{i}\subset X_{i}'\setminus Y_{i}'$ that converges to $z$ in $X_{i}$. Hence, by continuity, it suffices to show that $g_{i}|_{X_{i}'\setminus Y_{i}'}=h_{i}|_{X_{i}'\setminus Y_{i}'}$. To this end, suppose $x\in X_{i}'\setminus Y_{i}'$. Then $x\in \overline{X_{i}\cap\hat{\rho}(\hat{B})}^{X_{i}}=\overline{(X_{i}\setminus Y_{i})\cap\hat{\rho}(\hat{B})}^{X_{i}}$ and $x\notin Y_{i}$. By \autoref{image closed} and \autoref{DSH spectrum}, $(X_{i}\setminus Y_{i})\cap \hat{\rho}(\hat{B})$ is closed in $X_{i}\setminus Y_{i}$ in the subspace topology coming from $X_{i}$. Thus,
	$$
	\begin{aligned}
	x\in \overline{(X_{i}\setminus Y_{i})\cap\hat{\rho}(\hat{B})}^{X_{i}}\cap (X_{i}\setminus Y_{i})
	=\overline{(X_{i}\setminus Y_{i})\cap\hat{\rho}(\hat{B})}^{X_{i}\setminus Y_{i}}
	=(X_{i}\setminus Y_{i})\cap\hat{\rho}(\hat{B})\subset\hat{\rho}(\hat{B}).
	\end{aligned}
	$$
	Therefore, there is a $[\pi]\in \hat{B}$ such that $[\pi\circ\rho]=\hat{\rho}([\pi])=[\ev_{x}]$. But this implies that $g-h\in\ker \ev_{x}$ since $g-h\in \ker\rho$. Hence, $g_{i}(x)=h_{i}(x)$, as desired. Moreover, $\Gamma(b)_{i}$ belongs to $C(X_{i}',M_{n_{i}})$, being the restriction of a continuous function. To see that $\Gamma(b)$ respects the decomposition structure of $D$, suppose  $y\in Y_{i}'$ decomposes into $x_{1}\in X_{i_{1}}'\setminus Y_{i_{1}}', \ldots, x_{t}\in X_{i_{t}}'\setminus Y_{i_{t}}'$. Then,
	$$
	\begin{aligned}
	\Gamma(b)_{i}(y)=a_{i}(y)
	=\diag(a_{i_{1}}(x_{1}),\ldots, a_{i_{t}}(x_{t}))
	=\diag(\Gamma(b)_{i_{1}}(x_{1}),\ldots,\Gamma(b)_{i_{t}}(x_{t})).
	\end{aligned}
	$$ 
	Therefore, $\Gamma$ is well defined and it is straightforward to check that it is a $^*$-homomorphism. We have left only to check that it is a bijection. 
	
	To see that $\Gamma$ is injective, suppose $b\in B$ and $a\in A$ is such that $\rho(a)=b$. Assume that $\Gamma(b)=0$. Let $\pi$ be an arbitrary irreducible representation of $B$. To show that $b=0$, it suffices to show that $\pi(b)=0$. Note that $[\pi\circ\rho]=\hat{\rho}([\pi])\in\hat{\rho}(\hat{B})$. Thus, for some $1\leq i\leq  l$, there is an $x\in (X_{i}\setminus Y_{i})\cap \hat{\rho}(\hat{B})\subset X_{i}'$ such that $[\pi\circ \rho]=[\ev_{x}]$.  Since $\ev_{x}(a)=a_{i}(x)=\Gamma(b)_{i}(x)=0$, it follows that $\pi(b)=\pi(\rho(a))=0$. Thus, $\Gamma$ is injective.
	
	To see that $\Gamma$ is surjective, suppose $d\in D$. By property \textit{(5)} in \autoref{isomorphic quotient}, there is a $g\in A$ such that $g_{i}|_{X_{i}'}=d_{i}$ for all $1\leq i\leq l$. Let $h=\rho(g)\in B$ and observe that for all $1\leq i\leq  l$, we have $\Gamma(h)_{i}=g_{i}|_{X_{i}'}=d_{i}$. Thus, $\Gamma(h)=d$, so $\Gamma$ is surjective. 
	
	We have shown that $\Gamma$ is a $^*$-isomorphism, from which it follows that $B$ is a DSH algebra. 
\end{proof}

\begin{proposition}
	\label{injective}
	Given an inductive limit 
	$$ 
	A_{1}\overset{\psi_{1}}{\longrightarrow} A_{2}\overset{\psi_{2}}{\longrightarrow} A_{3}\overset{\psi_{3}}{\longrightarrow}\cdots\longrightarrow A:=\varinjlim A_{i}
	$$
	of DSH algebras with diagonal maps, there exist DSH algebras $D_{1},D_{2},\ldots$ and injective diagonal maps $\psi_{i}'\colon D_{i}\to D_{i+1}$ such that 
	$$
	D_{1}\overset{\psi_{1}'}{\longrightarrow} D_{2}\overset{\psi_{2}'}{\longrightarrow} D_{3}\overset{\psi_{3}'}{\longrightarrow}\cdots\longrightarrow A.
	$$
\end{proposition}

\begin{proof}
	For $n\in\NN$, let $\mu_{n}\colon A_{n}\to A$ denote the map in the construction of the inductive limit and consider the surjective map $\kappa_{n}\colon A_{n}\to A_{n}/\ker\mu_{n}=:B_{n}$. The induced map $\nu_{n}\colon B_{n}\to B_{n+1}$ given by $\nu_{n}(\kappa_{n}(a)):=\kappa_{n+1}(\psi_{n}(a))$ for all $a\in A_{n}$ is well defined and injective. Furthermore, $\varinjlim(B_{n},\{\nu_{n}\}_{n})=A$. Let $X_{1}^{n},\ldots,X_{ l(n)}^{n}$ denote the base spaces of $A_{n}$ and let $Y_{1}^{n},\ldots,Y_{ l(n)}^{n}$ denote the corresponding closed subsets. Let $D_{n}$ denote the DSH algebra given by \autoref{isomorphic quotient} and isomorphic to $B_{n}$ (with base spaces $\overline{X_{i}^{n}\cap\hat{\kappa}_{n}(\hat{B_{n}})}^{X_{i}^{n}}=:Z_{i}^{n}$ and corresponding closed subsets $Z_{i}^{n}\cap Y_{i}^{n}=:W_{i}^{n}$ for $1\leq i\leq  l(n)$). By \autoref{quotient DSH}, the injective map $\nu_{n}$ drops down to an injective map $\psi_{n}'\colon D_{n}\to D_{n+1}$ given by $\psi_{n}'(d)_{i}:=\psi_{n}(a)_{i}|_{Z_{i}^{n+1}}$ for all $1\leq i\leq  l(n+1)$, where $a\in A_{n}$ is any coordinate-wise extension of $d$. Moreover, $\varinjlim(D_{n},\{\psi'_{n}\}_{n})=A$.
	
	We need to check that $\psi'_{n}$ is diagonal. Fix $1\leq i\leq  l(n+1)$ and suppose $x\in Z_{i}^{n+1}\setminus W_{i}^{n+1}\subset X^{n+1}_{i}\setminus Y_{i}^{n+1}$ decomposes into $x_{1}\in X_{i_{1}}^{n}\setminus Y_{i_{1}}^{n},\ldots,x_{t}\in X_{i_{t}}^{n}\setminus Y_{i_{t}}^{n}$ under the diagonal map $\psi_{n}$. We need to show that $x_{j}\in Z_{i_{j}}^{n}\setminus W_{i_{j}}^{n}$ for all $1\leq j\leq t$. Since $\ev_{x}\circ\psi_{n}'$ is a $^*$-representation of $D_{n}$, it is unitarily equivalent to a finite direct sum of irreducible representations $\ev_{z_{1}},\ldots,\ev_{z_{k}}\in\bigsqcup_{s=1}^{ l(n)}(Z_{s}^{n}\setminus W_{s}^{n})\subset \hat{A}_{n}$. Fix $1\leq j\leq t$. If $x_{j}\notin\{z_{1},\ldots,z_{k}\}$, then by Proposition 4.2.5 of \cite{dix}, there is a function $a\in A_{n}$ such that $\ev_{z_{s}}(a)=0$ for all $1\leq s\leq k$, but $\ev_{x_{j}}(a)\not=0$. Since $x_{j}$ is in the decomposition of $x$ under $\psi_{n}$, this implies that $\ev_{x}(\psi_{n}(a))$ is both zero and non-zero simultaneously. Therefore, it must be that $x_{j}\in \{z_{1},\ldots,z_{k}\}$ and, thus, that $x_{j}\in Z_{i_{j}}^{n}\setminus W_{i_{j}}^{n}$, as desired.
\end{proof}

\subsection{Homogeneous DSH Algebras}
\label{sub: hom}

Suppose $A$ is an $n$-homogeneous DSH algebra. We show in this subsection that there is a compact metric space $X$ such that $A$ is isomorphic to $C(X,M_{n})$.

\begin{proposition}
	\label{homcomb}
	Let $X_{1},X_{2}$ be compact metric spaces. Let $Y_{2}$ be a closed subset of $X_{2}$. Let $\varphi\colon C(X_{1},M_{n})\to C(Y_{2},M_{n})$ be a unital $^*$-homomorphism and suppose the associated pullback $C(X_{1},M_{n})\oplus_{C(Y_{2},M_{n})}C(X_{2},M_{n})$ is a DSH algebra. Then, there is a compact metric space $Z^{*}$ such that
	$C(X_{1},M_{n})\oplus_{C(Y_{2},M_{n})}C(X_{2},M_{n})$ is isomorphic to $C(Z^{*},M_{n})$.
\end{proposition}

\begin{proof}
	For a given $y\in Y_{2}$, we know by \autoref{unique} that the point it decomposes into is unique; alternatively, note that if there were two distinct points in the decomposition of $y$ under $\varphi$, then these two points could not be separated by any function in $C(X_{1},M_{n})$. Denote this unique point by $\tau(y)$.
	
	\begin{claim}
		\label{closed and continuous}
		$\tau\colon Y_{2}\to X_{1}$ is a closed and continuous map. 
	\end{claim}
	
	\renewcommand\qedsymbol{/\hspace{-2pt}/\hspace{-2pt}/\hspace{-2pt}/}
	\begin{proof}[Proof of \autoref{closed and continuous}]
		Suppose first that $(y_{n})_{n}$ is a sequence in $Y_{2}$ converging to a point $y$ and let $f\in C(X_{1},M_{n_{1}})$ be arbitrary. As $\varphi(f)$ is continuous,
		$$
		\lim_{n}f(\tau(y_{n}))=\lim_{n}\varphi(f)(y_{n})=\varphi(f)(y)=f(\tau(y)),
		$$
		which proves that $(\tau(y_{n}))_{n}$ converges to $\tau(y)$, since functions in $C(X_{1},M_{n_{1}})$ separate points. Thus, $\tau$ is continuous.
		
		To see that $\tau$ is closed, fix a closed subset $F$ of $Y_{2}$, and suppose that $(x_{n})_{n}$ is a sequence in $\tau(F)$ converging to a point $x\in X_{1}$. Choose, for each $n$, a point $y_{n}\in F$ with $\tau(y_{n})=x_{n}$. Since $F$ is compact, we may assume (by passing to a subsequence) that $(y_{n})_{n}$ converges to a point $y$ in $F$. Letting $f\in C(X_{1},M_{n_{1}})$ be arbitrary, it follows that 
		$$
		f(x)=\lim_{n}f(x_{n})=\lim_{n}f(\tau(y_{n}))=\lim_{n}\varphi(f)(y_{n})=\varphi(f)(y)=f(\tau(y)).
		$$ 
		Since this holds for all $f\in C(X_{1},M_{n_{1}})$, it must be that $x=\tau(y)\in \tau(F)$, which proves that $\tau$ is closed, completing the proof of \autoref{closed and continuous}. 
	\end{proof}
	
	Continuing with the proof of \autoref{homcomb}, let $Z:=X_{1}\sqcup X_{2}$. Then $Z$ is a compact metric space. Given $z\in Z$, we define $[z]$ as follows:
	$$
	[z]:=
	\begin{aligned}
	\begin{cases}
	\{z\}&\text{ if }z\in X_{2}\setminus Y_{2}\\
	\{z\}\cup\tau^{-1}(z)&\text{ if }z\in X_{1}\\
	\{\tau(z)\}\cup\tau^{-1}(\tau(z))&\text{ if }z\in Y_{2}.
	\end{cases}
	\end{aligned}
	$$
	Let $Z^{*}:=\{[z]:z\in Z\}$ and let $p\colon Z\to Z^{*}$ denote the canonical surjection $p(z):=[z]$. Then $Z^{*}$ is a collection of sets that partition $Z$. We equip it with the quotient topology induced by $p$; that is, a set $U\subset Z^{*}$ is open in $Z^{*}$ if and only if $p^{-1}(U)$ is open in $Z$. Since $Z$ is compact, so is $Z^{*}$. To establish that $Z^{*}$ is in fact a metric space, it suffices to ensure that it is Hausdorff. Indeed, letting $w(Y)$ denote the smallest cardinality of a basis for a given topological space $Y$, it follows by Theorem 3.1.22 of \cite{eng} that $w(Z^{*})\leq w(Z)$. Since a compact Hausdorff space is metrizable if and only if it has a countable basis, showing that $Z^{*}$ is Hausdorff would guarantee that it is also metrizable. 
	
	\begin{claim}
		\label{hausdorff}
		$Z^{*}$ is Hausdorff.
	\end{claim}
	
	\begin{proof}[Proof of \autoref{hausdorff}]
		Suppose $z_{1},z_{2}\in Z$ with $[z_{1}]\not=[z_{2}]$. Let us show that $[z_{1}]$ and $[z_{2}]$ can be separated by open sets in $Z^{*}$. Without loss of generality, we must be in one of the following four cases. 
		
		Case one: $z_{1},z_{2}\in (X_{2}\setminus Y_{2})\cup (X_{1}\setminus \tau(Y_{2}))$. In this case, it is easy to see (since $Y_{2}$ and $\tau(Y_{2})$ are closed) that we may choose open sets $U_{1}\ni z_{1}$ and $U_{2}\ni z_{2}$ in $Z$ that are disjoint and such that $U_{i}\subset X_{2}\setminus Y_{2}$ if $z_{i}\in X_{2}\setminus Y_{2}$ and $U_{i}\subset X_{1}\setminus \tau(Y_{2})$ if $z_{i}\in X_{1}\setminus \tau(Y_{2})$. Since $p|_{(X_{2}\setminus Y_{2})\cup (X_{1}\setminus \tau(Y_{2}))}$ is a bijection, the sets $p(U_{1})$ and $p(U_{2})$ are open in $Z^{*}$, disjoint, and contain $[z_{1}]$ and $[z_{2}]$, respectively. 
		
		Case two: $z_{1}\in X_{1}\setminus \tau(Y_{2})$ and $z_{2}\in Y_{2}\cup \tau(Y_{2})$. Choose disjoint sets $U_{1}\ni z_{1}$ and $U_{2}\supset \tau(Y_{2})$ that are open in $X_{1}$. Let $V_{1}:=p(U_{1})\ni[z_{1}]$ and $V_{2}:=p(U_{2}\cup X_{2})\ni[z_{2}]$ and note that $V_{1}\cap V_{2}=\varnothing$. Since $p^{-1}(V_{1})=U_{1}$ and $p^{-1}(V_{2})=U_{2}\cup X_{2}$ are both open in $Z$, it follows that $V_{1}$ and $V_{2}$ are open in $Z^{*}$. 
		
		Case three: $z_{1}\in X_{2}\setminus Y_{2}$ and $z_{2}\in Y_{2}\cup \tau(Y_{2})$. Choose disjoint sets $U_{1}\ni z_{1}$ and $U_{2}\supset Y_{2}$ that are open in $X_{2}$. Let $V_{1}:=p(U_{1})\ni[z_{1}]$ and $V_{2}:=p(X_{1}\cup U_{2})\ni[z_{2}]$ and note that $V_{1}\cap V_{2}=\varnothing$. Since $p^{-1}(V_{1})=U_{1}$ and $p^{-1}(V_{2})=X_{1}\cup U_{2}$ are both open in $Z$, it follows that $V_{1}$ and $V_{2}$ are open in $Z^{*}$. 
		
		Case four: $z_{1},z_{2}\in Y_{2}\cup\tau(Y_{2})$. We may assume without loss of generality that $z_{1},z_{2}\in \tau(Y_{2})$. Choose sets $U_{1}\ni z_{1}$ and $U_{2}\ni z_{2}$, which are open in $X_{1}$ and satisfy $\overline{U_{1}}\cap\overline{U_{2}}=\varnothing$. Since $\tau$ is continuous, there are open subsets $V_{1}$ and $V_{2}$ of $X_{2}$ such that $\tau^{-1}(U_{1})=V_{1}\cap Y_{2}$ and $\tau^{-1}(U_{2})=V_{2}\cap Y_{2}$. Choose disjoint open subsets $W_{1}$ and $W_{2}$ of $X_{2}$ containing $\tau^{-1}(\overline{U_{1}})$ and $\tau^{-1}(\overline{U_{2}})$, respectively. Put $\mathcal{O}_{1}:=V_{1}\cap W_{1}$ and $\mathcal{O}_{2}:= V_{2}\cap W_{2}$ and let $\mathcal{E}_{1}:=p(\mathcal{O}_{1}\cup U_{1})$ and $\mathcal{E}_{2}:=p(\mathcal{O}_{2}\cup U_{2})$. Note that $[z_{1}]\in \mathcal{E}_{1}$ and $[z_{2}]\in\mathcal{E}_{2}$. Let us show that $\mathcal{E}_{1}$ and $\mathcal{E}_{2}$ are disjoint and open in $Z^{*}$. Suppose $t_{1}\in \mathcal{O}_{1}\cup U_{1}$ and $t_{2}\in \mathcal{O}_{2}\cup U_{2}$. Assume first that $t_{1}\in U_{1}$. If $t_{2}\in U_{2}\cup (X_{2}\setminus Y_{2})$, then $[t_{1}]\not=[t_{2}]$ since $U_{1}\cap U_{2}=\varnothing$. If instead $t_{2}\in \mathcal{O}_{2}\cap Y_{2}$, then $\tau(t_{2})\in U_{2}$, so that we again have $[t_{1}]\not=[t_{2}]$. A symmetric analysis shows that $[t_{1}]\not=[t_{2}]$ if $t_{2}\in U_{2}$. Thus, we may assume $t_{1}\in\mathcal{O}_{1}$ and $t_{2}\in\mathcal{O}_{2}$. If either $t_{1}$ or $t_{2}$ is in $X_{2}\setminus Y_{2}$, then $[t_{1}]\not=[t_{2}]$ since $t_{1}\not=t_{2}$, as $\mathcal{O}_{1}$ and $\mathcal{O}_{2}$ are disjoint. If instead $t_{1}\in\mathcal{O}_{1}\cap Y_{2}$ and $t_{2}\in\mathcal{O}_{2}\cap Y_{2}$, then $\tau(t_{1})\in U_{1}$, $\tau(t_{2})\in U_{2}$, and once again $[t_{1}]\not=[t_{2}]$. It follows that $\mathcal{E}_{1}\cap\mathcal{E}_{2}=\varnothing$. 
		
		It remains to be shown that $\mathcal{E}_{1}$ and $\mathcal{E}_{2}$ are open in $Z^{*}$. Owing to the symmetry of the setup, we only show that $\mathcal{E}_{1}$ is open in $Z^{*}$, and to do this it is sufficient to prove that $p^{-1}(\mathcal{E}_{1})\cap X_{1}=U_{1}$ and $p^{-1}(\mathcal{E}_{1})\cap X_{2}=\mathcal{O}_{1}$. Assume we are given $t\in p^{-1}(\mathcal{E}_{1})\cap X_{1}$. Thus $[t]\in p(\mathcal{O}_{1})\cup p(U_{1})$. If there is an $s\in U_{1}$ such that $[t]=[s]$, then, since both $t$ and $s$ lie in $X_{1}$, it must be that $t=s$. If instead there is an $s\in\mathcal{O}_{1}$ such that $[t]=[s]$, then it follows that $s\in Y_{2}$, and hence, that $\tau(s)\in U_{1}$. Thus, $[t]=[s]=[\tau(s)]$, from which we deduce as before that $t=\tau(s)\in U_{1}$. Therefore, we may conclude that $p^{-1}(\mathcal{E}_{1})\cap X_{1}\subset U_{1}$, and hence, $p^{-1}(\mathcal{E}_{1})\cap X_{1}=U_{1}$. Now, suppose that we are given $t\in p^{-1}(\mathcal{E}_{1})\cap X_{2}$. As before, $[t]\in p(\mathcal{O}_{1})\cup p(U_{1})$. Suppose first that there is an $s\in\mathcal{O}_{1}$ such that $[t]=[s]$. If either $t$ or $s$ is in $X_{2}\setminus Y_{2}$, then $t=s\in \mathcal{O}_{1}$; otherwise, it must be that $s,t\in Y_{2}$ and, in particular, $\tau(t)=\tau(s)\in U_{1}$. Therefore, $t\in V_{1}\cap W_{1}=\mathcal{O}_{1}$. If instead there is an $s\in U_{1}$ such that $[t]=[s]$, then it must be that $t\in Y_{2}$ and $\tau(t)=s\in U_{1}$, which  implies (as above) that $t\in\mathcal{O}_{1}$. Thus, $p^{-1}(\mathcal{E}_{1})\cap X_{2}\subset \mathcal{O}_{1}$, and hence, $p^{-1}(\mathcal{E}_{1})\cap X_{2}=\mathcal{O}_{1}$. Therefore, by our analysis, this implies that $\mathcal{E}_{1}$ and $\mathcal{E}_{2}$ are open in $Z^{*}$. This completes the proof of \autoref{hausdorff}.  
	\end{proof}

	\renewcommand\qedsymbol{/\hspace{-2pt}/}
	Returning to the proof of \autoref{homcomb}, define $\Lambda\colon C(X_{1},M_{n})\oplus_{C(Y_{2},M_{n})}C(X_{2},M_{n})\to C(Z^{*},M_{n})$ by:
	$$
	\Lambda((f,g))([z]):=
	\begin{aligned}
	\begin{cases}
	f(z)&\text{ if }z\in X_{1}\\
	g(z)&\text{ if }z\in X_{2}.
	\end{cases}
	\end{aligned}
	$$
	To conclude the proof, let us show that $\Lambda$ is a well-defined $^*$-isomorphism. To see that $\Lambda$ is well defined, suppose $z_{1},z_{2}\in Z$ and that $[z_{1}]=[z_{2}]$. Unless $z_{1}=z_{2}$, this implies that one of the two points is in the decomposition of the other. Assume without loss of generality that $\tau(z_{2})=z_{1}$. Then for all $(f,g)\in C(X_{1},M_{n})\oplus_{C(Y_{2},M_{n})}C(X_{2},M_{n})$, we have $g(z_{2})=\varphi(f)(z_{2})=f(\tau(z_{2}))=f(z_{1})$. This shows that $\Lambda$ is well defined. It is clear that $\Lambda$ is an injective $^*$-homomorphism. To see surjectivity, suppose $h\in C(Z^{*},M_{n})$ and define $f:=h\circ p|_{X_{1}}\in C(X_{1},M_{n})$ and $g:=h\circ p|_{X_{2}}\in C(X_{2},M_{n})$. Given $y\in Y_{2}$, we have
	$$
	g(y)=h([y])=h([\tau(y)])=f(\tau(y))=\varphi(f)(y),
	$$ 
	so that $(f,g)\in C(X_{1},M_{n})\oplus_{C(Y_{2},M_{n})}C(X_{2},M_{n})$. Moreover, $\Lambda((f,g))=h$, proving that $\Lambda$ is surjective. The proof of \autoref{homcomb} is now complete. 
\end{proof}

Applying \autoref{homcomb} inductively, we obtain the following corollary.

\begin{corollary}
	\label{homogeneous DSH}
	Every $n$-homogeneous DSH algebra is isomorphic to a full matrix algebra, i.e., isomorphic to $C(X,M_{n})$ for some compact metric space $X$.
\end{corollary}


\section{Stable Rank}
\label{ch: stable rank}

This section focuses on simple inductive limits of DSH algebras with diagonal bonding maps. \cref{sc: main proof} contains the principal result, which states that every limit algebra of this type necessarily has stable rank one (see \autoref{main}). In \cref{sc: cp}, \autoref{main} is applied to obtain two results about simple dynamical crossed products. Given an infinite compact metric space $T$ and a minimal homeomorphism $h\colon T\to T$, we show that every orbit-breaking subalgebra of the induced dynamical crossed product $\mathrm{C}^{*}(\mathbb{Z},T,h)$ associated to any non-isolated point is a simple inductive limit of DSH algebras with diagonal maps (see \autoref{Ax}). Consequently, we are able to show that $\mathrm{C}^{*}(\mathbb{Z},T,h)$ has stable rank one (see \autoref{main2}), and that $\mathcal{Z}$-stability is determined for such an algebra by strict comparison of positive elements (see \autoref{main3}). 

The proof of \autoref{main} is quite technical and requires several lemmas, which are developed in \cref{sc: prelim lemmas} and \cref{sc: main lemmas}. In \cref{sc: prelim lemmas} facts concerning continuous paths of unitary matrices are established. These results are used in \cref{sc: main lemmas} to construct certain unitary elements in DSH algebras that are needed to prove \autoref{main}. Before formulating these lemmas, in \cref{sc: outline of proof}, we provide a more detailed overview of how they come together to prove \autoref{main}, and we compare and contrast our approach to that used by Elliott, Ho, and Toms in \cite{EHT}.

\subsection{Outline of the Proof of the Main Theorem}
\label{sc: outline of proof}

\cref{sc: main lemmas} consists of all of the lemmas that are used in the proof of \autoref{main} in \cref{sc: main proof}, with the following dependency diagram:

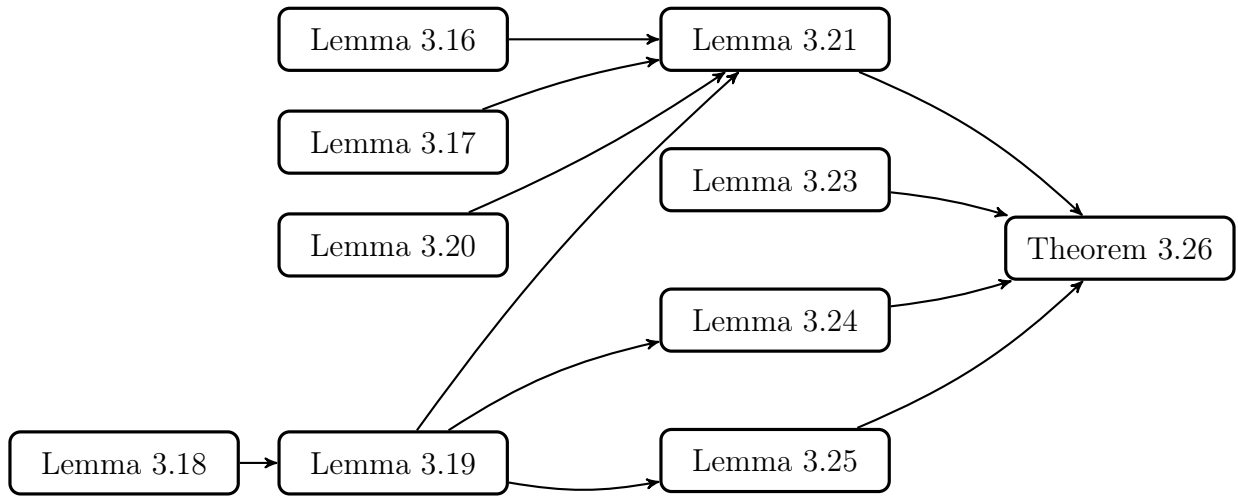
\begin{figure}[H]
	\centering
	\begin{tikzpicture}[node distance=1cm, auto]
	\node (dummy) {};
	\node[punkt, inner sep=5pt, below=2.2cm of dummy] (main) {\autoref{main}};
	\node[punkt, inner sep=5pt, left=2.9cm of dummy] (far out) {\autoref{far out}};
	\node[punkt, inner sep=5pt, below=1cm of far out] (block on set) {\autoref{block on set}};
	\node[punkt, inner sep=5pt, below=1cm of block on set] (cross shift) {\autoref{cross shift}};
	\node[punkt, inner sep=5pt, below=1cm of cross shift] (lower triangular) {\autoref{lower triangular}};
	\node[punkt, inner sep=5pt, left=2cm of far out] (simplicity) {\autoref{simplicity}};
	\node[punkt, inner sep=5pt, below=.5cm of simplicity] (perturb) {\autoref{perturb}};
	\node[punkt, inner sep=5pt, below=.5cm of perturb] (big enough lemma) {\autoref{big enough lemma}};
	\node[punkt, inner sep=5pt, below=3.4cm of perturb] (indicators) {\autoref{indicators}};
	\node[punkt, inner sep=5pt, left=0.5cm of indicators] (indicatorprep) {\autoref{indicatorprep}};
	
	\draw [->, thick] (far out) to [bend left=10] (main);  
	\draw [->, thick] (block on set) to [bend left=5] (main);  
	\draw [->, thick] (cross shift) to [bend right=5] (main);  
	\draw [->, thick] (lower triangular) to [bend right=10] (main);
	\draw [->, thick] (simplicity) -- (far out); 
	\draw [->, thick] (indicatorprep) -- (indicators);  
	\draw [->, thick] (perturb) to [bend left=5] (far out);
	\draw [->, thick] (big enough lemma) to [bend right=5] (far out);
	\draw [->, thick] (indicators) to [bend left=6] (far out);
	\draw [->, thick] (indicators) to [bend left=10] (cross shift);
	\draw [->, thick] (indicators) to [bend right=10] (lower triangular);
	\end{tikzpicture}	
	\caption{Dependency chart for the main lemmas used in the proof of \autoref{main}.}
	\label{chart}
\end{figure}

\noindent
Let us now outline the importance of each of these lemmas and give a brief overview of how they are used to prove \autoref{main}.

Our general strategy for proving that a simple inductive limit of DSH algebras with diagonal maps has stable rank one is essentially the one in \cite{EHT}. We start with a given element $f$ in the limit algebra $A$, which may be assumed to lie in some finite-stage building block $A_{j}$. If $f$ is invertible, then there is nothing to prove, and so we may assume that $f$ is not invertible. The goal is then to show that the image $\psi_{j',j}(f)$ of $f$ in a later stage algebra $A_{j'}$ is close to an invertible in $A_{j'}$. 

If we approximate $\psi_{j',j}(f)$, multiply this approximation by unitaries, approximate again, multiply the new approximation by unitaries, and show that an element thus obtained is close to an invertible, then, upon unpacking the approximations, it follows that $\psi_{j',j}(f)$ is close to an invertible in $A_{j'}$. Finally, as R\o rdam observed in \cite{Ror}, every nilpotent element of a unital $\mathrm{C}^*$-algebra is close to an invertible. Therefore, it suffices to show that an element, obtained from $\psi_{j',j}(f)$ as above, is nilpotent. 

To execute this strategy, we proceed as follows. In \autoref{perturb}, we first use \autoref{non-invertible at point} to show that there is a point in one of the base spaces $X_{i}$ of $A_{j}$ at which $f_{i}$ is a non-invertible matrix. After multiplying by unitary matrices on the left and right we obtain a new matrix whose first row and column contain only zeros (or one that has a \textit{zero cross at index $1$} (see \autoref{zero cross})). We show that after perturbing $f$ slightly, we may multiply this perturbation $f'$ on the left and right by unitaries $w,v\in A_{j}$, so that $wf'v$ has a zero cross at index $1$ not just at one point, but over a whole open subset of the spectrum of $A_{j}$.

By \autoref{injective}, we may assume the maps in the given sequence are injective. Hence, in \autoref{far out}, we may apply our simplicity criterion (\autoref{simplicity}) with the open subset of the spectrum obtained above to conclude that in some later stage algebra $A_{j'}$, the diagonal image $\psi_{j',j}(wf'v)$ has ``many'' (see the following paragraphs) zero crosses at every point in each base space of $A_{j'}$; because of simplicity and the fact that the maps in the sequence are diagonal, this ``many'' may be taken (using \autoref{big enough lemma}) to be as large as desired. We are then able to construct unitaries $V,V'\in A_{j'}$ that organize the location of these zero crosses, so that the element $f''=V\psi_{j',j}(f')V'$ has ``many'' zero crosses occurring at tractable locations at each point in every base space of $A_{j'}$. 

We use \autoref{block on set} to approximate $f''$ by a function $g\in A_{j'}$ that preserves the zero crosses of $f''$ at each point, and, in addition, extends the block-diagonal structure of the algebra to neighbourhoods of the closed subsets of the base spaces (the $Y_{i}$'s in the definition of $A_{j'}$). This allows us, in \autoref{cross shift}, to conjugate $g$ by a unitary $W\in A_{j'}$, so that in the resulting conjugation $g'=WgW^{*}$, the zero crosses of $g$ are grouped together into block zero crosses at every point in each of the base spaces of $A_{j'}$.

The unitaries $V,V'$, and $W$ above are constructed in such a way that at every point in each base space the \textit{bandwidth}, which measures how far a non-zero entry can occur from the diagonal in a matrix (see \autoref{bandwidth}), of $g'$ at that point is bounded above by a quantity independent of $j'$. Thus, by ensuring that the ``many'' above is at least as large as this upper bound, we are able to construct a unitary $W'$ in \autoref{lower triangular} that shifts the block zero cross mentioned above so that $g'W'$ is strictly lower triangular at each point. This ensures that $g'W'$ is nilpotent and yields the desired result. 

The unitaries $V,V',W$, and $W'$ above are all defined using continuous paths of unitaries between permutation matrices (see \autoref{unitaries} and \autoref{transposition prod}). In \autoref{indicatorprep}, we construct certain indicator-function-like elements of DSH algebras, the final versions of which (\autoref{indicators}) help to define $V,V',W$, and $W'$ by allowing us to implement the continuous paths of unitary matrices constructed in \cref{sc: prelim lemmas} in the DSH framework. Their job is to tell the continuous paths used in defining these unitaries which rows and columns to shift around, so as to ensure that they respect the decomposition structure of the algebra and that the zero crosses are achieved in the target locations.

The proof of \autoref{main} shares many similarities with the original AH proof of Elliott, Ho, and Toms found in \cite{EHT}. In particular, in the case that all of the DSH algebras in the context of \autoref{main} are homogeneous (hence, by \autoref{homogeneous DSH}, full matrix algebras), the unitaries $V,V',W$, and $W'$ constructed above essentially reduce the those constructed in \cite{EHT}. For a more in-depth analysis of this, see \S 5.1 of \cite{alb}, where it is also observed that the AH proof does not require the full matrix algebra building blocks in the inductive limit to be separable. In our ASH setting, however, separability is necessary since the indicator-function-like elements constructed in \autoref{indicators}, which are not required in the AH case, rely on the assumption that the base spaces of any given DSH algebra are metrizable.

\subsection{Preliminary Lemmas}
\label{sc: prelim lemmas}

The purpose of this subsection is to introduce some continuous paths of unitary matrices and prove certain facts about them. These paths will be used in the sequel to construct the unitaries in the DSH algebras used in the proof of the main result.

\begin{definition}[Zero cross]
	\label{zero cross}
	Given a matrix $D\in M_{n}$ and $1\leq k\leq n$, we say that $D$ has a \textit{zero cross at index $k$} provided that each entry in the $k$th row and column of $D$ is $0$. 
\end{definition}

\begin{definition}[Bandwidth of a matrix]
	\label{bandwidth}
	Given a matrix $D\in M_{n}$, we let 
	$$
	\mathfrak{r}(D):=\min\{m\geq 0:D_{i,j}=0\text{ whenever }|i-j|\geq m\}
	$$ 
	if it exists, and $\mathfrak{r}(D):= n$ otherwise, and we call this number the \textit{bandwidth} of $D$. 
\end{definition}

\begin{definition}[see \cite{EHT}]
	\label{unitaries}
	Given $n\in\NN$ and a permutation $\pi\in S_{n}$, let $U[\pi]$ denote the permutation unitary in $M_{n}$ obtained from the identity matrix by moving the $i$th row to the $\pi(i)$th row. If we are given a transposition $(i \ j) \in S_{n}$, let $u_{(i \  j)}\colon[0,1]\to \mathcal{U}(M_{n})$ denote a continuous path of unitaries with the following properties:
	\begin{enumerate}
		\item $u_{(i \ j)}(0)=1_{n}$;
		\item $u_{(i \ j)}(1)=U[(i \ j)]$;
		\item for all $0\leq \theta\leq 1$, $u_{(i \ j)}(\theta)$ may only differ from the identity matrix at entries $(i,i)$, $(i,j)$, $(j, i)$, and $(j,j)$. 
	\end{enumerate}
\end{definition}

\begin{lemma}
	\label{untouched crosses}
	Let $n,M,l\in\NN$ with $l+M-1\leq n$, and let $(\xi_{l},\ldots,\xi_{l+M-1})$ be a vector in $[0,1]^{M}$. Put 
	$$
	U:=\prod_{t=1}^{M-1}u_{(l \ l+t)}(\xi_{l+t})\in\mathcal{U}(M_{n}),
	$$
	where each $u_{(l \ l+t)}\colon [0,1]\to M_{n}$ is a connecting path of unitaries as described in \autoref{unitaries}.
	\begin{enumerate}[label=(\alph*)]
		\item Suppose $D\in M_{n}$, $\xi\in [0,1]$, and $(l_{1} \ l_{2})\in S_{n}$. If $D$ has a zero cross at index $l\not=l_{1},l_{2}$, then so does $u_{(l_{1} \ l_{2})}(\xi)Du_{(l_{1} \ l_{2})}(\xi)^{*}$. 
		\item Suppose $D\in M_{n}$. If $D$ has a zero cross at index $l'\in\{1,\ldots,n\}\setminus \{l,\ldots,l+M-1\}$, then so does $UDU^{*}$. 
		\item Suppose $D\in M_{n}$ is such that for all $l\leq l'\leq l+M-1$, $D$ has a zero cross at index $l'$ whenever $\xi_{l'}>0$. If at least one of $\xi_{l},\ldots,\xi_{l+M-1}$ is $1$, then $UDU^{*}$ has a zero cross at index $l$.
	\end{enumerate}
\end{lemma}
\begin{proof}
	Let us start by proving \textit{(a)}. Suppose $D$ has a zero cross at index $l\not=l_{1},l_{2}$. By property (3) of \autoref{unitaries}, the $l_{1}$th and $l_{2}$th columns of $Du_{(l_{1} \ l_{2})}(\xi)^{*}$ are linear combinations of the $l_{1}$th and $l_{2}$th columns of $D$, while every other column is identical to its corresponding column in $D$. Since $l\not=l_{1},l_{2}$ and since every entry in the $l$th row of $D$ is zero, it follows that $Du_{(l_{1} \ l_{2})}(\xi)^{*}$ has a zero cross at index $l$. A similar analysis involving rows shows that $u_{(l_{1} \ l_{2})}(\xi)Du_{(l_{1} \ l_{2})}(\xi)^{*}$ has a zero cross at index $l$, which proves \textit{(a)}. Looking at the definition of $U$, we see that \textit{(b)} follows from $M-1$ applications of \textit{(a)}. 
	
	Let us now prove \textit{(c)}. Suppose that $D$ has a zero cross at index $l'$ whenever $\xi_{l'}>0$ and that at least one of $\xi_{l},\ldots,\xi_{l+M-1}$ is $1$. Let $T:=\{l+1\leq q\leq l+M-1:\xi_{q}>0\}$. If $\xi_{l+t}=0$, we have $u_{(l \ l+t)}(\xi_{l+t})=1_{n}$. Hence, 
	$$
	U:=
	\begin{aligned}
	\begin{cases}
	u_{(l \ l_{1})}(\xi_{l_{1}})\cdots u_{(l \ l_{r})}(\xi_{l_{r}}) & \text{if }T=\{l_{1}<\cdots<l_{r}\}\\
	\qquad\qquad\quad1_{n} & \text{if }T=\varnothing.
	\end{cases}
	\end{aligned}
	$$
	If $T=\varnothing$, then, since at least one of $\xi_{l},\ldots,\xi_{l+M-1}$ is $1$, it must be that $\xi_{l}=1$. Hence, $UDU^{*}=D$ has a zero cross at index $l$ in this case by the assumption in the lemma. Thus, we may assume $T\not=\varnothing$, so that $D$ has zero crosses at indices $l_{1},\ldots,l_{r}$. We consider two cases.
	
	Case one: $\xi_{l_{s}}<1$ for all $1\leq s\leq r$. In this case, as we argued above, it must be that $D$ has a zero cross at index $l$. When conjugating $D$ by $u_{(l \ l_{r})}(\xi_{l_{r}})$, we can see by property (3) of \autoref{unitaries} that $u_{(l \ l_{r})}(\xi_{l_{r}})$ is only acting on two zero crosses (the one at index $l$ and the one at index $l_{r}$) of $D$ and, hence, 
	$$
	u_{(l \ l_{r})}(\xi_{l_{r}})Du_{(l \ l_{r})}(\xi_{l_{r}})^{*}=D.
	$$
	From this we can inductively see that $UDU^{*}=D$, which has a zero cross at index $l$. 
	
	Case two: $\xi_{l_{s}}=1$ for some $1\leq s\leq r$. Let 
	$$
	D':=\left(\prod_{p=s+1}^{r}u_{(l \ l_{p})}(\xi_{l_{p}})\right)D\left(\prod_{p=s+1}^{r}u_{(l \ l_{p})}(\xi_{l_{p}})\right)^{*}.
	$$
	Then $r-s$ applications of $(a)$ show that $D'$ has zero crosses at indices $l_{1},\ldots,l_{s}$. Note that 
	$$
	\begin{aligned}
	UDU^{*}
	=\left(\prod_{p=1}^{s}u_{(l \ l_{p})}(\xi_{l_{p}})\right)D'\left(\prod_{p=1}^{s}u_{(l \ l_{p})}(\xi_{l_{p}})\right)^{*}
	=\left(\prod_{p=1}^{s-1}u_{(l \ l_{p})}(\xi_{l_{p}})\right)E\left(\prod_{p=1}^{s-1}u_{(l \ l_{p})}(\xi_{l_{p}})\right)^{*},
	\end{aligned}
	$$ 
	where $E:=U[(l \ l_{s})]D'U[(l \ l_{s})]^{*}$. 
	Since $D'$ has a zero cross at index $l_{s}$, conjugating it by $U[(l \ l_{s})]$ brings this zero cross to index $l$, so that the matrix $E$ has zero crosses at indices $l,l_{1},\ldots,l_{s-1}$. Hence, as in the argument used in case one, the matrix $E$ is unaltered when conjugated by $\prod_{p=1}^{s-1}u_{(l \ l_{p})}(\xi_{l_{p}})$. Therefore, $UDU^{*}=E$, which has a zero cross at index $l$. This proves \textit{(c)} and establishes the lemma.
\end{proof}

\begin{definition}
	\label{matrix conjugator}
	Let $n\in\mathbb{N}$. For $1\leq i\leq j\leq n$, let $\delta_{j}^{i}\colon[0,1]\to[0,1]$ be given by the following definition:
	$$
	\begin{aligned}
	\delta_{j}^{i}(\xi):=
	\begin{cases}
	0&\text{ if }0\leq\xi\leq\frac{i-1}{j}\\
	\text{linear}&\text{ if }\frac{i-1}{j}\leq\xi\leq\frac{i}{j}\\
	1&\text{ if }\frac{i}{j}\leq\xi\leq 1. 
	\end{cases}
	\end{aligned}
	$$
	Moreover, for $1\leq i<j\leq n$, let $w_{j}^{i}\in C([0,1],M_{n})$ be the unitary defined by 
	$$
	w_{j}^{i}(\xi):=u_{(i \ i+1)}(\delta_{j-i}^{j-i}(\xi)) u_{(i+1 \ i+2)}(\delta_{j-i}^{j-i-1}(\xi))\cdots u_{(j-1 \ j)}(\delta_{j-i}^{1}(\xi)),
	$$	
	where the unitaries $u_{(k \ k+1)}\colon [0,1]\to M_{n}$ are those of \autoref{unitaries}, and set $w_{i}^{i}\equiv 1_{n}$. In particular, 
	\begin{equation}
	\label{permutation}
	\begin{aligned}
	w_{j}^{i}(1)=u_{(i \ i+1)}(1)\cdots u_{(j-1 \ j)}(1)=U[(i \  \ i+1 \ \cdots \ j)].
	\end{aligned}
	\end{equation}
\end{definition}

\begin{lemma}
	\label{matrix1}
	Suppose $D\in M_{n}$ has a zero cross at index $j$. Then, $\mathfrak{r}(w_{j}^{1}(1)Dw_{j}^{1}(1)^{*})\leq \mathfrak{r}(D)$ and, for $2\leq i\leq j$, $\mathfrak{r}(w_{j}^{i}(1)Dw_{j}^{i}(1)^{*})\leq\mathfrak{r}(D)+1$. 
\end{lemma}
\begin{proof}
	By \autoref{permutation}, $w_{j}^{1}(1)=U[(1 \ 2 \ \cdots \ j)]$. Consider the matrix $D$ broken up into the four regions created by the zero cross at $j$, together with the matrix $w_{j}^{1}(1)Dw_{j}^{1}(1)^{*}$:
	
	\begin{figure}[H]
		\begin{minipage}{0.5\textwidth}
			\centering
			\includegraphics[scale=1]{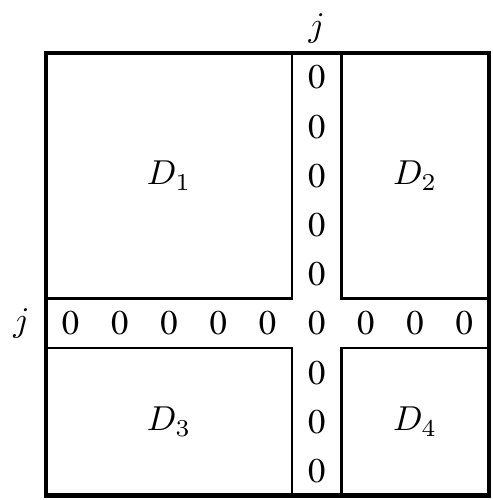}
			\caption{The Matrix $D$}
		\end{minipage}\hfill
		\begin{minipage}{.5\textwidth}
			\centering
			\includegraphics[scale=1]{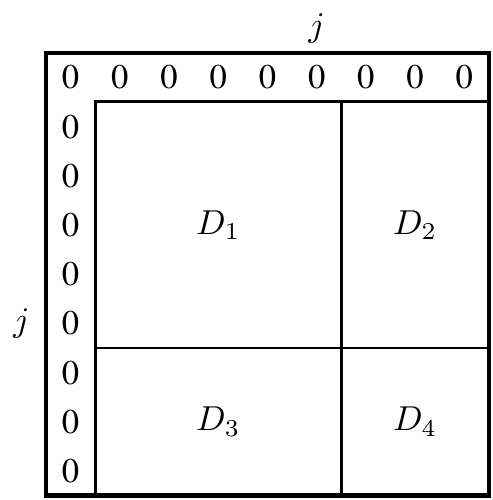}
			\caption{The Matrix $w_{j}^{1}(1)Dw_{j}^{1}(1)^{*}$}
		\end{minipage}
	\end{figure}
	\noindent
	Since no non-zero entry gets shifted away from the diagonal, $\mathfrak{r}(w_{j}^{1}(1)Dw_{j}^{1}(1)^{*})\leq\mathfrak{r}(D)$.
	
	Suppose now that $2\leq i\leq j$. If $i=j$, then the desired inequality is trivial, so we may assume that $i<j$. By \autoref{permutation}, $w_{j}^{i}(1)=U[(i \ i+1 \ \cdots \ j)]$. Consider the matrix $D$ broken up into the following nine regions created by the zero cross at $j$ and the $i$th row and column, together with the matrix $w_{j}^{i}(1)Dw_{j}^{i}(1)^{*}$:
	
	\begin{figure}[H]
		\begin{minipage}{0.5\textwidth}
			\centering
			\includegraphics[scale=1]{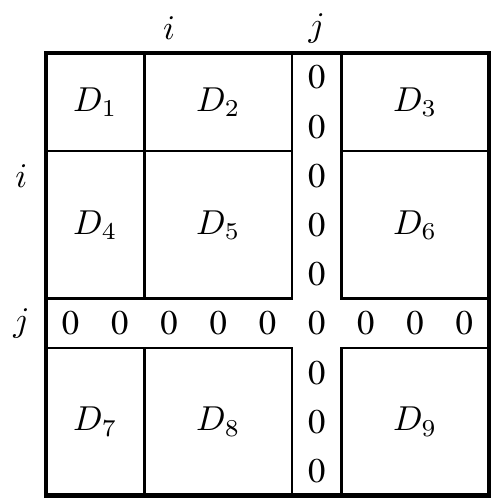}
			\caption{The Matrix $D$}
		\end{minipage}\hfill
		\begin{minipage}{.5\textwidth}
			\centering
			\includegraphics[scale=1]{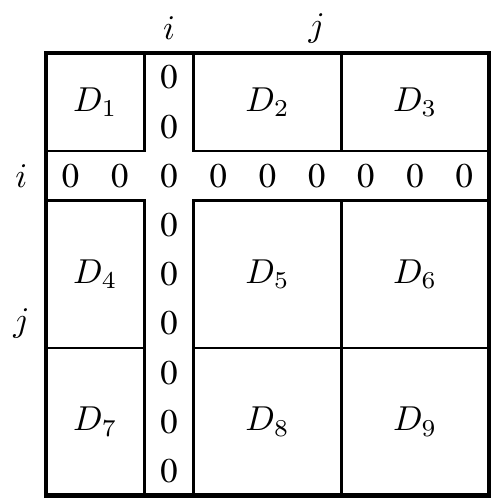}
			\caption{The Matrix $w_{j}^{i}(1)Dw_{j}^{i}(1)^{*}$}
		\end{minipage}
	\end{figure}
	\noindent
	
	With the exception of $D_{2}$ and $D_{4}$, which get shifted one unit away from the diagonal, no entry in the other seven regions is moved away from the diagonal. Therefore,  $\mathfrak{r}(w_{j}^{i}(1)Dw_{j}^{i}(1)^{*})\leq \mathfrak{r}(D)+1$, which proves \autoref{matrix1}. 
\end{proof}

\begin{lemma}
	\label{matrix2}
	Suppose $D\in M_{n}$ has a zero cross at index $j$. If $1\leq i\leq j$ and $\xi\in [0,1]$, then $\mathfrak{r}(w_{j}^{i}(\xi)D w_{j}^{i}(\xi)^{*})\leq \mathfrak{r}(D)+2$. 
\end{lemma}

\begin{proof}
	Fix $1\leq i\leq j$ and $\xi\in[0,1]$. If $\xi=0$ or if $i=j$, then $w_{j}^{i}(\xi)=1_{n}$ and the result is trivial. Hence, we may assume $i<j$ and $\xi\in(0,1]$. Let $1\leq k\leq j-i$ be the unique integer such that $\xi\in\left(\frac{k-1}{j-i},\frac{k}{j-i}\right]$. Then,
	\begin{equation}
	\label{wij}
	\begin{aligned}
	&w_{j}^{i}(\xi)\\
	=&u_{(i \ i+1)}(0)\cdots u_{(j-k-1\ j-k)}(0)u_{(j-k \ j-k+1)}(\delta_{j-i}^{k}(\xi))u_{(j-k+1 \ j-k+2)}(1)\cdots u_{(j-1 \ j)}(1)\\
	=&u_{(j-k \ j-k+1)}(\delta_{j-i}^{k}(\xi))w^{j-k+1}_{j}(1).
	\end{aligned}
	\end{equation}
	Let $D':=w^{j-k+1}_{j}(1) D w^{j-k+1}_{j}(1)^{*}$. By \autoref{matrix1}, $\mathfrak{r}(D')\leq \mathfrak{r}(D)+1$. Now, consider the conjugation of $D'$ by $u_{(j-k \ j-k+1)}(\delta_{j-i}^{k}(\xi))$, which we denote by $E$. The entries of $D'$ affected by this conjugation lie in one of the following three regions:
	
	\begin{figure}[H]
		\begin{minipage}{0.333\textwidth}
			\includegraphics[scale=.81]{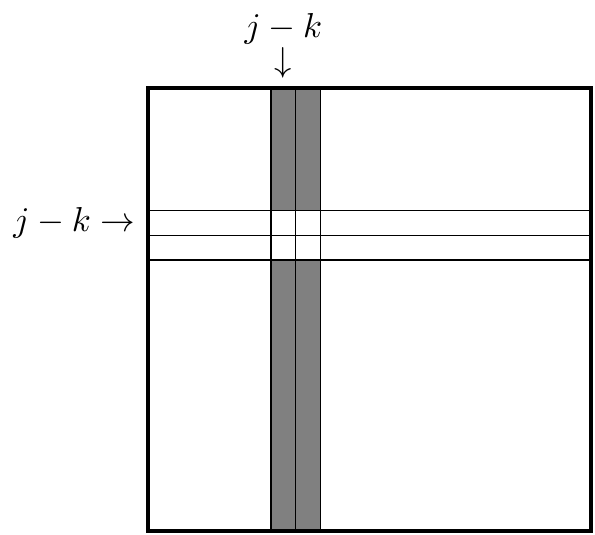}
			\caption{Region A}
			\label{regionA1}
		\end{minipage}\hfill
		\begin{minipage}{0.333\textwidth}
			\includegraphics[scale=.81]{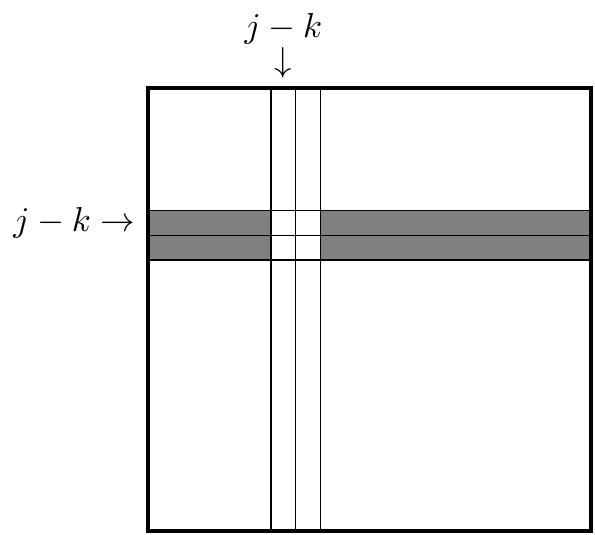}
			\caption{Region B}
			\label{regionB1}
		\end{minipage}\hfill
		\begin{minipage}{0.333\textwidth}
			\includegraphics[scale=.81]{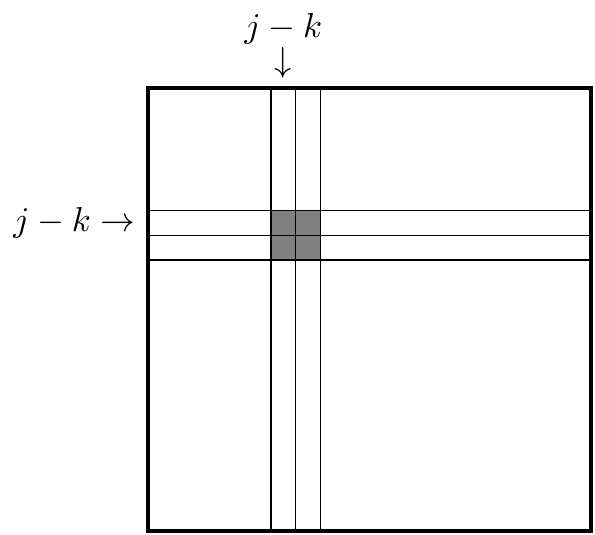}
			\caption{Region C}
			\label{regionC1}
		\end{minipage}
	\end{figure}

	We make the following observations:
	\begin{itemize}
		\item An entry in $E$ lying in \nameref{regionA1} will be a linear combination of the two corresponding shaded entries in $D'$ lying in the same row;
		\item An entry in $E$ lying in \nameref{regionB1} will be a linear combination of the two corresponding shaded entries in $D'$ lying in the same column;
		\item An entry in $E$ lying in \nameref{regionC1} will be a linear combination of the four shaded entries in $D'$ lying in \nameref{regionC1}.
	\end{itemize}
	We see that in all instances, a non-zero entry in $E$ never appears more than one unit further away from the diagonal than a non-zero entry in $D'$. Thus,	
	$$
	\mathfrak{r}(u_{(j-k \ j-k+1)}(\delta_{j-i}^{k}(\xi))D'u_{(j-k \ j-k+1)}(\delta_{j-i}^{k}(\xi))^{*})\leq \mathfrak{r}(D')+1\leq\mathfrak{r}(D)+2,
	$$
	which proves \autoref{matrix2}. 
\end{proof}

\begin{lemma}
	\label{matrix}
	Suppose $n\in\NN$ and $1\leq z_{1}<z_{2}<\cdots<z_{m}\leq n$. There is a unitary $W\in C([0,1],M_{n})$  with the following properties:
	\begin{enumerate}[label=(\alph*)]
		\item $W(0)=1_{n}$;
		\item  if $T\in M_{n}$ has zero crosses at indices $z_{1},\ldots,z_{m}$, then $W(1)TW(1)^{*}$ has zero crosses at indices $1,2,\ldots,m$ and $\mathfrak{r}(W(\xi)TW(\xi)^{*})\leq\mathfrak{r}(T)+2$ for all $\xi\in[0,1]$;
		\item if $b\in\NN$, $T\in M_{n\times b}$, and the rows of $T$ at indices $z_{1},\ldots,z_{m}$  consist entirely of zeros, then the first $m$ rows of $W(1)T$ consist entirely of zeros and for all $\xi\in[0,1]$, 
		$$
		\mathfrak{r}\left(
		\begin{pmatrix}
		0_{n\times n} & W(\xi)T\\
		0_{b\times n} & 0_{b\times b}
		\end{pmatrix}
		\right)\leq\mathfrak{r}\left(
		\begin{pmatrix}
		0_{n\times n} & T\\
		0_{b\times n} & 0_{b\times b}
		\end{pmatrix}
		\right);
		$$  
		\item if $b\in\NN$, $T\in M_{b\times n}$, and the columns of $T$ at indices $z_{1},\ldots,z_{m}$ consist entirely of zeros, then the first $m$ columns of $TW(1)^{*}$ consist entirely of zeros and for all $\xi\in[0,1]$, 
		$$
		\mathfrak{r}\left(
		\begin{pmatrix}
		0_{n\times n} & 0_{n\times b}\\
		TW(\xi)^{*} & 0_{b\times b}
		\end{pmatrix}
		\right)\leq\mathfrak{r}\left(
		\begin{pmatrix}
		0_{n\times n} & 0_{n\times b}\\
		T & 0_{b\times b}
		\end{pmatrix}
		\right).
		$$ 
	\end{enumerate} 
\end{lemma}	

\begin{proof}	
	For $1\leq i\leq j\leq n$, let $\delta_{j}^{i}$ and $w_{j}^{i}\in C([0,1],M_{n})$ be given as in \autoref{matrix conjugator}. Define 
	$$
	W:=(w_{z_{m}}^{1}\circ\delta_{m}^{m}) \cdots(w_{z_{1}}^{1}\circ\delta_{m}^{1}),
	$$	
	which is a unitary in $C([0,1],M_{n})$.
	
	By \autoref{unitaries} and \autoref{matrix conjugator}, $W(0)=w_{z_{m}}^{1}(0)\cdots w_{z_{1}}^{1}(0)=1_{n}$, so that \textit{(a)} holds. Since \textit{(d)} follows immediately from \textit{(c)} by taking adjoints, only \textit{(b)} and \textit{(c)} remain to be verified. Let $\sigma$ denote the permutation 
	$$
	(1 \ 2 \ \cdots \ z_{m})(1 \ 2 \ \cdots \ z_{m-1})\cdots(1 \ 2 \ \cdots \ z_{1})\in S_{n}
	$$
	and note that $\sigma(z_{k})=m-k+1$ for $1\leq k\leq m$. Hence, if $T$ is any matrix in $M_{n}$ (resp. $M_{n\times b}$) with zero crosses (resp. rows) at indices $z_{1},\ldots,z_{m}$, then $U[\sigma]TU[\sigma]^{*}$ (resp. $U[\sigma]T$) has zero crosses (resp. rows) at indices $1,\ldots,m$. Since $W(1)=U[\sigma]$ by \autoref{permutation}, this proves the first half of \textit{(b)} and \textit{(c)}.
	
	Let us now establish the bandwidth approximations in \textit{(b)} and \textit{(c)}. Fix $\xi\in[0,1]$. If $\xi=0$, the results are trivial, and so we may assume $\xi\in(0,1]$. Let $1\leq k\leq m$ be the unique integer such that $\xi\in\left(\frac{k-1}{m},\frac{k}{m}\right]$. Then we may write
	$$
	\begin{aligned}
	W(\xi)&=w_{z_{m}}^{1}(0)\cdots w_{z_{k+1}}^{1}(0) w_{z_{k}}^{1}(\delta_{m}^{k}(\xi)) w_{z_{k-1}}^{1}(1)\cdots w_{z_{1}}^{1}(1)\\
	&=w_{z_{k}}^{1}(\delta_{m}^{k}(\xi)) w_{z_{k-1}}^{1}(1)\cdots w_{z_{1}}^{1}(1).
	\end{aligned}
	$$
	Suppose first that $T\in M_{n}$ has zero crosses at indices $z_{1},\ldots,z_{m}$. Then by \autoref{matrix1}, $\mathfrak{r}(w_{z_{1}}^{1}(1)Tw_{z_{1}}^{1}(1)^{*})\leq \mathfrak{r}(T)$ since $T$ has a zero cross at index $z_{1}$. Moreover, $z_{1}-1$ applications of part \textit{(a)} of \autoref{untouched crosses} shows that $w_{z_{1}}^{1}(1)Tw_{z_{1}}^{1}(1)^{*}$ has a zero cross at indices $z_{2},\ldots,z_{m}$, since $z_{2},\ldots,z_{m}$ are not among the indices affected by the conjugation. Hence, we may apply \autoref{matrix1} again to conclude that 
	$$
	\mathfrak{r}(w_{z_{2}}^{1}(1)w_{z_{1}}^{1}(1)Tw_{z_{1}}^{1}(1)^{*}w_{z_{2}}^{1}(1)^{*})\leq\mathfrak{r}(w_{z_{1}}^{1}(1)Tw_{z_{1}}^{1}(1)^{*})\leq \mathfrak{r}(T).
	$$
	Continuing inductively in this way, it follows that $\mathfrak{r}(D)\leq\mathfrak{r}(T)$, where 
	$$
	D=w_{z_{k-1}}^{1}(1)\cdots w_{z_{1}}^{1}(1)Tw_{z_{1}}^{1}(1)^{*}\cdots w_{z_{k-1}}^{1}(1)^{*}
	$$
	and, moreover, $D$ has a zero cross at indices $z_{k},\ldots,z_{m}$. Thus, by \autoref{matrix2},
	$$
	\mathfrak{r}(W(\xi)TW(\xi)^{*})=\mathfrak{r}(w_{z_{k}}^{1}(\delta_{m}^{k}(\xi))Dw_{z_{k}}^{1}(\delta_{m}^{k}(\xi))^{*})\leq\mathfrak{r}(D)+2\leq \mathfrak{r}(T)+2,
	$$ 
	which yields the approximation in \textit{(b)}.
	
	To complete the proof of \textit{(c)} suppose $T\in M_{n\times b}$ and that the rows of $T$ at indices $z_{1},\ldots,z_{m}$ consist entirely of zeros. Following the lines of the proof of \autoref{matrix1}, we have
	$$
	\mathfrak{r}\left(
	\begin{pmatrix}
	0_{n\times n} & w_{z_{k-1}}^{1}(1)\cdots w_{z_{1}}^{1}(1)T\\
	0_{b\times n} & 0_{b\times b}
	\end{pmatrix}
	\right)\leq\mathfrak{r}\left(
	\begin{pmatrix}
	0_{n\times n} & T\\
	0_{b\times n} & 0_{b\times b}
	\end{pmatrix}\right)
	$$
	by \autoref{permutation} since only rows of zeros are shifted up when multiplying $T$ on the left by $w_{z_{k-1}}^{1}(1)\cdots w_{z_{1}}^{1}(1)$, while non-zero entries remain in place or are shifted down towards the diagonal. Similar reasoning to that used when deducing \autoref{wij} shows that there exist $\beta\in[0,1]$ and $1\leq p\leq z_{k}-1$ such that
	$$
	w_{z_{k}}^{1}(\delta_{m}^{k}(\xi))=u_{(z_{k} - p \ z_{k}-p+1)}(\beta) w_{z_{k}}^{z_{k}-p+1}(1).
	$$
	Since the $z_{k}$th row of a given matrix remains unchanged when multiplying on the left by $w_{z_{k-1}}^{1}(1)\cdots w_{z_{1}}^{1}(1)$, the $z_{k}$th row of $w_{z_{k-1}}^{1}(1)\cdots w_{z_{1}}^{1}(1)T$ contains only zeros. Hence,  multiplying this given matrix on the left by $w_{z_{k}}^{z_{k}-p+1}(1)=U[(z_{k}-p+1 \ \cdots \ z_{k})]$ shifts the zero row from index $z_{k}$ to index $z_{k}-p+1$, while shifting the rows $z_{k}-p+1,\ldots,z_{k}-1$ down by one towards the diagonal. Thus, 
	$$
	\mathfrak{r}\left(
	\begin{pmatrix}
	0_{n\times n} & E\\
	0_{b\times n} & 0_{b\times b}
	\end{pmatrix}
	\right)\leq
	\mathfrak{r}\left(
	\begin{pmatrix}
	0_{n\times n} & w_{z_{k-1}}^{1}(1)\cdots w_{z_{1}}^{1}(1)T\\
	0_{b\times n} & 0_{b\times b}
	\end{pmatrix}
	\right)\leq\mathfrak{r}\left(
	\begin{pmatrix}
	0_{n\times n} & T\\
	0_{b\times n} & 0_{b\times b}
	\end{pmatrix}\right),
	$$
	where $E:=U[(z_{k}-p+1 \ \cdots \ z_{k})]w_{z_{k-1}}^{1}(1)\cdots w_{z_{1}}^{1}(1)T$. Now, $E$ and $u_{(z_{k}-p \ z_{k}-p+1)}(\beta)E$ may differ only on rows $z_{k}-p$ and $z_{k}-p+1$, where these two rows of the latter matrix are linear combinations of the same two rows of $E$. From this and the fact that the $z_{k}-p+1$ row of $E$ consists only of zeros, it is clear that for a given column $\lambda$, the $(z_{k}-p,\lambda)$- or $(z_{k}-p+1,\lambda)$-entry of $u_{(z_{k}-p \ z_{k}-p+1)}(\beta)E$ can be non-zero only if the $(z_{k}-p,\lambda)$ entry of $E$ is non-zero. Hence, 
	$$
	\mathfrak{r}\left(
	\begin{pmatrix}
	0_{n\times n} & u_{(z_{k}-p \ z_{k}-p+1)}(\beta)E\\
	0_{b\times n} & 0_{b\times b}
	\end{pmatrix}
	\right)\leq
	\mathfrak{r}\left(
	\begin{pmatrix}
	0_{n\times n} & E\\
	0_{b\times n} & 0_{b\times b}
	\end{pmatrix}
	\right)\leq
	\mathfrak{r}\left(
	\begin{pmatrix}
	0_{n\times n} & T\\
	0_{b\times n} & 0_{b\times b}
	\end{pmatrix}\right).
	$$
	Since $W(\xi)T=u_{(z_{k}-p \ z_{k}-p+1)}(\beta)E$, this establishes the bandwidth approximation in \textit{(c)}, thus completing the proof of the lemma. 
\end{proof}

\begin{definition}
	\label{transposition prod}
	Let $N\in\NN$. For $j,k,n\in\NN$ satisfying $N\leq j\leq k\leq n$, we define
	$$
	\sigma^{n}_{j,k}:=(j-N+1 \ \ k-N+1)\cdots(j \ k)\in S_{n}.
	$$
	We define $u_{j,k}^{n}\colon[0,1]\to M_{n}$ to be the unitary
	$$
	u^{n}_{j,k}(\xi):=u_{(j-N+1  \ k-N+1)}(\xi)\cdots u_{(j \ k)}(\xi),
	$$
	where $u_{(i \ i')}\colon [0,1]\to M_{n}$ is a continuous path of unitaries defined as in \autoref{unitaries}. 
\end{definition}

\begin{remark}
	\label{commuting perms}
	Note that in the definition above, if $j\leq n-N$, then $\sigma^{n}_{j,n}$ is the permutation in $S_{n}$ that interchanges $j-N+1,\ldots,j$ and $n-N+1,\ldots,n$; moreover, in this case, all of the factors in the definition of $u^{n}_{j,n}(\xi)$ (for any $\xi\in[0,1]$) commute with each other by \autoref{unitaries}. 
\end{remark}

\begin{lemma}
	\label{factormain}
	Suppose $N,n,k,i\in\NN$ satisfy $N\leq k\leq i-N$ and $i\leq n-N$, and that $\xi\in[0,1]$. Then,
	$$
	u^{n}_{k,n}(\xi)=U[\sigma^{n}_{i,n}]u^{n}_{k,i}(\xi)U[\sigma^{n}_{i,n}],
	$$
	where $u^{n}_{k,n}(\xi)$, $u^{n}_{k,i}(\xi)$, and $\sigma^{n}_{i,n}$ are each products of $N$ factors as defined in \autoref{transposition prod}.
\end{lemma}

\begin{proof}
	By definition,
	\begin{equation}
	\label{negprod}
	U[\sigma^{n}_{i,n}]u^{n}_{k,i}(\xi)
	=\left(\prod_{j=-(N-1)}^{0}U[(i+j \ \ n+j)]\right)\left(\prod_{j=-(N-1)}^{0}u_{(k+j \  i+j)}(\xi)\right).
	\end{equation}
	Note that for any $-(N-1)\leq j,j'\leq 0$, 
	$$
	i+j\leq i<n-(N-1)\leq n+j'
	$$
	and 
	$$
	k+j\leq i-N+j\leq i-N<i-(N-1)\leq i+j'.
	$$
	Thus, whenever $-(N-1)\leq j,j'\leq 0$ with $j\not=j'$, the permutations $(i +j' \ \ n+j')$ and $(k+j \ \ i+j)$ are disjoint, and hence, $U[(i +j' \ \ n+j')]$ and $u_{(k+j \ \  i+j)}(\xi)$ commute. Hence, \autoref{negprod} can be restated as
	$$
	U[\sigma^{n}_{i,n}]u^{n}_{k,i}(\xi)=\prod_{j=-(N-1)}^{0}U[(i+j \ \ n+j)]u_{(k+j \  i+j)}(\xi).
	$$
	By the same reasoning,
	$$
	\begin{aligned}
	U[\sigma^{n}_{i,n}]u^{n}_{k,i}(\xi)U[\sigma^{n}_{i,n}]
	&=\prod_{j=-(N-1)}^{0}U[(i+j \ \ n+j)]u_{(k+j \  i+j)}(\xi)\prod_{j=-(N-1)}^{0}U[(i+j \ \ n+j)]\\
	&=\prod_{j=-(N-1)}^{0}U[(i+j \ \ n+j)]u_{(k+j \  i+j)}(\xi)U[(i+j \ \ n+j)].
	\end{aligned}
	$$
	It is elementary to see that $U[(a \ b)]u_{(c \ b)}(\zeta)U[(a \ b)]=u_{(c \ a)}(\zeta)$ whenever $c\leq b\leq a$ and $\zeta\in[0,1]$. Hence,
	$$
	U[\sigma^{n}_{i,n}]u^{n}_{k,i}(\xi)U[\sigma^{n}_{i,n}]
	=\prod_{j=-(N-1)}^{0}u_{(k+j \ n+j)}(\xi)=u^{n}_{k,n}(\xi),
	$$
	which proves the lemma.
\end{proof}

\begin{definition}
	For integers $1\leq j\leq k\leq n$, define 
	$$
	\gamma^{n}_{j,k}:=(j \ \ j+1 \ \cdots \ k)\in S_{n}.
	$$
\end{definition}

\begin{lemma}
	\label{calculation}
	Let $N,n,t\in\NN$ with $1\leq N<t\leq n-N$, and let $(\xi_{N+1},\ldots,\xi_{t-1})$ be a vector in $[0,1]^{t-1-N}$ whose final $N-1$ entries consist only of zeros. Then,	
	$$
	\begin{aligned}
		U[\gamma^{n}_{1,n}]^{N}\left(\prod_{k=N}^{t-2}u^{n}_{k,n}(\xi_{k+1})\right)U[\sigma^{n}_{t-1,n}]
		=U[\gamma^{n}_{1,t-1}]^{N}\left(\prod_{k=N}^{t-2}u^{n}_{k,t-1}(\xi_{k+1})\right)U[\gamma^{n}_{t,n}]^{N},
	\end{aligned}
	$$
	where $u_{k,n}^{n}(\xi_{k+1})$, $u_{k,t-1}^{n}(\xi_{k+1})$, and $\sigma_{t-1,n}^{n}$ are each products of $N$ factors as defined in \autoref{transposition prod}.
\end{lemma}	
\begin{proof}
	If $t-1=N$, then the products on either side of the equality above are empty and the equation reduces to
	\begin{equation}
		\label{elementary}
		U[\gamma^{n}_{1,n}]^{N}U[\sigma^{n}_{t-1,n}]=U[\gamma^{n}_{1,t-1}]^{N}U[\gamma^{n}_{t,n}]^{N}.
	\end{equation}
	By \autoref{commuting perms}, it is elementary to see that \autoref{elementary} holds. Therefore, for the remainder of the proof, we may assume that $t-2\geq N$. 
	
	For any $N\leq k\leq t-1-N$, we may apply \autoref{factormain} (recalling that $t-1\leq n-N$) to conclude that
	\begin{equation}
		\label{factor}
		u^{n}_{k,n}(\xi_{k+1})=U[\sigma^{n}_{t-1,n}]u^{n}_{k,t-1}(\xi_{k+1})U[\sigma^{n}_{t-1,n}].
	\end{equation}
	
	In fact, \autoref{factor} holds for all $N\leq k\leq t-2$. Indeed, if $t-N\leq k\leq t-2$, then by the assumption of the lemma, it must be that $\xi_{k+1}=0$. In this case, \autoref{factor} reduces to $1_{n}=U[\sigma^{n}_{t-1,n}]^{2}$, which holds by \autoref{commuting perms}. 
	
	Therefore,  
	$$
	\begin{aligned}
		\prod_{k=N}^{t-2}u^{n}_{k,n}(\xi_{k+1})&=\prod_{k=N}^{t-2}U[\sigma^{n}_{t-1,n}]u^{n}_{k,t-1}(\xi_{k+1})U[\sigma^{n}_{t-1,n}]\\
		&=U[\sigma^{n}_{t-1,n}]\left(\prod_{k=N}^{t-2}u^{n}_{k,t-1}(\xi_{k+1})\right)U[\sigma^{n}_{t-1,n}],
	\end{aligned}
	$$
	which, together with \autoref{elementary}, yields that
	\begin{equation}
		\label{penultimate}
		\begin{aligned}
			U[\gamma^{n}_{1,n}]^{N}\prod_{k=N}^{t-2}u^{n}_{k,n}(\xi_{k+1})
			=U[\gamma^{n}_{1,t-1}]^{N}U[\gamma^{n}_{t,n}]^{N}\left(\prod_{k=N}^{t-2}u^{n}_{k,t-1}(\xi_{k+1})\right)U[\sigma^{n}_{t-1,n}].
		\end{aligned}
	\end{equation}
	Moreover, for each $k=N,\ldots,t-2$, the indices in each transposition-like unitary factor in $u^{n}_{k,t-1}(\xi_{k+1})$ are distinct from $t,\ldots,n$. Hence,  $U[\gamma^{n}_{t,n}]^{N}$ and  $\prod_{k=N}^{t-2}u^{n}_{k,t-1}(\xi_{k+1})$ commute, and so, using \autoref{penultimate}, it follows that
	$$
	\begin{aligned}
		U[\gamma^{n}_{1,n}]^{N}\prod_{k=N}^{t-2}u^{n}_{k,n}(\xi_{k+1})
		=
		U[\gamma^{n}_{1,t-1}]^{N}\left(\prod_{k=N}^{t-2}u^{n}_{k,t-1}(\xi_{k+1})\right)U[\gamma^{n}_{t,n}]^{N}U[\sigma^{n}_{t-1,n}].
	\end{aligned}
	$$
	\autoref{calculation} then follows by multiplying $U[\sigma^{n}_{t-1,n}]$ on the right of both sides in the above expression. 
\end{proof}

\begin{definition}
	\label{ltmatrix}
	Let $N\in\NN$. Given $n\in\NN$ with $n\geq N$, let $W_{n}\in C([0,1]^{n},M_{n})$ be the unitary
	\begin{equation}
	\label{Wn}
	W_{n}(\xi_{1},\ldots,\xi_{n}):=U[\gamma^{n}_{1,n}]^{N}\left(\prod_{k=N}^{n-1}u^{n}_{k,n}(\xi_{k+1})\right),
	\end{equation}
	where $u_{k,n}^{n}$ is the product of $N$ factors as in \autoref{transposition prod}.
	We adopt the convention that $W_{n}:=1_{n}$ if $n=N$.
\end{definition}

\begin{lemma}
	\label{Wn reduction}
	Let $N\in\NN$. Suppose $n$ is an integer greater than $N$ and $\vec{\xi}:=(\xi_{1},\ldots,\xi_{n})$ is a vector in $[0,1]^{n}$ with the property that $\xi_{1}=1$, the final $N$ entries are all zero, and for any consecutive $N$ entries, at most one is non-zero. Suppose $K=\{1=k_{1}<k_{2}<\cdots<k_{m}\}$ is any set of indices, containing $1$, at which $\vec{\xi}$ is $1$; put $k_{m+1}:=n+1$. Then, 
	\begin{equation}
	\label{Wn decomp}
	W_{n}(\vec{\xi})=\diag\left(W_{k_{2}-k_{1}}(\xi_{k_{1}},\ldots,\xi_{k_{2}-1}),\ldots,W_{k_{m+1}-k_{m}}(\xi_{k_{m}},\ldots,\xi_{k_{m+1}-1})\right),
	\end{equation}
	where $N$ is the fixed positive integer being used to define $W_{n},W_{k_{2}-k_{1}},\ldots,W_{k_{m+1}-k_{m}}$ in \autoref{ltmatrix}.
\end{lemma}

\begin{proof}
	Fix an integer $n>N$, a vector $\vec{\xi}$, and an associated set $K$, satisfying the hypotheses of the lemma. Let us proceed by induction on the size $m$ of $K$. If $m=1$, there is nothing to show. Fix $m\geq 2$ and suppose that \autoref{Wn reduction} holds for every natural number $n'>N$, vector $\zeta$, and associated set $K'$ of size $m-1$, provided they satisfy the required hypotheses. Assume that $|K|=m$. Let us show \autoref{Wn decomp} holds in this case.
	
	Note that by assumption $\xi_{k_{2}-(N-1)},\ldots,\xi_{k_{2}-1}=0$ and
	\begin{equation}
	\label{k2upperbound}
	N<k_{2}\leq n-N.
	\end{equation}
	Therefore, we may apply \autoref{calculation} with $N$, $n$, $k_{2}$, and $(\xi_{N+1},\ldots,\xi_{k_{2}-1})$ to conclude that
	\begin{equation}
	\label{calculation2}
		\begin{aligned}
			U[\gamma^{n}_{1,n}]^{N}\left(\prod_{k=N}^{k_{2}-2}u^{n}_{k,n}(\xi_{k+1})\right)U[\sigma^{n}_{k_{2}-1,n}]
			=U[\gamma^{n}_{1,k_{2}-1}]^{N}\left(\prod_{k=N}^{k_{2}-2}u^{n}_{k,k_{2}-1}(\xi_{k+1})\right)U[\gamma^{n}_{k_{2},n}]^{N}.
		\end{aligned}
	\end{equation}
	
	By \autoref{Wn} and \autoref{k2upperbound},  
	$$
	\begin{aligned}
	W_{n}(\vec{\xi})&=U[\gamma^{n}_{1,n}]^{N}\left(\prod_{k=N}^{n-1}u^{n}_{k,n}(\xi_{k+1})\right)\\
	&=U[\gamma^{n}_{1,n}]^{N}\left(\prod_{k=N}^{k_{2}-2}u^{n}_{k,n}(\xi_{k+1})\right)u^{n}_{k_{2}-1,n}(\xi_{k_{2}})\left(\prod_{k=k_{2}}^{n-1}u^{n}_{k,n}(\xi_{k+1})\right).
	\end{aligned}
	$$
	Since $u^{n}_{k_{2}-1,n}(\xi_{k_{2}})=u^{n}_{k_{2}-1,n}(1)=U[\sigma^{n}_{k_{2}-1,n}]$, we  may apply \autoref{calculation2} to obtain
	\begin{equation}
	\label{induction}
	W_{n}(\vec{\xi})=U[\gamma^{n}_{1,k_{2}-1}]^{N}\left(\prod_{k=N}^{k_{2}-2}u^{n}_{k,k_{2}-1}(\xi_{k+1})\right)U[\gamma^{n}_{k_{2},n}]^{N}\left(\prod_{k=k_{2}}^{n-1}u^{n}_{k,n}(\xi_{k+1})\right).
	\end{equation}
	
	Let $\vec{\xi'}:=(\xi_{1},\ldots,\xi_{k_{2}-1})$ and $\vec{\xi''}:=(\xi_{k_{2}},\ldots,\xi_{n})$. By \autoref{k2upperbound}, $|\vec{\xi'}|\geq N$, so that
	\begin{equation}
	\label{induction1}
	W_{k_{2}-k_{1}}(\vec{\xi'})=U[\gamma^{k_{2}-1}_{1,k_{2}-1}]^{N}\left(\prod_{k=N}^{(k_{2}-1)-1}u^{k_{2}-1}_{k,k_{2}-1}(\xi_{k+1})\right)\in M_{k_{2}-k_{1}}.
	\end{equation}
	Furthermore, $|\vec{\xi''}|>N$, so that
	\begin{equation}
	\label{induction2}
	\begin{aligned}
	W_{n+1-k_{2}}(\vec{\xi''})&=W_{n+1-k_{2}}(\xi_{k_{2}},\ldots,\xi_{n})\\
	&=U[\gamma^{n+1-k_{2}}_{1,n+1-k_{2}}]^{N}\left(\prod_{k=N}^{n+1-k_{2}-1}u^{n+1-k_{2}}_{k, n+1-k_{2}}(\xi_{k_{2}+k})\right)\\
	&=U[\gamma^{n+1-k_{2}}_{1,n+1-k_{2}}]^{N}\left(\prod_{k=1}^{n-k_{2}}u^{n+1-k_{2}}_{k, n+1-k_{2}}(\xi_{k_{2}+k})\right)\\
	&=U[\gamma^{n+1-k_{2}}_{1,n+1-k_{2}}]^{N}\left(\prod_{k=k_{2}}^{n-1}u^{n+1-k_{2}}_{k+1-k_{2}, n+1-k_{2}}(\xi_{k+1})\right)\in M_{n+1-k_{2}},
	\end{aligned}	
	\end{equation}
	where the penultimate equality follows since $\xi_{k_{2}}=1$ and at most one of any $N$ consecutive entries of $\vec{\xi}$ is non-zero. Therefore, 
	$$
	\begin{aligned}
	&\diag(W_{k_{2}-k_{1}}(\vec{\xi'}),W_{n+1-k_{2}}(\vec{\xi''}))\\
	=&\diag(W_{k_{2}-k_{1}}(\vec{\xi'}),1_{n+1-k_{2}})\diag(1_{k_{2}-k_{1}},W_{n+1-k_{2}}(\vec{\xi''}))\\
	=&U[\gamma^{n}_{1,k_{2}-1}]^{N}\left(\prod_{k=N}^{(k_{2}-1)-1}u^{n}_{k,k_{2}-1}(\xi_{k+1})\right) U[\gamma^{n}_{k_{2},n}]^{N}\left(\prod_{k=k_{2}}^{n-1}u^{n}_{k, n}(\xi_{k+1})\right),
	\end{aligned}
	$$
	where in the last equality the indices in the $\gamma$'s and $u$'s have been altered appropriately from the ones in \autoref{induction1} and \autoref{induction2} to accommodate for the identity factors in the diagonal. Combining this with \autoref{induction} yields that
	\begin{equation}
	\label{breakup}
	W_{n}(\vec{\xi})=\diag(W_{k_{2}-k_{1}}(\vec{\xi'}),W_{n+1-k_{2}}(\vec{\xi''})).
	\end{equation}
	We may apply the inductive hypothesis to $n'=|\vec{\xi''}|>N$, vector $\vec{\xi''}$, and associated set $K'=\{k_{2},\ldots,k_{m}\}$ of size $m-1$ to conclude that
	$$
	\begin{aligned}
	W_{n+1-k_{2}}(\vec{\xi''})
	=\diag\left(W_{k_{3}-k_{2}}(\xi_{k_{2}},\ldots,\xi_{k_{3}-1}),\ldots,W_{k_{m+1}-k_{m}}(\xi_{k_{m}},\ldots,\xi_{k_{m+1}-1})\right).
	\end{aligned}
	$$
	Substituting this into \autoref{breakup} yields that
	$$
	W_{n}(\vec{\xi})=\diag\left(W_{k_{2}-k_{1}}(\xi_{k_{1}},\ldots,\xi_{k_{2}-1}),\ldots,W_{k_{m+1}-k_{m}}(\xi_{k_{m}},\ldots,\xi_{k_{m+1}-1})\right),
	$$
	which proves \autoref{Wn reduction}. 
\end{proof}

\subsection{The Main Lemmas}
\label{sc: main lemmas}

With the results of \cref{sc: prelim lemmas} in hand, we are now in position to prove the lemmas listed in \autoref{chart}, which are needed to prove \autoref{main} in the sequel.

We start with a lemma that characterizes when a unital injective limit of subhomogeneous algebras is simple in terms of the corresponding maps between their spectra. This is essentially Proposition 2.1 of \cite{4Rom}, except that ours discusses the general unital subhomogeneous case. The proof is very similar.

Given unital subhomogeneous $\mathrm{C}^*$-algebras $A$ and $B$ and a unital $^{*}$-homomorphism $\psi\colon A\to B$, an irreducible representation $\pi$ of $B$ yields a representation $\pi\circ\psi$ of $A$. The finite-dimensional representation $\pi\circ\psi$ is unitarily equivalent to a direct sum $\tau_{1}\oplus\cdots\oplus\tau_{s}$ of irreducible representations of $A$. In this way, we get a map $\hat{\psi}\colon\hat{B}\to\mathcal{P}(\hat{A})$ given by $\hat{\psi}([\pi]):=\{[\tau_{1}],\ldots,[\tau_{s}]\}$, where multiplicities are ignored. 

\begin{lemma}
	\label{simplicity}
	Suppose we have an inductive limit of the form 
	$$
	A_{1}\overset{\psi_{1}}{\longrightarrow} A_{2}\overset{\psi_{2}}{\longrightarrow} A_{3}\overset{\psi_{3}}{\longrightarrow}\cdots\longrightarrow A:=\varinjlim A_{i},
	$$
	where $A$ is unital and for each $i\in \NN$, $A_{i}$ is subhomogeneous and $\psi_{i}$ is injective. Let $\psi_{j,i}:=\psi_{j-1}\circ\cdots\circ\psi_{i}$. Then, the following statements are equivalent.
	\begin{enumerate}
		\item $A$ is simple.
		\item For all $i\in\NN$ and all non-empty open $U\subset \hat{A_{i}}$, there is a  $j>i$ such that $\hat{\psi}_{j,i}([\pi])\cap U\not=\varnothing$ for all $[\pi]\in \hat{A_{j}}$.
		\item For all $i\in\NN$, if $f\in A_{i}$ is non-zero, then there is a $j > i$ such that $\pi(\psi_{j,i}(f))\not=0$ for every non-zero irreducible representation $\pi$ of $A_{j}$. 
	\end{enumerate} 
\end{lemma}

\begin{proof}
	For $n\in\NN$, let $\mu_{n}\colon A_{n}\to A$ denote the map in the construction of the inductive limit. Since the $\psi_{j}$'s are injective and $A$ is unital, we may assume that the $A_{j}$'s are all unital and that the $\mu_{j}$'s are injective and unit-preserving.
	
	Let us start by showing that \textit{(1)} implies \textit{(2)}. Suppose that \textit{(2)} is false. To show \textit{(1)} is false, let us construct a closed proper non-zero two-sided ideal of $A$. Choose $i\in\mathbb{N}$ and a non-empty open set $U\subset \hat{A_{i}}$ such that for all $j > i$ there is a $[\pi]\in\hat{A_{j}}$ with $\hat{\psi}_{j,i}([\pi])\cap U=\varnothing$. We may assume that $U\not=\hat{A_{i}}$. For $j > i$, set $F_{j}:=\{[\pi]\in \hat{A_{j}}:\hat{\psi}_{j,i}([\pi])\cap U=\varnothing\}$ and set $I_{j}:=\{f\in A_{j}:f\in\bigcap_{[\pi]\in F_{j}}\ker\pi\}$. It is straightforward to verify that for all $j>i$, $I_{j}$ is a closed proper non-zero two-sided ideal of $A_{j}$.
	For $k>j>i$ and $[\pi]\in \hat{A}_{k}$, we have $\hat{\psi}_{k,i}([\pi])=\hat{\psi}_{j,i}(\hat{\psi}_{k,j}([\pi]))$, from which it follows that $\hat{\psi}_{k,j}(F_{k})\subset F_{j}$. Thus, $\psi_{k,j}(I_{j})\subset I_{k}$ for all $k>j>i$. Hence, $\{\mu_{j}(I_{j})\}_{j>i}$ is an increasing sequence of $\mathrm{C}^*$-algebras, and so $I:=\overline{\bigcup_{j>i}\mu_{j}(I_{j})}$ is a sub-$\mathrm{C}^*$-algebra of $A$. It is not hard to see that $I$ is a closed two-sided ideal of $A$. Since the $\mu_{j}$'s are injective and the $I_{j}$'s are non-zero, $I\not=\{0\}$. If $1_{A}\in I$, then for large enough $j$, $I_{j}$ contains $1_{A_{j}}$, contradicting that $I_{j}$ is proper. Hence, $I$ is the desired closed proper non-zero two-sided ideal of $A$. This proves that \textit{(1)} implies \textit{(2)}.
	
	Let us now show that \textit{(2)} implies \textit{(3)}. Fix $i\in\NN$ and suppose $0\not=f\in A_{i}$. Let $U:=\{[\rho]\in\hat{A_{i}}:\rho(f)\not=0\}$. $U$ is a non-empty open subset of $\hat{A_{i}}$. By \textit{(2)}, there is a $j>i$ such that $\hat{\psi}_{j,i}([\pi])\cap U\not=\varnothing$ for all $[\pi]\in \hat{A_{j}}$. Thus, if $\pi$ is any irreducible representation of $A_{j}$, $\pi(\psi_{j,i}(f))\not=0$, which proves \textit{(3)}.
	
	Finally, let us prove that \textit{(3)} implies \textit{(1)}. Suppose $J$ is a non-zero closed two-sided ideal of $A$. For $j\in\NN$, put $J_{j}:=\mu_{j}^{-1}(J)$. Then for all $j\in\NN$, $J_{j}$ is a closed two-sided ideal of $A_{j}$. It will be shown that $J_{j}=A_{j}$ for some $j\in\NN$.	Take $0\not=a\in J$. It is well known that $J=\overline{\bigcup_{j=1}^{\infty}(\mu_{j}(A_{j})\cap J)}$. Hence, there must be an $i$ and an $a_{i}\in A_{i}$ such that $0\not=\mu_{i}(a_{i})\in J$. Thus, $a_{i}\not=0$. By \textit{(3)}, there is a $j>i$ such that for all irreducible representations $\pi$ of $A_{j}$, $\pi(\psi_{j,i}(a_{i}))\not=0$. Since $\mu_{j}(\psi_{j,i}(a_{i}))=\mu_{i}(a_{i})\in J$, it follows that $\psi_{j,i}(a_{i})\in J_{j}$. The bijective correspondence between closed two-sided ideals of $A_{j}$ and closed subsets of $\hat{A_{j}}$ thus forces $J_{j}$ to be all of $A_{j}$. Hence, $1_{A}=\mu_{j}(1_{A_{j}})\in J$, which shows that $J=A$. Therefore, $A$ is simple, which proves \textit{(1)}.
\end{proof}

\begin{lemma}
	\label{perturb}
	Let $A$ be a DSH algebra of length $ l $. Let $\epsilon >0$. Suppose that $f\in A$ is not invertible. Then, there is an $f'\in A$ with $\|f-f'\|\leq\epsilon$ and there are unitaries $w,v\in A$ such that for some $1\leq i\leq  l $, $(wf'v)_{i}$ has a zero cross at index $1$ everywhere on some non-empty set $U\subset \hat{A}\cap (X_{i}\setminus Y_{i})$, which is open with respect to the hull-kernel topology on $\hat{A}$. Moreover, there is a $\Delta\in A$ such that for every $1\leq j\leq l $ and $x\in X_{j}$, $\Delta_{j}(x)$ is a diagonal matrix with entries in $[0,1]$, where $\Delta_{j}(x)_{k,k}>0$ implies $(wf'v)_{j}(x)$ has a zero cross at index $k$; moreover, $\Delta_{i}(z)_{1,1}=1$ for all $z\in U$. 
\end{lemma}

\begin{proof}
	Using \autoref{non-invertible at point}, choose $1\leq i\leq l $ and $x\in X_{i}\setminus Y_{i}$ such that $f_{i}(x)$ is a non-invertible matrix. We break the proof up into two cases.
	
	Case one: $x$ is not in the decomposition of any point in $Y_{j}$ for any $j>i$. By \autoref{not decomposition}, there is set $U_{1}\subset X_{i}$ containing $x$, which is open in $X_{i}$ and has the property that no point in it is in the decomposition of any point in $Y_{j}$ for any $j>i$. Since $Y_{i}$ is closed in $X_{i}$, the set $U_{1}\cap(X_{i}\setminus Y_{i})$ is open in $X_{i}$. By shrinking $U_{1}$, we may assume that $\|f_{i}(x)-f_{i}(z)\|\leq\epsilon$ for all $z\in U_{1}$. Choose a set $U_{2}$ that is open in $X_{i}$ and satisfies $x\in U_{2}\subset\overline{U_{2}}^{X_{i}}\subset U_{1}\cap(X_{i}\setminus Y_{i})$. Using Urysohn's Lemma, we can define a function $h\in C(X_{i},M_{n_{i}})$ such that $h|_{\overline{U_{2}}^{X_{i}}}\equiv f_{i}(x)$,  $h|_{X_{i}\setminus(U_{1}\cap(X_{i}\setminus Y_{i}))}=f_{i}|_{X_{i}\setminus(U_{1}\cap(X_{i}\setminus Y_{i}))}$, and $\|f_{i}-h\|\leq\epsilon$. Define $f'$ coordinate-wise by $f':=(f_{1},\ldots,f_{i-1},h,f_{i+1},\ldots,f_{ l })$. Since $h|_{Y_{i}}=f_{i}|_{Y_{i}}$, we have $(f_{1},\ldots,f_{i-1},h)\in A^{(i)}$.  Since no point in $U_{1}$ is in the decomposition of any point in $Y_{j}$ for any $j>i$, and because $h$ may only differ from $f_{i}$ on $U_{1}\cap (X_{i}\setminus Y_{i})\subset U_{1}$, this perturbation does not violate the diagonal decomposition at any point. Thus, $f'\in A$ since $f\in A$, and $\|f-f'\|\leq\epsilon$ because $\|f_{i}-h\|\leq\epsilon$. Since $f_{i}(x)$ is a non-invertible matrix, there are unitary matrices $W$ and $V$ in $M_{n_{i}}$ with the property that $Wf_{i}(x)V$ has a zero cross at index $1$. Since the unitary group in $M_{n_{i}}$ is connected we may, using the same reasoning as above, define unitaries $w,v\in A$ coordinate-wise with $w_{j}=v_{j}\equiv 1_{n_{j}}$ for all $j\not=i$ and $w_{i},v_{i}\in C(X_{i},M_{n_{i}})$ satisfying $w_{i}|_{\overline{U_{2}}^{X_{i}}}\equiv W$, $v_{i}|_{\overline{U_{2}}^{X_{i}}}\equiv V$, and $w_{i}|_{X_{i}\setminus(U_{1}\cap(X_{i}\setminus Y_{i}))}=v_{i}|_{X_{i}\setminus(U_{1}\cap(X_{i}\setminus Y_{i}))}\equiv 1_{n_{i}}$. Finally, choose a set $U_{3}$ that is open in $X_{i}$ and satisfies $x\in U_{3}\subset\overline{U_{3}}^{X_{i}}\subset U_{2}$. Define $\Delta\in A$ coordinate-wise as follows: $\Delta_{j}\equiv 0$ for $j\not=i$; let $g\colon X_{i}\to [0,1]$ be any continuous function such that $g|_{\overline{U_{3}}^{X_{i}}}\equiv 1$ and $g|_{X_{i}\setminus U_{2}}\equiv 0$, and put $\Delta_{i}:=\diag(g,0,\ldots,0)\in C(X_{i},M_{n_{i}})$. As argued above for $f'$, we have $\Delta\in A$. Take $U:=U_{3}$. Applying \autoref{open in spectrum}, we conclude $U$ is open in $\hat{A}$. Since $(wf'v)_{i}$ has a zero cross at index $1$ everywhere on $U_{2}$ and since $\Delta$ vanishes outside $U_{2}$, the lemma holds in this case. 
	
	Case two: There is a $j>i$ such that $x$ is in the decomposition of some point in $Y_{j}$. In this case, we cannot define $f'$ as above, because we are not guaranteed a neighbourhood around $x$ in which we may freely perturb $f$ while remaining in $A$. Let $i'$ denote the largest integer for which $x$ is in the decomposition of some point in $Y_{i'}$. Choose $y\in Y_{i'}$ such that $x$ is in the decomposition of $y$. Then $f_{i'}(y)$ is a non-invertible matrix. Since $x$ is not in the decomposition of any point in $Y_{j'}$ for any $j'>i'$, neither is $y$. Hence, by \autoref{not decomposition}, there is a set $U_{1}\subset X_{i'}$ containing $y$ that is open in $X_{i'}$ with the property that no point in $U_{1}$ is in the decomposition of any point in $Y_{j'}$ for any $j'>i'$. Hence, as in case one, we are able to perturb $f$ on $U_{1}\cap (X_{i'}\setminus Y_{i'})$, while remaining in $A$.  By shrinking $U_{1}$, we may assume that $\|f_{i'}(y)-f_{i'}(z)\|\leq \epsilon$ for all $z\in U_{1}$. By \autoref{empty interior}, we may assume that $Y_{i'}$ has empty interior and, thus, that there is a point $x'\in U_{1}\cap (X_{i'}\setminus Y_{i'})$. Choose a set $U_{2}$ which is open in $X_{i'}$ and satisfies $x'\in U_{2}\subset\overline{U_{2}}^{X_{i'}}\subset U_{1}\cap(X_{i'}\setminus Y_{i'})$. As in case one, we may define $f'\in A$ with $\|f-f'\|\leq \epsilon$, $f_{j'}'=f_{j'}$ for $j'\not=i'$, $f_{i'}'|_{\overline{U_{2}}^{X_{i'}}}\equiv f_{i'}(y)$, and $f_{i'}'|_{X_{i'}\setminus(U_{1}\cap (X_{i'}\setminus Y_{i'}))}=f_{i'}|_{X_{i'}\setminus(U_{1}\cap (X_{i'}\setminus Y_{i'}))}$. Choose unitary matrices $W,V\in M_{n_{i'}}$ such that $Wf_{i'}(y)V$ has a zero cross at index $1$. Then the rest of the proof proceeds verbatim as the proof of case one with $i'$ in place of $i$ and $x'$ in place of $x$. 
\end{proof}

The following two lemmas guarantee the existence of certain indicator-function-like elements in DSH algebras. As outlined in \cref{sc: outline of proof}, these unitaries are used, together with the results from \cref{sc: prelim lemmas}, in the proofs of future lemmas to construct the unitaries needed to prove \autoref{main}. It is for these two key lemma that we require the base spaces of a given DSH algebra to be metrizable.

\begin{lemma}
	\label{indicatorprep}
	Suppose $A$ is a DSH algebra of length $ l $. Suppose $M\in\NN$ and $K:=\{K_{1}< K_{2}<\cdots<K_{m}\}$ are such that $K_{1}\geq 0$, $K_{m}<\mathfrak{s}(A)-M$, and $K_{t+1}-K_{t}\geq M$ for $1\leq t < m$. Then, there is a function $\Phi\in A$ such that:
	\begin{enumerate}
		\item for all $1\leq i\leq  l $ and $x\in X_{i}$, $\Phi_{i}(x)$ is a diagonal matrix with entries in $[0,1]$ whose final $M$ diagonal entries are all $0$, and such that at most one of every $M$ consecutive diagonal entries is non-zero;
		\item for all $1\leq i\leq \ell$ and $1\leq j\leq n_{i}$: $\Phi_{i}(x)_{j,j}=1$ if and only if there is a $1\leq t\leq m$ such that $x\in B_{i,j-K_{t}}$. 
	\end{enumerate}
\end{lemma}

\begin{proof}
	We define $\Phi$ coordinate-wise inductively. To start, put $\Phi_{1}\equiv\diag(\chi_{K}(0),\ldots, \chi_{K}(n_{1}-1))$, where $\chi_{K}$ is the indicator function corresponding to the set $K=\{K_{1},\ldots,K_{m}\}$. By the assumption on the set $K$, condition \textit{(1)} holds for $\Phi_{1}$. To see that \textit{(2)} holds, suppose $\Phi_{1}(x)_{j,j}=1$. Then $\chi_{K}(j-1)=1$, so there is a $1\leq t\leq m$ such that $j=K_{t}+1$. By \autoref{Bik},  $x\in X_{1}= B_{1,1}=B_{1,j-K_{t}}$. Conversely, if there is a $1\leq t\leq m$ such that $x\in B_{1,j-K_{t}}$, then by \autoref{Bik}, $j-K_{t}=1$, so that $\Phi_{1}(x)_{j,j}=\chi_{K}(j-1)=1$, which proves \textit{(2)}. 
	
	Now suppose we have a fixed $1<i\leq  l $ and assume we have defined $(\Phi_{1},\ldots,\Phi_{i-1})\in A^{(i-1)}$ such that for all $i'<i$ and $x\in X_{i'}$: 
	\begin{enumerate}[label=(\Roman*)]
		\item the matrix $\Phi_{i'}(x)$ satisfies the properties of conditions \textit{(1)} and \textit{(2)};
		\item $\Phi_{i'}(x)_{j,j} = \chi_{K}(j-1)$ for all $1\leq j\leq \mathfrak{s}(A)$.
	\end{enumerate}
	Let $\Phi'_{i}:=\varphi_{i-1}((\Phi_{1},\ldots,\Phi_{i-1}))\in C(Y_{i},M_{n_{i}})$. Fix $y\in Y_{i}$ and suppose $y$ decomposes into $x_{1}\in X_{i_{1}}\setminus Y_{i_{1}},\ldots,x_{r}\in X_{i_{r}}\setminus Y_{i_{r}}$. Let us first check that conditions \textit{(1)} and \textit{(2)} hold for $\Phi_{i}'(y)=\diag(\Phi_{i_{1}}(x_{1}),\ldots,\Phi_{i_{r}}(x_{r}))$. By the inductive hypothesis, $\Phi_{i}(y)$ is a diagonal matrix with entries in $[0,1]$ and the last $M$ diagonal entries of $\Phi_{i}'(y)$ are all $0$. Given $M$ consecutive entries down the diagonal of $\Phi_{i}'(y)$, if they are all contained in one of the diagonal blocks, then by the inductive hypothesis applied to that one block, at most one of these entries is non-zero. If instead the $M$ consecutive entries span two blocks $\Phi_{i_{q}}(x_{q})$ and $\Phi_{i_{q+1}}(x_{q+1})$, then by the inductive hypothesis, the last $M$ diagonal entries of $\Phi_{i_{q}}(x_{q})$ are $0$ and at most $1$ of the first $M$ diagonal entries of $\Phi_{i_{q+1}}(x_{q+1})$ can be non-zero. This shows that \textit{(1)} holds for $\Phi_{i}'(y)$. Let us now show that \textit{(2)} holds for $\Phi_{i}'(y)$. Fix $1\leq j\leq n_{i}$. Let $1\leq q\leq r$ and $1\leq j'\leq n_{i_{q}}$ be such that $\Phi_{i}'(y)_{j,j}=\Phi_{i_{q}}(x_{q})_{j',j'}$. Note that $j=n_{i_{1}}+\cdots+n_{i_{q-1}}+j'$. Given $1\leq t\leq m$, we know by \autoref{Bik} that $y\in B_{i,j-K_{t}}$ if and only if there is a $1\leq p\leq r$ such that 
	\begin{equation}
		\label{pq}
		j'-K_{t}+n_{i_{1}}+\cdots +n_{i_{q-1}}=j-K_{t}=1+n_{i_{1}}+\cdots+ n_{i_{p-1}}
	\end{equation}
	(the right-hand side is $1$ if $p=1$). We claim that if \autoref{pq} holds, then $p=q$. Indeed, using the upper and lower bounds on $j'$ and $K_{t}$, we have
	$$
	1-\mathfrak{s}(A)< 1-(\mathfrak{s}(A)-M-1)\leq j'-K_{t}\leq n_{i_{q}},
	$$
	whence
	$$
	1+n_{i_{1}}+\cdots+n_{i_{q-1}}-\mathfrak{s}(A)<1+n_{i_{1}}+\cdots +n_{i_{p-1}}\leq n_{i_{1}}+\cdots +n_{i_{q-1}}+n_{i_{q}}.
	$$
	The first inequality and the definition of $\mathfrak{s}(A)$ imply that $q\leq p$, while the second inequality forces $q\geq p$, so that $p=q$. Therefore, since $x_{q}\in X_{i_{q}}\setminus Y_{i_{q}}$, the above and \autoref{Bik} show that
	$$
	y\in B_{i,j-K_{t}}\iff j-K_{t}=1+n_{i_{1}}+\cdots+n_{i_{q-1}}\iff j'-K_{t}=1\iff x_{q}\in B_{i_{q},j'-K_{t}}.
	$$
	Since the matrix $\Phi_{i_{q}}(x_{q})$ satisfies \textit{(2)} by the inductive hypothesis, it follows that there is a $1\leq t\leq m$ with $y\in B_{i,j-K_{t}}$ if and only if $\Phi_{i}'(y)_{j,j}=\Phi_{i_{q}}(x_{q})_{j',j'}=1$, which proves that \textit{(2)} holds for $\Phi_{i}'(y)$. 
	
	Let us now define $\Phi_{i}\in C(X_{i},M_{n_{i}})$ to be a suitable extension of $\Phi_{i}'$. Write $\Phi_{i}'=\diag(h_{1}',\ldots,h_{n_{i}}')$, where $h'_{j}\in C(Y_{i},[0,1])$ for $1\leq j\leq n_{i}$. We define $\Phi_{i}=\diag(h_{1},\ldots,h_{n_{i}})$ by specifying each $h_{j}$ to be a continuous function $h_{j}\colon X_{i}\to [0,1]$ that extends $h_{j}'$. For $1\leq j\leq \mathfrak{s}(A)$, put $h_{j}\equiv \chi_{K}(j-1)$ to insure that (II) in the inductive hypothesis is verified and set $h_{j}\equiv 0$ for $n_{i}-M+1\leq j\leq n_{i}$ (since (I) and (II) hold for $\Phi_{1},\ldots,\Phi_{i-1}$, these $h_{j}$'s do indeed extend the corresponding $h_{j}'$'s). We define $h_{j}$ for $\mathfrak{s}(A)+1\leq j\leq n_{i}-M$ inductively. Fix $\mathfrak{s}(A)+1\leq j\leq n_{i}-M$ and assume we have defined $h_{1},\ldots,h_{j-1}$ so that the following property holds:
	\begin{itemize}
		\item[($\clubsuit$)] $\bigcup_{t=1}^{M-1}\overline{\supp(h_{j-t})}\subset X_{i}$ is disjoint from $\supp(h_{j}')\subset Y_{i}$.
	\end{itemize}
	Note that $\bigcup_{t=1}^{M-1}\overline{\supp(h_{\mathfrak{s}(A)+1-t})}=\varnothing$, and so ($\clubsuit$) holds for the base case $j=\mathfrak{s}(A)+1$. Since $X_{i}$ is a metric space and, hence, perfectly normal, we may use $(\clubsuit)$ to extend $h_{j}'$ to a function $f_{j}$ in $C(X_{i},[0,1])$ that vanishes on $\bigcup_{t=1}^{M-1}\overline{\supp(h_{j-t})}$ and is strictly less than $1$ on $X_{i}\setminus Y_{i}$. Define $g^{0}_{j}:=h_{j}'-\sum_{t=1}^{M-1}h_{j+t}'\in C(Y_{i})$. Then the range of $g_{j}^{0}$ is contained in $[-1,1]$ since by (I) at most one of $h_{j}',\ldots,h_{j+M-1}'$ is non-zero at any given point in $Y_{i}$. Extend $g_{j}^{0}$ to a function $g_{j}'$ in $C(X_{i},[-1,1])$. Put $g_{j}:=\max(g_{j}',0)$ and note that $g_{j}|_{Y_{i}}=h_{j}'$. Since $h_{j}'(y)=0$ for each $y\in \bigcup_{t=1}^{M-1}\supp(h_{j+t}')$, we may choose an open subset $U\supset \bigcup_{t=1}^{M-1}\supp(h_{j+t}')$ of $X_{i}$ on which $g_{j}'$ is strictly negative, so that $g_{j}|_{U}\equiv 0$. Define $h_{j}:=\min(f_{j},g_{j})\in C(X_{i},[0,1])$ and note that $h_{j}|_{Y_{i}}=h_{j}'$. Since $h_{j}|_{U}\equiv 0$, we have $\supp(h_{j})\cap U=\varnothing$, from which it follows that $\overline{\supp(h_{j})}\cap\left(\bigcup_{t=1}^{M-1}\supp(h_{j+t}')\right)=\varnothing$. This ensures that $(\clubsuit)$ holds with $j+1$ in place of $j$ and, hence, that $\Phi_{i}:=\diag(h_{1},\ldots,h_{n_{i}})$ is well defined. 
	
	To conclude the proof, let us check that $\Phi_{i}$ satisfies \textit{(1)} and \textit{(2)}. In light of the analysis above, we may restrict ourselves to the diagonal entries $\mathfrak{s}(A)+1\leq j\leq n_{i}-M$. By definition, the range of each $h_{j}$ is contained in $[0,1]$. If $h_{j}(x)>0$ for some $x\in X_{i}$, then $f_{j}(x)>0$ and, hence, by the definition of $f_{j}$, $x\notin \bigcup_{t=1}^{M-1}\overline{\supp(h_{j-t})}$. This proves that at most one of any $M$ consecutive entries down the diagonal of $\Phi_{i}(x)$ is non-zero. Hence, \textit{(1)} is established. To prove \textit{(2)}, suppose $x\in X_{i}$ satisfies $h_{j}(x)=1$. Then $f_{j}(x)=1$, which implies that $x\in Y_{i}$. Thus $h_{j}'(x)=1$ and we already established that $x\in B_{i,j-K_{t}}$ for some $t$ in this case. Conversely, suppose $x\in B_{i,j-K_{t}}$ for some $t$. If $j-K_{t}\not=1$, then by \autoref{Bik}, $x\in Y_{i}$ and we already concluded in this case that $h_{j}(x)=h_{j}'(x)=1$. If instead $j-K_{t}=1$, then it must be that $j< \mathfrak{s}(A)$ and we previously defined $h_{j}\equiv 1$ in this case. Therefore, property \textit{(2)} holds.
	
	We verified that both (I) and (II) hold for $\Phi_{i}=\diag(h_{1},\ldots,h_{ l })$, and since $\Phi_{i}|_{Y_{i}}=\Phi_{i}'=\varphi_{i-1}((\Phi_{1},\ldots,\Phi_{i-1}))$, it follows that $(\Phi_{1},\ldots,\Phi_{i})\in A^{(i)}$. Thus, by induction, we obtain $\Phi:=(\Phi_{1},\ldots,\Phi_{ l })\in A$, which satisfies the requirements of the lemma. 
\end{proof}

\begin{lemma}
	\label{indicators}
	Suppose $A$ is a DSH algebra of length $ l $. Suppose $M\in\NN$ and $K:=\{K_{1}< K_{2}<\cdots<K_{m}\}$ are such that $K_{1}\geq 0$, $K_{m}<\mathfrak{s}(A)-M$, and $K_{t+1}-K_{t}\geq M$ for $1\leq t < m$. Suppose that for each $1\leq i\leq  l $ and $1\leq j\leq n_{i}$, we have a set $F_{i,j}\subset X_{i}$ that is closed in $X_{i}$ and disjoint from each set $B_{i,j-K_{t}}$ (see \autoref{Bik def}) for $1\leq t\leq m$. Then, there is a function $\Theta\in A$ such that:
	\begin{enumerate}
		\item for all $1\leq i\leq  l $ and $x\in X_{i}$, $\Theta_{i}(x)$ is a diagonal matrix with entries in $[0,1]$ whose final $M$ diagonal entries are all $0$, and such that at most one of every $M$ consecutive diagonal entries is non-zero;
		\item for all $1\leq i\leq  l $, $1\leq j\leq n_{i}$, and $x\in F_{i,j}$, we have $\Theta_{i}(x)_{j,j}=0$;
		\item for all $1\leq i\leq  l $ and $1\leq j\leq n_{i}$, there is a (possibly empty) open subset $U_{i,j}\subset X_{i}$ containing $B_{i,j}$ with the property that if $x\in U_{i,j}$, then $\Theta_{i}(x)_{j+K_{t},j+K_{t}}=1$ for all $1\leq t\leq m$. 
	\end{enumerate}
\end{lemma}

\begin{proof}	
	Using the hypotheses of this lemma, \autoref{indicatorprep} furnishes a function $\Phi\in A$ such that: \begin{enumerate}[font=\itshape, label=(\alph*)]
		\item for all $1\leq i\leq  l $ and $x\in X_{i}$, $\Phi_{i}(x)$ is a diagonal matrix with entries in $[0,1]$ whose final $M$ diagonal entries are all $0$, and such that at most one of every $M$ consecutive diagonal entries is non-zero;
		\item for all $1\leq i\leq \ell$ and $1\leq j\leq n_{i}$: $\Phi_{i}(x)_{j,j}=1$ if and only if there is a $1\leq t\leq m$ such that $x\in B_{i,j-K_{t}}$. 
	\end{enumerate}
	Let us use $\Phi$ to construct a function $\Theta\in A$ satisfying conditions \textit{(1)} to \textit{(3)}. 
	
	Given $\delta\in [0,1)$, define $g\colon[0,1]\to[0,1]$ by
	$$
	\begin{aligned}
	g(x):=
	\begin{cases}
	0& \text{ if }  0\leq x\leq\delta\\
	\text{linear} &\text{ if }  \delta\leq x\leq\frac{1+\delta}{2}\\
	1&\text{ if } \frac{1+\delta}{2}\leq x\leq 1.	
	\end{cases}
	\end{aligned}
	$$
	For $1\leq i\leq  l $, define $\Theta_{i}\colon X_{i}\to M_{n_{i}}$ by 
	$$
	\Theta_{i}(x):=\diag(g(\Phi_{i}(x)_{1,1}),\ldots,g(\Phi_{i}(x)_{n_{i},n_{i}})).
	$$
	Then $\Theta:=\bigoplus_{i=1}^{ l }\Theta_{i}\in\bigoplus_{i=1}^{ l }C(X_{i},M_{n_{i}})$. Since each diagonal entry of $\Phi$ is modified in the same way in the definition of $\Theta$, it is straightforward to check that $\Theta$ is compatible with the diagonal structure of $A$. Hence, $\Theta\in A$. Moreover, since $\Theta_{i}(x)_{j,j}=0$ whenever $\Phi_{i}(x)_{j,j}=0$, it is clear that $\Theta$ satisfies \textit{(1)} since $\Phi$ satisfies \textit{(a)}.
		
	To see that $\Theta$ satisfies \textit{(2)}, fix $1\leq i\leq  l $ and $1\leq j\leq n_{i}$. Since $F_{i,j}$ is disjoint from each $B_{i,j-K_{t}}$ (for $1\leq t\leq m$), condition \textit{(b)} guarantees that $\Phi_{i}(x)_{j,j}<1$ for all $x\in F_{i,j}$. Since $F_{i,j}$ is compact, there is a $\delta_{i,j}\in [0,1)$ such that $\Phi_{i}(x)_{j,j}\leq \delta_{i,j}$ for all $x\in F_{i,j}$. On choosing $\delta:=\max\{\delta_{i,j}:1\leq i\leq  l ,1\leq j\leq n_{i}\}\in [0,1)$ in our definition of $g$ above, it follows that $\Theta_{i}(x)_{j,j}=0$ whenever $1\leq i\leq  l $, $1\leq j\leq n_{i}$, and $x\in F_{i,j}$, which proves \textit{(2)}.
		
	Finally, to see that $\Theta$ satisfies \textit{(3)}, fix $1\leq i\leq  l $ and $1\leq j\leq n_{i}$. If $j>n_{i}-(\mathfrak{s}(A)-1)$, we may take $U_{i,j}=\varnothing$ since $B_{i,j}=\varnothing$ by \autoref{Bik} for such $j$. For $j\leq n_{i}-(\mathfrak{s}(A)-1)$, note that if $x\in B_{i,j}$, then by \textit{(b)}, $\Phi_{i}(x)_{j+K_{t},j+K_{t}}=1$ for all $1\leq t\leq m$. Since $g$ is $1$ in a neighbourhood of $1$, it follows that for each $t$, there is an open set $U_{t}\supset B_{i,j}$ on which the function $\Theta_{i}(\,\cdot\,)_{j+K_{t},j+K_{t}}\colon X_{i}\to [0,1]$ is equal to $1$. Taking $U_{i,j}:=\bigcap_{1\leq t\leq m} U_{t}$ yields \textit{(3)} and proves the lemma. 
\end{proof}


Given a sequence of DSH algebras $A_{1},A_{2},\ldots$, we denote by $ l (j)$ the length of the DSH algebra $A_{j}$. We denote the base spaces of $A_{j}$ by $X^{j}_{1},\ldots,X^{j}_{ l (j)}$ and the corresponding closed subspaces by $Y^{j}_{1},\ldots,Y^{j}_{ l (j)}$. We denote the size of the matrix algebras in the pullback definition of $A_{j}$ by $n^{j}_{1},\ldots,n^{j}_{ l (j)}$. Finally, we denote the sets defined in \autoref{Bik def} corresponding to $A_{j}$ by $B^{j}_{i,k}$. 

\begin{lemma}
	\label{big enough lemma}
	Suppose 
	$$
	A_{1}\overset{\psi_{1}}{\longrightarrow} A_{2}\overset{\psi_{2}}{\longrightarrow} A_{3}\overset{\psi_{3}}{\longrightarrow}\cdots\longrightarrow A:=\varinjlim A_{i}
	$$
	is a simple limit of infinite-dimensional DSH algebras with injective diagonal maps. Then, for all $j,m\in\NN$, there is a $j'>j$ such that $\mathfrak{s}(A_{j'})>m$ (where, recall, $\mathfrak{s}(A_{j})=\min\{n^{j}_{t}:1\leq t\leq l (j)\}$). 
\end{lemma}
\begin{proof}
	Since $A_{j}$ is infinite-dimensional, at least one of the base spaces must be infinite. Let $1\leq i\leq  l (j)$ be the largest integer for which $X_{i}^{j}$ is infinite. By \autoref{empty interior}, $X_{i}^{j}\setminus Y_{i}^{j}$ is also infinite and $Y_{i'}^{j}=\varnothing$ for $i<i'\leq  l (j)$. Choose pairwise-disjoint open in $X_{i}^{j}$ sets $\mathcal{O}_{1},\ldots,\mathcal{O}_{m+1}\subset X_{i}^{j}\setminus Y_{i}^{j}$. \autoref{open in spectrum} guarantees that $\mathcal{O}_{1},\ldots,\mathcal{O}_{m+1}$ are all open with respect to the hull-kernel topology on $\hat{A}_{j}$. By \autoref{simplicity}, there is a $j'>j$ such that for all $x\in \hat{A}_{j'}$, $\hat{\psi}_{j',j}(x)$ contains a point from each of $\mathcal{O}_{1},\ldots,\mathcal{O}_{m+1}$. Hence, $n_{i}^{j'}\geq m+1$ for all $1\leq i\leq  l (j')$, which proves the lemma.
\end{proof}

\begin{lemma}
	\label{far out}
	Suppose 
	$$
	A_{1}\overset{\psi_{1}}{\longrightarrow} A_{2}\overset{\psi_{2}}{\longrightarrow} A_{3}\overset{\psi_{3}}{\longrightarrow}\cdots\longrightarrow A:=\varinjlim A_{i}
	$$
	is a simple limit of infinite-dimensional DSH algebras with injective diagonal maps. Suppose that $f$ is a non-invertible element belonging to some $A_{j}$ and that $\epsilon > 0$. Then, there exist $f'\in A_{j}$ with $\|f-f'\|\leq \epsilon$ and $M\in\NN$ such that for all $N\in\NN$ there exist $j'>j$ satisfying $\mathfrak{s}(A_{j'})>NM$ and unitaries $V,V'\in A_{j'}$ with the following properties:
	\begin{enumerate}
		\item for any $1\leq i\leq  l (j')$ and $1\leq k\leq n_{i}^{j'}$, there is a (possibly empty) open subset $U_{i,k}$ of $X^{j'}_{i}$ containing $B_{i,k}^{j'}$ such that for all $x\in U_{i,k}$, $(V\psi_{j',j}(f')V')_{i}(x)$ has zero crosses at indices $k, k+M, k+2M,\ldots,k+(N-1)M$;
		\item for all $1\leq i\leq l (j')$ and $x\in X^{j'}_{i}$, we have $\mathfrak{r}((V\psi_{j',j}(f')V')_{i}(x))\leq \mathfrak{S}(A_{j})+M-1$ (where, recall, $\mathfrak{S}(A_{j})=\max\{n^{j}_{t}:1\leq t\leq l (j)\}$). 
	\end{enumerate}
\end{lemma}

\begin{proof}
	Let $f',w,v,\Delta\in A_{j}$ and $U\subset\hat{A}_{j}$ be given as in \autoref{perturb} (when applied to $f$ and $\epsilon$) and set $g:=wf'v$. Then, at every point in $U$, $g$ has a zero cross at index $1$ and the $(1,1)$-entry of $\Delta$ is $1$. By \autoref{simplicity} and \autoref{DSH spectrum}, there is a $j''>j$ such that $\hat{\psi}_{j'',j}([\ev_{x}])$ contains a point in $U$ for all $x\in \bigsqcup_{i=1}^{ l (j'')}(X_{i}^{j''}\setminus Y_{i}^{j''})$. Since $\psi_{j'',j}$ is diagonal, this means that for $1\leq i\leq l (j'')$ and $x\in X^{j''}_{i}\setminus Y_{i}^{j''}$, at least one of the points $x$ decomposes into under $\psi_{j'',j}$ lies in $U$, so that the matrices $\psi_{j'',j}(g)_{i}(x)$ and $\psi_{j'',j}(\Delta)_{i}(x)$ have a zero cross and a $1$, respectively, at the same index along their respective diagonal. Owing to the decomposition structure of $A_{j''}$, these two results hold, in fact, for all $1\leq i\leq  l (j'')$ and $x\in X_{i}^{j''}$. Take $M:=2\mathfrak{S}(A_{j''})$ and let $N\in\NN$ be arbitrary. 
	
	By \autoref{big enough lemma}, there is a $j'>j''$ such that 
	\begin{equation}
	\label{big enough}	
		\mathfrak{s}(A_{j'})>NM.
	\end{equation}
	Let $\Delta':=\psi_{j',j}(\Delta)=\psi_{j',j''}(\psi_{j'',j}(\Delta))$ and $g':=\psi_{j',j}(g)=\psi_{j',j''}(\psi_{j'',j}(g))$. Given $1\leq i\leq  l (j')$ and $x\in X^{j'}_{i}$ and regarding $\Delta'$ as a diagonal image under $\psi_{j',j''}$, it follows from the definition of $M$ that any $M$ consecutive entries down the diagonal of $\Delta_{i}'(x)$ must contain a $1$. Moreover, regarding $g'$ and $\Delta'$ as diagonal images under $\psi_{j',j}$ shows that $g'_{i}(x)$ has a zero cross at index $k$ whenever $\Delta'_{i}(x)_{k,k}>0$ (as a consequence of the conclusion of \autoref{perturb}) and that $\mathfrak{r}(g_{i}'(x))\leq \mathfrak{S}(A_{j})$.
	
	We now apply \autoref{indicators} with the natural number $M$, with $m=N$, $K_{1}=0, K_{2}=M, \ldots, K_{N}=(N-1)M$, and $F_{i,k}=\varnothing$ for $1\leq i\leq  l (j')$ and $1\leq k\leq n_{i}^{j'}$ (note that $K_{N}<\mathfrak{s}(A_{j'})-M$ by \autoref{big enough}). This furnishes a function $\Theta\in A_{j'}$ with the following properties:
	\begin{enumerate}[label=(\Roman*)]
		\item for all $1\leq i\leq  l (j')$ and $x\in X_{i}^{j'}$, $\Theta_{i}(x)$ is a diagonal matrix with entries in $[0,1]$ whose final $M$ diagonal entries are all $0$, and such that at most one of every $M$ consecutive diagonal entries is non-zero;
		\item for all $1\leq i\leq  l (j')$ and $1\leq k\leq n_{i}^{j'}$, there is a (possibly empty) open subset $U_{i,k}\subset X_{i}^{j'}$ containing $B_{i,k}^{j'}$ with the property that if $x\in U_{i,k}$, then $\Theta_{i}(x)_{k+aM,k+aM}=1$ for all $0\leq a\leq N-1$.
	\end{enumerate}
	Fix $1\leq i\leq l (j')$. Given $x\in X_{i}^{j'}$ and $1\leq k\leq n_{i}^{j'}-(M-1)$, let
	$$
	u_{k}^{i}(x):=\prod_{t=1}^{M-1}u^{i}_{(k \ k+t)}\left(\Theta_{i}(x)_{k,k} \Delta_{i}'(x)_{k+t,k+t}\right)\in M_{n_{i}^{j'}},
	$$
	where each $u^{i}_{(k \ k+t)}\colon [0,1]\to M_{n_{i}^{j'}}$ is a connecting path of unitaries as described in \autoref{unitaries}. 	
	Define $W_{i}\in C(X_{i}^{j'}, M_{n_{i}^{j'}})$ to be the unitary
	$$
	W_{i}(x):=\prod_{k=1}^{n_{i}^{j'}-M}u_{k}^{i}(x).
	$$
	Set $W:=(W_{1},\ldots,W_{ l (j')})$ and take $V:=W\psi_{j',j}(w)$ and $V':=\psi_{j',j}(v)W^{*}$. Before showing that $W\in A_{j'}$, let us prove that statements \textit{(1)} and \textit{(2)} of \autoref{far out} hold.
	
	Fix $x\in X_{i}^{j'}$. Note that if $\Theta_{i}(x)_{k',k'}=0$, then $u_{k'}^{i}(x)=1_{n_{i}^{j'}}$. Let $\{k_{1}<\cdots<k_{s}\}$ denote the set of indices $r$ at which $\Theta_{i}(x)_{r,r}>0$. Then, $W_{i}(x)=u_{k_{1}}^{i}(x)\cdots u_{k_{s}}^{i}(x)$, where, by (I) above, $k_{p+1}-k_{p}\geq M$ for $1\leq p < s$ and $k_{s}\leq n_{i}^{j'}-M$. Note that conjugating any matrix by $u_{k_{p}}^{i}(x)$ only affects the $k_{p},\ldots, k_{p}+(M-1)$ rows and columns of that matrix. Thus, for $p\not=q$, the indices of the rows and columns affected when conjugating by $u_{k_{p}}^{i}(x)$ do not overlap with the indices of the rows and columns affected when conjugating by $u_{k_{q}}^{i}(x)$. This observation will be used to prove \textit{(1)} and \textit{(2)} below. 
	
	To prove \textit{(1)}, fix $1\leq k\leq n_{i}^{j'}$ and assume $x\in U_{i,k}$. For $p=s,s-1,\ldots,1$, let 
	$$
	D_{p}:=u_{k_{p}}^{i}(x)\cdots u_{k_{s}}^{i}(x)g_{i}'(x)u_{k_{s}}^{i}(x)^{*}\cdots u_{k_{p}}^{i}(x)^{*}
	$$
	(setting $D_{s+1}:=g'_{i}(x)$) and apply part \textit{(b)} of \autoref{untouched crosses} with $M$, $n=n^{j'}_{i}$, $l=k_{p}$, $\xi_{l+t}=\Theta_{i}(x)_{l,l}\Delta_{i}'(x)_{l+t,l+t}$ for $t=0,1,\ldots M-1$, $D=D_{p+1}$, and $U=u_{k_{p}}^{i}(x)$ to conclude that $D_{p}$ has a zero cross at any index among $\{1,\ldots,n_{i}^{j'}\}\setminus\{k_{p},\ldots,k_{p}+M-1\}$ whenever $D_{p+1}$ does. 
	
	Now, fix an integer $0\leq a\leq N-1$. Let us show that $W_{i}(x)g_{i}'(x)W_{i}(x)^{*}$ has a zero cross at index $k+aM$. By (II) above, $\Theta_{i}(x)_{k+aM, k+aM}=1$. Let $r$ denote the unique integer such that $k_{r}=k+aM$. Applying the result obtained just above inductively $s-r$ times, it follows that for every $q\in\{k_{r},\ldots,k_{r}+M-1\}$,  $D_{r+1}$ has a zero cross at index $q$ whenever $g_{i}'(x)$ does; in particular, for any such $q$, $D_{r+1}$ has a zero cross at index $q$ provided that $\Delta_{i}'(x)_{q,q}>0$. Hence, since any $M$ consecutive entries along the diagonal of $\Delta_{i}'(x)$ must contain a $1$, the assumptions of part \textit{(c)} of \autoref{untouched crosses} are satisfied with $M$, $n=n^{j'}_{i}$, $l=k+aM$, $\xi_{l+t}=\Theta_{i}(x)_{l,l}\Delta_{i}'(x)_{l+t,l+t}=\Delta_{i}'(x)_{l+t,l+t}$ for $t=0,1,\ldots, M-1$, $D=D_{r+1}$, and $U=u_{k_{r}}^{i}(x)$. Thus, we may apply that part of the lemma to deduce that $D_{r}$ has a zero cross at index $k+aM$. Appealing to the conclusion of the previous paragraph inductively $r-1$ times, it follows that $D_{1}=W_{i}(x)g_{i}'(x)W_{i}(x)^{*}$ has a zero cross at index $k+aM$, since $k+aM=k_{r}$ is not among the indices affected upon conjugation by $u_{k_{1}}^{i}(x)\cdots u_{k_{r-1}}^{i}(x)$. This proves \textit{(1)}.
	
	Next, recall that $g'$ is the diagonal image of $g$, which has bandwidth at most $\mathfrak{S}(A_{j})$ at every point. To prove \textit{(2)}, therefore, it suffices to show that for any given matrix $D=(D_{q,t})\in M_{n_{i}^{j'}}$, we have $\mathfrak{r}(W_{i}(x)DW_{i}(x)^{*})\leq \mathfrak{r}(D)+M-1$. This is most easily seen by drawing a picture and examining which rows and columns are potentially affected upon conjugation by the $u_{k_{p}}^{i}$'s:
	
	\begin{figure}[H]
		\begin{minipage}{0.5\textwidth}
			\centering
			\includegraphics[scale=.9]{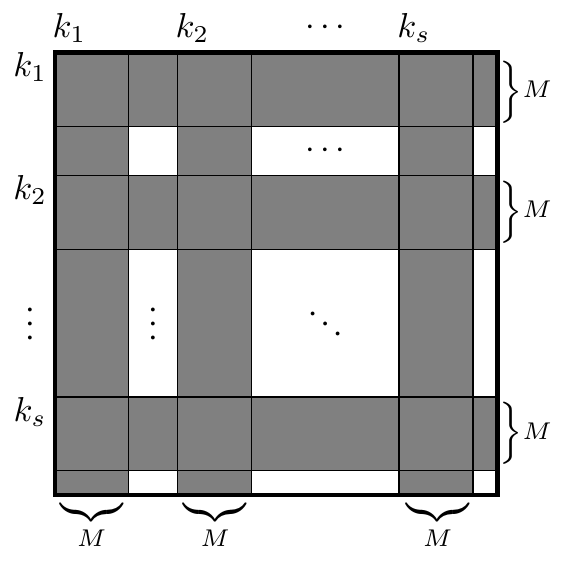}
			\caption{Affected Rows/Columns}
			\label{affected rows/columns}
		\end{minipage}\hfill
		\begin{minipage}{.5\textwidth}
			Since the sets $\{k_{p},\ldots,k_{p}+(M-1)\}$ for $1\leq p\leq s$ are disjoint, the block rows and columns are disjoint. Suppose we are given an index $(q,t)$ that lies in the shaded region of the diagram in \autoref{affected rows/columns}, and suppose that $\lambda$ is the number at entry $(q,t)$ of $W_{i}(x)DW_{i}(x)^{*}$. Upon partitioning this shaded region, it follows that the index $(q,t)$ lies in one of the following three shaded subregions in \autoref{regionA} to \autoref{regionC}:
		\end{minipage}	
		
	\end{figure}
	\begin{figure}[H]
		\centering
		\begin{minipage}{0.333\textwidth}
			\centering
			\includegraphics[scale=.81]{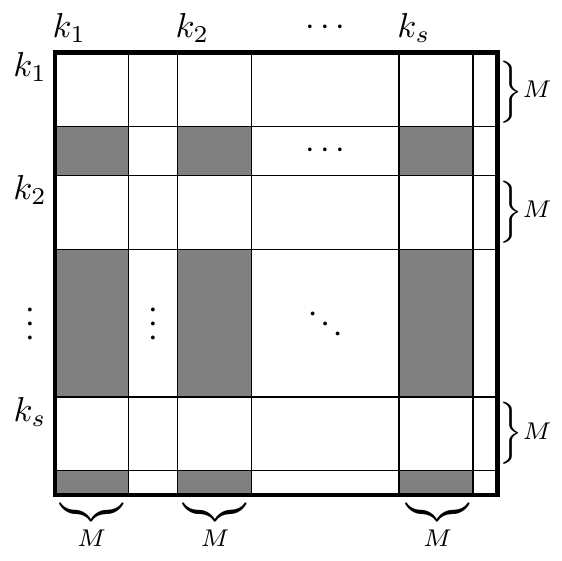}
			\caption{Region A}
			\label{regionA}
		\end{minipage}\hfill
		\begin{minipage}{0.333\textwidth}
			\centering
			\includegraphics[scale=.81]{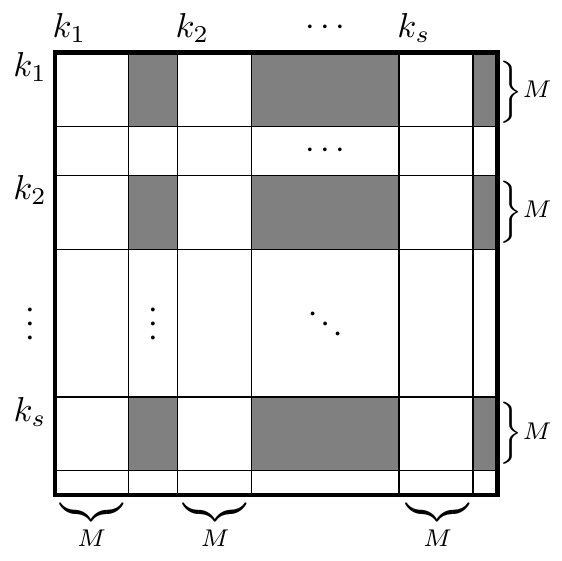}
			\caption{Region B}
			\label{regionB}
		\end{minipage}\hfill
		\begin{minipage}{0.333\textwidth}
			\centering
			\includegraphics[scale=.81]{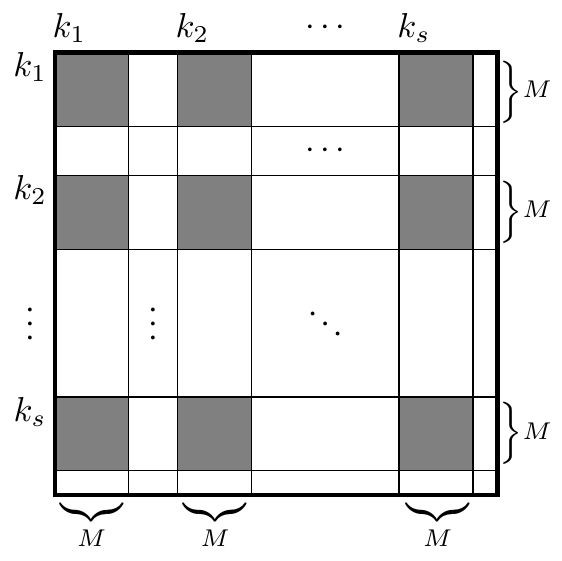}
			\caption{Region C}
			\label{regionC}
		\end{minipage}
	\end{figure}
	
	On \nameref{regionA}, the matrices $W_{i}(x)DW_{i}(x)^{*}$ and $DW_{i}(x)^{*}$ are equal. Hence, if $(q,t)$ lies in \nameref{regionA} and $p$ is such that $k_{p}\leq t\leq k_{p}+M-1$, then $\lambda$ is a linear combination of $D_{q,k_{p}},\ldots,D_{q,k_{p}+M-1}$. Thus, $\lambda$ can be non-zero only if one of $D_{q,k_{p}},\ldots,D_{q,k_{p}+M-1}$ is non-zero. Hence, no non-zero entry in this region is more than $M-1$ indices away from a non-zero entry in $D$. On \nameref{regionB}, the matrices $W_{i}(x)DW_{i}(x)^{*}$ and $W_{i}(x)D$ are equal, and so a symmetrical analysis shows that the same is true also for non-zero entries in this region. If $(q,t)$ lies in one of the $s^{2}$ disjoint $M\times M$ blocks in \nameref{regionC}, then $\lambda$ is a linear combination of the corresponding entries in $D$ lying in that block. Hence, in this case $\lambda$ is $0$ unless that $M\times M$ block in $D$ contains a non-zero entry. Thus, no non-zero entry of $W_{i}(x)DW_{i}(x)^{*}$ in \nameref{regionC} is more than $M-1$ units further away from the diagonal than a non-zero entry of $D$. This analysis proves that $\mathfrak{r}(W_{i}(x)DW_{i}(x)^{*})\leq \mathfrak{r}(D)+M-1$, yielding \textit{(2)}.
	
	To conclude, let us show that $W\in A_{j'}$. Fix $1\leq i\leq  l (j')$ and suppose that $y\in Y_{i}^{j'}$ decomposes into $x_{1}\in X_{i_{1}}^{j'}\setminus Y_{i_{1}}^{j'},\ldots,x_{s}\in X_{i_{s}}^{j'}\setminus Y_{i_{s}}^{j'}$. For $1\leq k\leq s$, let $p_{k}:=1+n^{j'}_{i_{1}}+\cdots+n^{j'}_{i_{k-1}}$. Note that by \autoref{big enough}, $p_{s}\leq n_{i}^{j'}-\mathfrak{s}(A_{j'})+1\leq n_{i}^{j'}-M$. Thus, we may write
	\begin{equation}
	\label{Wi}
	W_{i}(y)=\prod_{k=1}^{n_{i}^{j'}-M}u_{k}^{i}(y)=\prod_{m=1}^{s-1}\prod_{k=p_{m}}^{p_{m+1}-1}u_{k}^{i}(y)\times\prod_{k=p_{s}}^{n_{i}^{j'}-M}u_{k}^{i}(y).
	\end{equation}
	Fix $1\leq m<s$. Then,
	$$
	\prod_{k=p_{m}}^{p_{m+1}-1}u_{k}^{i}(y)=\prod_{k=p_{m}}^{p_{m+1}-1}\prod_{t=1}^{M-1}u_{(k \ k+t)}^{i}(\Theta_{i}(y)_{k,k}\Delta_{i}'(y)_{k+t,k+t}).
	$$ 
	By (I), the last $M$ entries of $\Theta_{i_{m}}(x_{m})$ are zero. Hence, on account of the diagonal decomposition of $\Theta_{i}(y)$, the quantity above is equal to
	$$
	\prod_{k=p_{m}}^{(p_{m+1}-1)-M}\prod_{t=1}^{M-1}u_{(k \ k+t)}^{i}(\Theta_{i_{m}}(x_{m})_{k-p_{m}+1,k-p_{m}+1}\Delta_{i_{m}}'(x_{m})_{k-p_{m}+1+t,k-p_{m}+1+t}),
	$$
	which, upon relabelling indices, becomes
	\begin{equation}
	\label{reindex}
	\prod_{q=1}^{p_{m+1}-p_{m}-M}\prod_{t=1}^{M-1}u^{i}_{(q+p_{m}-1 \ q+p_{m}-1+t)}(\Theta_{i_{m}}(x_{m})_{q,q}\Delta_{i_{m}}'(x_{m})_{q+t,q+t}).
	\end{equation}
	For each $1\leq q\leq p_{m+1}-p_{m}-M$ and $1\leq t\leq M-1$, note that 
	$$
	u^{i}_{(q+p_{m}-1 \ q+p_{m}-1+t)}=\diag(1_{p_{m}-1},u_{(q \ q+t)}^{i_{m}},1_{n_{i}^{j'}-p_{m+1}+1}).
	$$
	Hence, we may rewrite \autoref{reindex} as
	$$
	\begin{aligned}
	\diag\left(1_{p_{m}-1},\prod_{q=1}^{p_{m+1}-p_{m}-M}u_{q}^{i_{m}}(x_{m}),1_{n_{i}^{j'}-p_{m+1}+1}\right)
	=\diag\left(1_{p_{m}-1},W_{i_{m}}(x_{m}),1_{n_{i}^{j'}-p_{m+1}+1}\right).
	\end{aligned}
	$$
	Therefore, for all $1\leq m <s$,
	$$
	\prod_{k=p_{m}}^{p_{m+1}-1}u_{k}^{i}(y)=\diag\left(1_{p_{m}-1},W_{i_{m}}(x_{m}),1_{n_{i}^{j'}-p_{m+1}+1}\right)
	$$
	and, similarly, 
	$$
	\prod_{k=p_{s}}^{n_{i}^{j'}-M}u_{k}^{i}(y)=\diag(1_{p_{s}-1},W_{i_{s}}(x_{s})).
	$$
	Plugging this into \autoref{Wi} shows that $W_{i}(y)=\diag(W_{i_{1}}(x_{1}),\ldots,W_{i_{s}}(x_{s}))$, which proves that $W\in A$. The proof of \autoref{far out} is now complete.
\end{proof}

\begin{definition}[Block point]
	Given a matrix $D\in M_{n}$ and $1\leq k\leq n$, we say that $D$ has a \textit{block point at index $k$} provided that $D_{i,j}=0$ if either $i\geq k$ and $j<k$ or $i<k$ and $j\geq k$. 
\end{definition}

\begin{lemma}
	\label{block on set}
	Suppose $A$ is a DSH algebra of length $ l $. Suppose $f\in A$ and $\epsilon>0$. Then there is a $g\in A$ with $\|g-f\|\leq \epsilon$ and with the property that for all $1\leq i\leq  l $ and $1\leq k\leq n_{i}$, there are (possible empty) open sets $\mathcal{O}_{i,k}\supset B_{i,k}$ in $X_{i}$ such that $g_{i}(x)$ has a block point at index $k$ whenever $x\in \mathcal{O}_{i,k}$. Moreover, $g$ can be chosen so that for each $1\leq i\leq l $ and $x\in X_{i}$, $g_{i}(x)$ has a zero cross at index $k$ whenever $f_{i}(x)$ does, and $\mathfrak{r}(g_{i}(x))\leq \mathfrak{r}(f_{i}(x))$. 
\end{lemma}

\begin{proof}
	Given $1\leq i\leq  l $ and $1\leq s,t\leq n_{i}$, let $f_{i}(\, \cdot\,)_{s,t}\in C(X_{i})$ denote the function taking $x$ into $f_{i}(x)_{s,t}$. Let $\delta=\epsilon/\mathfrak{S}(A)^{2}$. Define $h\in  C(\mathbb{C})$ by $h(z):=\frac{z}{|z|}\cdot\max(0,|z|-\delta)$, where it is understood that $h(0)=0$. Note that for any $z\in \mathbb{C}$, if $|z|\leq\delta$, then $|h(z)-z|=|z|\leq \delta$, and if $|z|> \delta$, then $|h(z)-z|=\left|\frac{z}{|z|}(|z|-\delta)-z\right|=\left|\frac{z}{|z|}\delta\right|=\delta$. Thus, for all $z\in \mathbb{C}$, $|h(z)-z|\leq\delta$ and, hence, $|f_{i}(x)_{s,t}-h(f_{i}(x)_{s,t})|\leq \delta$ given any $1\leq i\leq l $, $1\leq s,t\leq n_{i}$, and $x\in X_{i}$. Define $g_{i}(x)_{s,t}:=h(f_{i}(x)_{s,t})$ and denote by $g_{i}$ the matrix-valued function in $C(X_{i},M_{n_{i}})$ given by $(g_{i}(\,\cdot\,)_{s,t})_{s,t}$. Set $g:=(g_{1},\ldots,g_{ l })\in\bigoplus_{i=1}^{ l }C(X_{i},M_{n_{i}})$. For $x\in X_{i}$, $$
	\|f_{i}(x)-g_{i}(x)\|\leq\sum_{1\leq s,t\leq n_{i}}\|f_{i}(x)_{s,t}-g_{i}(x)_{s,t}\|\leq n_{i}^{2}\delta\leq \epsilon.
	$$
	Hence, $\|f-g\|\leq\epsilon$.
	
	To see that $g\in A$, observe that if $y\in Y_{i}$ decomposes into $x_{1}\in X_{i_{1}}\setminus Y_{i_{1}},\ldots,x_{t}\in X_{i_{t}}\setminus Y_{i_{t}}$, then $f_{i}(y)=\diag(f_{i_{1}}(x_{1}),\ldots,f_{i_{t}}(x_{t}))$. Applying $h$ to each coordinate yields that
	$g_{i}(y)=\diag(g_{i_{1}}(x_{1}),\ldots,g_{i_{t}}(x_{t}))$. Furthermore, since $h(0)=0$, $g_{i}(x)$ must have a zero cross at any index that $f_{i}(x)$ does, and $\mathfrak{r}(g_{i}(x))\leq \mathfrak{r}(f_{i}(x))$. 
	
	Lastly, fix $1\leq i\leq l $ and $1\leq k\leq n_{i}$. Let us show how to construct $\mathcal{O}_{i,k}$. If $B_{i,k}=\varnothing$, take $\mathcal{O}_{i,k}:=\varnothing$. Otherwise, suppose $x\in B_{i,k}$. Then, $f_{i}(x)$ has a block point at index $k$. Let $I\subset \{1,\ldots,n_{i}\}^{2}$ denote the set of indices $(s,t)$ such that $s\geq k$ and $t<k$ or such that $s<k$ and $t\geq k$. Given $(s,t)\in I$, it follows that $f_{i}(x)_{s,t}=0$, and hence, that $g_{i}( \, \cdot \,)_{s,t}$ is $0$ on an open set $U_{s,t}(x)\subset X_{i}$ containing $x$. Then $U_{s,t}:=\bigcup_{x\in B_{i,k}}U_{s,t}(x)$ is an open set containing $B_{i,k}$ on which $g_{i}( \, \cdot \,)_{s,t}$ vanishes. Take $\mathcal{O}_{i,k}:=\bigcap_{(s,t)\in I}U_{s,t}$. By construction, then, $g_{i}(x)_{s,t}=0$ whenever $x\in \mathcal{O}_{i,k}$ and $(s,t)\in I$. Thus, $g_{i}(x)$  has a block point at index $k$ provided that $x\in \mathcal{O}_{i,k}$, which completes the proof.
\end{proof}

\begin{lemma}
	\label{cross shift}
	Suppose $A$ is a DSH algebra of length $ l $ and that $M,N\in\NN$ with $NM<\mathfrak{s}(A)$. Suppose $f$ is an element of $A$ with the property that for all $1\leq i\leq  l $ and $1\leq j\leq n_{i}$, there is a (possibly empty) open set $U_{i,k}\supset B_{i,k}$ in $X_{i}$ such that if $x\in U_{i,k}$, then $f_{i}(x)$ has zero crosses at indices $k,k+M,\ldots,k+(N-1)M$ and a block point at index $k$. Then, there exists a unitary $V\in A$ with the following properties:
	\begin{enumerate}
		\item for all $1\leq i\leq l $ and $1\leq k\leq n_{i}$, there are open sets $\mathcal{O}_{i,k}\supset B_{i,k}$ in $X_{i}$ such that $V_{i}(x)f_{i}(x)V_{i}(x)^{*}$ has zero crosses at indices $k,k+1,\ldots,k+N-1$ whenever $x\in \mathcal{O}_{i,k}$;
		\item$\mathfrak{r}(V_{i}(x)f_{i}(x)V_{i}(x)^{*})\leq \mathfrak{r}(f_{i}(x))+2$ for all $1\leq i\leq  l $ and $x\in X_{i}$.
	\end{enumerate} 
\end{lemma}

\begin{proof}

	Apply \autoref{indicators} with the natural number $NM$, the index set $K=\{0\}$, and closed sets $F_{i,k}:=X_{i}\setminus U_{i,k}$ for $1\leq i\leq  l $ and $1\leq k\leq n_{i}$ to obtain a function $\Theta\in A$ possessing the following properties:
	\begin{enumerate}[label=(\Roman*)]
		\item for all $1\leq i\leq  l $ and $x\in X_{i}$, $\Theta_{i}(x)$ is a diagonal matrix with entries in $[0,1]$ whose final $NM$ diagonal entries are all $0$ and such that at most one of any $NM$ consecutive diagonal entries is non-zero;
		\item for all $1\leq i\leq  l $ and $1\leq k\leq n_{i}$, if $x\notin U_{i,k}$, then $\Theta_{i}(x)_{k,k}=0$;
		\item for all $1\leq i\leq  l $ and $1\leq k\leq n_{i}$, there is a (possibly empty) open subset $\mathcal{O}_{i,k}\subset X_{i}$ containing $B_{i,k}$ with the property that $\Theta_{i}(x)_{k,k}=1$ whenever $x\in\mathcal{O}_{i,k}$.
	\end{enumerate}
	
	Now, fix $1\leq i\leq l $. For $1\leq k\leq n_{i}-NM$, let $u_{k}\in C(X_{i},M_{n_{i}})$ be the unitary
	$$
	u_{k}(x):=\diag(1_{k-1}, W(\Theta_{i}(x)_{k,k}),1_{n_{i}-(NM+k-1)}),
	$$
	where $W$ is the unitary in $C([0,1],M_{NM})$ given by \autoref{matrix} with $z_{1}:=1$, $z_{2}:=1+M$, \ldots, $z_{N}:=1+(N-1)M$. For $n_{i}-NM<k\leq n_{i}$, set $u_{k}\equiv1_{n_{i}}$. Define $V_{i}\in C(X_{i},M_{n_{i}})$ to be the unitary
	$$
	V_{i}:=\prod_{k=1}^{n_{i}}u_{k}.
	$$
	
	For $x\in X_{i}$, let $K(x):=\{1\leq k\leq n_{i}:\Theta_{i}(x)_{k,k}>0\}$ and write $K(x)=\{k_{1},\ldots,k_{s}\}$, where $k_{1}<\cdots<k_{s}$ and put $k_{s+1}:=n_{i}+1$. Note that $k_{1}=1$ by (III) above since $B_{i,1}=X_{i}$ by \autoref{Bik}, and for $1\leq t\leq s$, $k_{t+1}-k_{t}\geq NM$ by (I) above. If $k\notin K(x)$, then $u_{k}\equiv 1_{n_{i}}$. Hence, we may write
	\begin{equation}
	\label{Vi}
	\begin{aligned}
	V_{i}(x)&=\prod_{t=1}^{s}u_{k_{t}}(x)\\
	&=\diag\left(W(\Theta_{i}(x)_{k_{1},k_{1}}),1_{d_{1}},W(\Theta_{i}(x)_{k_{2},k_{2}}),1_{d_{2}},\ldots,W(\Theta_{i}(x)_{k_{s},k_{s}}),1_{d_{s}}\right),
	\end{aligned}
	\end{equation}
	where $d_{t}:=k_{t+1}-(k_{t}+NM)$ for $1\leq t\leq s$. 
	
	Let $V:=(V_{1},\ldots,V_{ l })$. In order to prove \autoref{cross shift}, let us show that \textit{(1)} holds, then that \textit{(2)} holds, and finally that $V\in A$. 
	
	To prove \textit{(1)} and \textit{(2)}, fix $1\leq i\leq  l $ and $x\in X_{i}$, and let $K(x)=\{k_{1},\ldots,k_{s}\}$ and $k_{s+1}$ be defined as above. For $1\leq t\leq s$, we have $\Theta_{i}(x)_{k_{t},k_{t}}>0$. Hence, by (II) above, it must be that $x\in U_{i,k_{t}}$ and, thus, $f_{i}(x)$ has a block point at index $k_{t}$ and zero crosses at indices $k_{t},k_{t}+M,\ldots,k_{t}+(N-1)M$ by the assumption of the lemma. Thus, $f_{i}(x)=\diag(Q_{1},Q_{2},\ldots,Q_{s})$, where $Q_{t}$ is a $k_{t+1}-k_{t}$ block for $1\leq t\leq s$ and has zero crosses at $1,1+M,\ldots,1+(N-1)M$. Therefore, in light of the decomposition of $V_{i}(x)$ in \autoref{Vi}, we may view $V_{i}(x)f_{i}(x)V_{i}(x)^{*}$ as a block-diagonal matrix $\diag(B_{1},\ldots,B_{s})$ with
	$$
	B_{t}=\diag(W(\Theta_{i}(x)_{k_{t},k_{t}}),1_{d_{t}})\cdot Q_{t}\cdot\diag(W(\Theta_{i}(x)_{k_{t},k_{t}}),1_{d_{t}})^{*}.
	$$
	Thus, to prove \textit{(2)}, it suffices to show that $\mathfrak{r}(B_{t})\leq\mathfrak{r}(Q_{t})+2$. Furthermore, if $x\in \mathcal{O}_{i,k}$ for some $1\leq k\leq n_{i}$, then by (III), $\Theta_{i}(x)_{k,k}=1>0$, and so $k=k_{t}$ for some $1\leq t\leq s$. Since the block $B_{t}$ begins at index $k_{t}$ down the diagonal of $V_{i}(x)f_{i}(x)V_{i}(x)^{*}$, to prove \textit{(1)} it suffices to show that $B_{t}$ has zero crosses at indices $1, 2,\ldots,N$ whenever $\Theta_{i}(x)_{k_{t},k_{t}}=1$. 
	
	To this end, fix $1\leq t\leq s$ and write
	$$
	Q_{t}=
	\begin{pmatrix}
	D_{11} & D_{12}\\
	D_{21} & D_{22}
	\end{pmatrix},
	$$
	where $D_{11}\in M_{NM}$, $D_{22}\in M_{d_{t}}$, and $D_{12}$ and $D_{21}$ are $NM\times d_{t}$ and $d_{t}\times NM$ matrices, respectively. Note that $D_{11}$ has zero crosses at indices $1,1+M,\ldots,1+(N-1)M$, while the rows of $D_{12}$ and the columns of $D_{21}$ at these same indices consist entirely of zeros. We may write
	$$\begin{aligned}
	B_{t}=&
	\begin{pmatrix}
	W(\Theta_{i}(x)_{k_{t},k_{t}}) & 0_{NM\times d_{t}}\\
	0_{d_{t}\times NM} & 1_{d_{t}}
	\end{pmatrix}
	\begin{pmatrix}
	D_{11} & D_{12}\\
	D_{21} & D_{22}
	\end{pmatrix}
	\begin{pmatrix}
	W(\Theta_{i}(x)_{k_{t},k_{t}})^{*} & 0_{NM\times d_{t}}\\
	0_{d_{t}\times NM} & 1_{d_{t}}
	\end{pmatrix}\\
	=&\begin{pmatrix}
	W(\Theta_{i}(x)_{k_{t},k_{t}})D_{11}W(\Theta_{i}(x)_{k_{t},k_{t}})^{*} & W(\Theta_{i}(x)_{k_{t},k_{t}})D_{12}\\
	D_{21}W(\Theta_{i}(x)_{k_{t},k_{t}})^{*} & D_{22}
	\end{pmatrix}.
	\end{aligned}
	$$
	
	If $\Theta_{i}(x)_{k_{t},k_{t}}=1$, then 
	$$
	B_{t}=
	\begin{pmatrix}
	W(1)D_{11}W(1)^{*} & W(1)D_{12}\\
	D_{21}W(1)^{*} & D_{22}
	\end{pmatrix}.
	$$
	By our definition of $W$, we may apply \autoref{matrix} to conclude that $W(1)D_{11}W(1)^{*}$ has zero crosses at indices $1,2,\ldots,N$, and that the first $1,2,\ldots, N$ rows of $W(1)D_{12}$ and columns of $D_{21}W(1)^{*}$ consist only of zeros. It follows that $B_{t}$ has zero crosses at indices $1,2,\ldots, N$, which, based on the aforementioned analysis, proves \textit{(1)}. 
	
	Let us now prove \textit{(2)} by showing that $\mathfrak{r}(B_{t})\leq\mathfrak{r}(Q_{t})+2$. By our definition of $W$, we may apply \autoref{matrix} to obtain the following estimates:
	\begin{itemize}
		\item $\mathfrak{r}(W(\Theta_{i}(x)_{k_{t},k_{t}})D_{11}W(\Theta_{i}(x)_{k_{t},k_{t}})^{*})\leq \mathfrak{r}(D_{11})+2;$
		\vspace{5pt}
		\item $\mathfrak{r}\left(
		\begin{pmatrix}
		0_{NM\times NM} & W(\Theta_{i}(x)_{k_{t},k_{t}})D_{12}\\
		0_{d_{t}\times NM} & 0_{d_{t}\times d_{t}}
		\end{pmatrix}
		\right)\leq\mathfrak{r}\left(
		\begin{pmatrix}
		0_{NM\times NM} & D_{12}\\
		0_{d_{t}\times NM} & 0_{d_{t}\times d_{t}}
		\end{pmatrix}
		\right);$
		\vspace{5pt}
		\item $\mathfrak{r}\left(
		\begin{pmatrix}
		0_{NM\times NM} & 0_{NM\times d_{t}}\\
		D_{21}W(\Theta_{i}(x)_{k_{t},k_{t}})^{*} & 0_{d_{t}\times d_{t}}
		\end{pmatrix}
		\right)\leq\mathfrak{r}\left(
		\begin{pmatrix}
		0_{NM\times NM} & 0_{NM\times d_{t}}\\
		D_{21} & 0_{d_{t}\times d_{t}}
		\end{pmatrix}
		\right).$
	\end{itemize}	
	Combining these estimates gives 
	$$
	\mathfrak{r}(B_{t})=\mathfrak{r}\left(\begin{pmatrix}
	W(\Theta_{i}(x)_{k_{t},k_{t}})D_{11}W(\Theta_{i}(x)_{k_{t},k_{t}})^{*} & W(\Theta_{i}(x)_{k_{t},k_{t}})D_{12}\\
	D_{21}W(\Theta_{i}(x)_{k_{t},k_{t}})^{*} & D_{22}
	\end{pmatrix}\right)
	\leq \mathfrak{r}(Q_{t})+2,
	$$
	which proves \textit{(2)}. 
	
	Finally, let us verify that $V\in A$. Suppose $y\in Y_{i}$ decomposes into $x_{1}\in X_{i_{1}}\setminus Y_{i_{1}},\ldots,x_{r}\in X_{i_{r}}\setminus Y_{i_{r}}$. We need to show that $V_{i}(y)=\diag(V_{i_{1}}(x_{1}),\ldots,V_{i_{r}}(x_{r}))$. Let $K(y)=\{1\leq k\leq n_{i}:\Theta_{i}(y)_{k,k}>0\}$, as defined above. Write $K(y)=\{k_{1},\ldots,k_{s}\}$, where $1=k_{1}<\cdots<k_{s}$, and put $k_{s+1}:=n_{i}+1$. As before, let $d_{t}:=k_{t+1}-(k_{t}+NM)$ for $1\leq t\leq s$. Define $B(y):=\{1\leq k\leq n_{i}:y\in B_{i,k}\}$. By (III) above, $B(y)\subset K(y)$. Hence by \autoref{Bik}, for each $1\leq j\leq r$, there is a $t_{j}\in\{1,\ldots,s\}$ such that $1+n_{i_{1}}+\cdots+n_{i_{j-1}}=k_{t_{j}}$ (where $k_{t_{1}}=1=k_{1}$, so that $t_{1}=1$); set $t_{r+1}:=s+1$, so that $k_{t_{r+1}}=k_{s+1}=n_{i}+1$. 
	
	Now, fix $1\leq j\leq r$ and observe that $\Theta_{i_{j}}(x_{j})_{k,k}=\Theta_{i}(y)_{k_{t_{j}}+k-1, k_{t_{j}}+k-1}$. Therefore, 
	$$
	\begin{aligned}
	K(x_{j})&=\{1\leq k\leq n_{i_{j}}:\Theta_{i_{j}}(x_{j})_{k,k}>0\}\\
	&=\{k-k_{t_{j}}+1:k\in K(y)\text{ and }k_{t_{j}}\leq k<k_{t_{j+1}}\}\\
	&=\{k_{t}-k_{t_{j}}+1:t_{j}\leq t<t_{j+1}\}.
	\end{aligned}
	$$
	Moreover, if $t_{j}\leq t<t_{j+1}$, then $(k_{t+1}-k_{t_{j}}+1)-(k_{t}-k_{t_{j}}+1+NM)=d_{t}$. Given matrices $E_{1},\ldots,E_{p}$, let $\bigoplus_{q=1}^{p}E_{q}:=\diag(E_{1},\ldots,E_{p})$. Then, by the computation of $K(x_{j})$ above and \autoref{Vi}, it follows that
	$$
	\begin{aligned}
	V_{i_{j}}(x_{j})
	&=\bigoplus_{t_{j}\leq t<t_{j+1}}\diag\left(W(\Theta_{i_{j}}(x_{j})_{k_{t}-k_{t_{j}}+1,k_{t}-k_{t_{j}}+1}),1_{(k_{t+1}-k_{t_{j}}+1)-(k_{t}-k_{t_{j}}+1+NM)}\right)\\
	&=\bigoplus_{t_{j}\leq t<t_{j+1}}\diag\left(W(\Theta_{i}(y)_{k_{t},k_{t}}),1_{d_{t}}\right).
	\end{aligned}
	$$
	Therefore, 
	$$
	\begin{aligned}
	\diag(V_{i_{1}}(x_{1}),\ldots,V_{i_{r}}(x_{r}))&=\bigoplus_{1\leq j\leq r}\bigoplus_{t_{j}\leq t<t_{j+1}}\diag\left(W(\Theta_{i}(y)_{k_{t},k_{t}}),1_{d_{t}}\right)\\
	&=\bigoplus_{1\leq t\leq s}\diag\left(W(\Theta_{i}(y)_{k_{t},k_{t}}),1_{d_{t}}\right)\\
	&=V_{i}(y),
	\end{aligned}
	$$
	where the last equality follows by \autoref{Vi}. This shows that $V\in A$. The proof of \autoref{cross shift} is now complete.
\end{proof}


\begin{lemma}
	\label{lower triangular}
	Suppose $A$ is a DSH algebra of length $ l $ and that $1\leq N<\mathfrak{s}(A)$. Suppose $f\in A$ is such that for all $1\leq i\leq l $ and $1\leq k\leq n_{i}$, there is an open subset $U_{i,k}\subset X_{i}$ containing $B_{i,k}$ with the property that if $x\in U_{i,k}$, then $f_{i}(x)$ has zero crosses at indices $k,k+1,\ldots,k+N-1$, and such that $\mathfrak{r}(f_{i}(x))\leq N$ for all $x\in X_{i}$. Then, there is a unitary $V\in A$ such that for all $1\leq i\leq  l $ and $x\in X_{i}$, the matrix $(fV)_{i}(x)$ is strictly lower triangular.
\end{lemma}

\begin{proof}
	Apply \autoref{indicators} with the natural number $N$, the index set $K=\{0\}$, and the closed sets $F_{i,k}:=X_{i}\setminus U_{i,k}$ for $1\leq i\leq  l $ and $1\leq k\leq n_{i}$. This yields a function $\Theta\in A$ with the following properties:
	\begin{enumerate}[label=(\Roman*)]
		\item for all $1\leq i\leq  l $ and $x\in X_{i}$, $\Theta_{i}(x)$ is a diagonal matrix with entries in $[0,1]$ whose final $N$ entries are all $0$ and such that at most one of every $N$ consecutive diagonal entries is non-zero;
		\item for all $1\leq i\leq  l $ and $1\leq k\leq n_{i}$, if $x\notin U_{i,k}$, then $\Theta_{i}(x)_{k,k}=0$;
		\item for all $1\leq i\leq  l $ and $1\leq k\leq n_{i}$, if $x\in B_{i,k}$, then $\Theta_{i}(x)_{k,k}=1$. 
	\end{enumerate}
	
	Since $N<\mathfrak{s}(A)$, we may, for each $1\leq i\leq  l $,  define $W_{n_{i}}\in C([0,1]^{n_{i}},M_{n_{i}})$ as in \autoref{ltmatrix}. For $1\leq i\leq l$ and $x\in X_{i}$, define the unitary $V_{i}\in C(X_{i},M_{n_{i}})$ by 
	$$
	V_{i}(x):=W_{n_{i}}(\Theta_{i}(x)_{1,1},\ldots,\Theta_{i}(x)_{n_{i},n_{i}})
	$$
	and set $V:=(V_{1},\ldots,V_{\ell})$. Let us first argue that $(fV)_{i}(x)$ is strictly lower triangular for all $1\leq i\leq  l $ and $x\in X_{i}$, and then show that $V\in A$.
		
	Fix $1\leq i\leq  l $ and $x\in X_{i}$. From \autoref{Wn}, we have
	\begin{equation}
	\label{fV}
	\begin{aligned}
	(fV)_{i}(x)&=f_{i}(x)W_{n_{i}}(\Theta_{i}(x)_{1,1},\ldots,\Theta_{i}(x)_{n_{i},n_{i}})\\
	&=f_{i}(x)U[\gamma_{1,n_{i}}^{n_{i}}]^{N}\prod_{k=N}^{n_{i}-1}u^{n_{i}}_{k,n_{i}}(\Theta_{i}(x)_{k+1,k+1}).
	\end{aligned}
	\end{equation}
	If we write $f_{i}(x)=\big[\, C_{1}\, |\, \cdots\, |\, C_{n_{i}}\, \big]$, where $C_{j}$ is the $j$th column of $f_{i}(x)$, then $f_{i}(x)U[\gamma_{1,n_{i}}^{n_{i}}]^{N}=\big[C_{N+1}\, |\, \cdots\, |\, C_{n_{i}}\, |\, C_{1}\, |\, \cdots\, |\, C_{N}\, \big]$. By the assumption of the lemma, $\mathfrak{r}(f_{i}(x))\leq N$. Hence, all non-zero entries in the first $n_{i}-N$ columns of the matrix $f_{i}(x)U[\gamma_{1,n_{i}}^{n_{i}}]^{N}$ must lie strictly below the diagonal. But by \autoref{Bik}, $x\in B_{i,1}$ and, hence, by the assumptions of the lemma, $f_{i}(x)$ has zero crosses at indices $1,\ldots, N$. In particular, the columns $C_{1},\ldots,C_{N}$ consist entirely of zeros. Therefore, $f_{i}(x)U[\gamma_{1,n_{i}}^{n_{i}}]^{N}$ is strictly lower triangular. To show that $(fV)_{i}(x)$ is strictly lower triangular, we thus only need to verify that $(fV)_{i}(x)=f_{i}(x)U[\gamma_{1,n_{i}}^{n_{i}}]^{N}$. To do this, it is enough, by \autoref{fV}, to check that for each integer $N\leq k\leq n_{i}-1$, 
	\begin{equation}
	\label{final}
	f_{i}(x)U[\gamma_{1,n_{i}}^{n_{i}}]^{N}u^{n_{i}}_{k,n_{i}}(\Theta_{i}(x)_{k+1,k+1})=f_{i}(x)U[\gamma_{1,n_{i}}^{n_{i}}]^{N}.
	\end{equation}
	
	To this end, fix $N\leq k\leq n_{i}-1$. If $\Theta_{i}(x)_{k+1,k+1}=0$, then $u^{n_{i}}_{k,n_{i}}(\Theta_{i}(x)_{k+1,k+1})=1_{n_{i}}$ and there is nothing to show, and so we may assume $\Theta_{i}(x)_{k+1,k+1}>0$. Then, by (II) above, necessarily $x\in U_{i,k+1}$. Hence, by the assumption of the lemma, $f_{i}(x)$ has zero crosses at indices $k+1, \ldots,k+N$, from which it follows that the columns $C_{k+1},\ldots,C_{k+N}$ consist entirely of zeros. As noted above, these correspond to the columns of $f_{i}(x)U[\gamma_{1,n_{i}}^{n_{i}}]^{N}$ at indices $k+1-N,\ldots,k$. By \autoref{transposition prod}, the columns of $f_{i}(x)U[\gamma_{1,n_{i}}^{n_{i}}]^{N}u^{n_{i}}_{k,n_{i}}(\Theta_{i}(x)_{k+1,k+1})$ at indices $k-N+1,\ldots,k$ and $n_{i}-N+1,\ldots,n_{i}$ are linear combinations of the same set of columns of $f_{i}(x)U[\gamma_{1,n_{i}}^{n_{i}}]^{N}$ (i.e., of $C_{k+1},\ldots,C_{k+N},C_{1},\ldots,C_{N}$). But every column in this latter set consists entirely of zeros. Hence, since multiplying by the unitary $u^{n_{i}}_{k,n_{i}}(\Theta_{i}(x)_{k+1,k+1})$ on the right only alters columns at indices $k-N+1,\ldots,k$ and $n_{i}-N+1,\ldots,n_{i}$, \autoref{final} holds. 
	
	Let us now ensure that $V\in A$. Suppose that $2\leq i\leq  l $ and that $y\in Y_{i}$ decomposes into $x_{1}\in X_{i_{1}}\setminus Y_{i_{1}},\ldots,x_{s}\in X_{i_{s}}\setminus Y_{i_{s}}$. Then,
	$$
	\begin{aligned}
	&(\Theta_{i}(y)_{1,1},\ldots,\Theta_{i}(y)_{n_{i},n_{i}})\\
	=&(\Theta_{i_{1}}(x_{1})_{1,1},\ldots,\Theta_{i_{1}}(x_{1})_{n_{i_{1}},n_{i_{1}}},\ldots\ldots\ldots,\Theta_{i_{s}}(x_{s})_{1,1},\ldots,\Theta_{i_{s}}(x_{s})_{n_{i_{s}},n_{i_{s}}}).
	\end{aligned}
	$$
	Define $B(y):=\{1\leq k\leq n_{i}:y\in B_{i,k}\}$. By \autoref{Bik}, $B(y)=\{1=k_{1}<\cdots<k_{s}\}$, where $k_{t}= 1+n_{i_{1}}+\cdots+n_{i_{t-1}}$ for $1\leq t\leq s$. Set $k_{s+1}:=n_{i}+1$ and note that $k_{t+1}-k_{t}=n_{i_{t}}$ for all $1\leq t\leq s$. By assumption $n_{i}\geq\mathfrak{s}(A)>N$, and so, in light of (I) and (III) above, we may apply \autoref{Wn reduction} with the vector $(\Theta_{i}(y)_{1,1},\ldots,\Theta_{i}(y)_{n_{i},n_{i}})$ and the set $B(y)$, to obtain
	$$
	\begin{aligned}
	V_{i}(y)&=W_{n_{i}}(\Theta_{i}(y)_{1,1},\ldots,\Theta_{i}(y)_{n_{i},n_{i}})\\	&=\diag\left(W_{n_{i_{1}}}(\Theta_{i_{1}}(x_{1})_{1,1},\ldots,\Theta_{i_{1}}(x_{1})_{n_{i_{1}},n_{i_{1}}}),\ldots,W_{i_{s}}(\Theta_{i_{s}}(x_{s})_{1,1},\ldots,\Theta_{i_{s}}(x_{s})_{n_{i_{s}},n_{i_{s}}})\right)\\	&=\diag(V_{i_{1}}(x_{1}),\ldots,V_{i_{s}}(x_{s})).
	\end{aligned}
	$$
	Therefore, $V\in A$. This completes the proof of \autoref{lower triangular}.
\end{proof}

\subsection{Proof of the Main Theorem}
\label{sc: main proof}

\begin{theorem}
	\label{main}
	Suppose 
	$$
	A_{1}\overset{\psi_{1}}{\longrightarrow} A_{2}\overset{\psi_{2}}{\longrightarrow} A_{3}\overset{\psi_{3}}{\longrightarrow}\cdots\longrightarrow A:=\varinjlim A_{i}
	$$
	is a simple inductive limit of DSH algebras with diagonal bonding maps. Then $A$ has stable rank one. 	
\end{theorem}

\begin{proof}
	For $n\in\NN$, let $\mu_{n}\colon A_{n}\to A$ denote the map in the construction of the inductive limit, which is unital (since the bonding maps are) and, by \autoref{injective} we may assume, injective. Furthermore, we may assume that the $A_{j}$'s are infinite-dimensional.
	
	Fix $\epsilon >0$ and $a\in A$. Our goal is to find an invertible element $a'\in A$ with $\|a-a'\|\leq\epsilon$. To start, choose $j\in\NN$ and $f\in A_{j}$ such that $\|a-\mu_{j}(f)\|\leq\epsilon/4$. If $f$ is invertible in $A_{j}$, then $\mu_{j}(f)$ is invertible in $A$, in which case we are finished. Thus, we may assume that $f$ is not invertible in $A_{j}$. 
	
	Since $A_{j}$ is infinite-dimensional, we may apply \autoref{far out} with $f$, $\epsilon/4$, and $N=\mathfrak{S}(A_{j})+M+1$, where $M$ is the natural number depending on $f$ and $\epsilon/4$, coming from the statement of \autoref{far out}. This yields a function $f'\in A_{j}$ with $\|f-f'\|\leq\epsilon/4$, a $j'>j$ such that $\mathfrak{s}(A_{j'})>NM$, and unitaries $V,V'\in A_{j'}$ with the following two properties (we adopt the same notation for the decomposition of $A_{j'}$ introduced just above \autoref{far out}):
	\begin{enumerate}
		\item for any $1\leq i\leq  l $ and $1\leq k\leq n_{i}$, there is a (possibly empty) open subset $U_{i,k}$ of $X_{i}$ containing $B_{i,k}$ such that $(V\psi_{j',j}(f')V')_{i}(x)$ has zero crosses at indices $k, k+M, k+2M,\ldots,k+(N-1)M$;
		\item for all $1\leq i\leq l $ and $x\in X_{i}$, we have $\mathfrak{r}((V\psi_{j',j}(f')V')_{i}(x))\leq \mathfrak{S}(A_{j})+M-1$.
	\end{enumerate}
	Let $f'':=V\psi_{j',j}(f')V'\in A_{j'}$. 
	
	Next, apply \autoref{block on set} with $A_{j'}$, $f''$, and $\epsilon/4$. This yields a function $g\in A_{j'}$ with $\|g-f''\|\leq \epsilon/4$ and, for $1\leq i\leq  l $ and $1\leq k\leq n_{i}$, open sets $\mathcal{O}_{i,k}\subset X_{i}$ containing $B_{i,k}$ on which $g_{i}$ always has a block point at index $k$; moreover, for all $1\leq i\leq l $ and $x\in X_{i}$, the matrix $g_{i}(x)$ has a zero cross at every index that $f''_{i}(x)$ does, and $\mathfrak{r}(g_{i}(x))\leq\mathfrak{r}(f''_{i}(x))\leq\mathfrak{S}(A_{j})+M-1$. Thus, intersecting the $\mathcal{O}_{i,k}$'s with the $U_{i,k}$'s, we may assume that $g_{i}(x)$ has zero crosses at indices $k,k+M,\ldots,k+(N-1)M$ whenever $x\in \mathcal{O}_{i,k}$. 
	
	Since $\mathfrak{s}(A_{j'})>NM$, we may now apply \autoref{cross shift} on $A_{j'}$ with $g$ and the $\mathcal{O}_{i,k}$'s above to obtain a unitary $W\in A_{j'}$ with the following properties:
	\begin{enumerate}[label=(\Roman*)]
		\item for all $1\leq i\leq l $ and $1\leq k\leq n_{i}$, there are open sets $\mathcal{O}_{i,k}'\supset B_{i,k}$ in $X_{i}$ such that $W_{i}(x)g_{i}(x)W_{i}(x)^{*}$ has zero crosses at indices $k,k+1,\ldots,k+N-1$ whenever $x\in \mathcal{O}_{i,k}'$;
		\item $\mathfrak{r}(W_{i}(x)g_{i}(x)W_{i}(x)^{*})\leq \mathfrak{r}(g_{i}(x))+2\leq\mathfrak{S}(A_{j})+M+1=N$ for all $1\leq i\leq  l $ and $x\in X_{i}$.
	\end{enumerate} 
	Let $g':=WgW^{*}\in A_{j'}$. 
	
	Using these properties and the fact that $\mathfrak{s}(A_{j'})>NM\geq N$, we may apply \autoref{lower triangular} on $A_{j'}$ with $g'$ and the $\mathcal{O}_{i,k}'$'s to conclude that there is a unitary $W'\in A_{j'}$ such that for all $1\leq i\leq  l $ and $x\in X_{i}$, the matrix $(g'W')_{i}(x)$ is strictly lower triangular. Thus, $g'W'$ is a nilpotent element. As observed in Section 4 of \cite{Ror}, every nilpotent element of a unital $\mathrm{C}^*$-algebra is arbitrarily close to an invertible element. Thus, there is an invertible element $h\in A_{j'}$ such that $\|g'W'-h\|\leq\epsilon/4$.
	
	Take $a':=\mu_{j'}(V^{*}W^{*}h(W')^{*}W(V')^{*})$ and observe that $a'$ is invertible in $A$. Then, since the $\mu_{n}$'s are injective,
	$$
	\begin{aligned}
	\|\mu_{j}(f')-a'\|&=\|\psi_{j',j}(f')-V^{*}W^{*}h(W')^{*}W(V')^{*}\|\\
	&=\|V^{*}W^{*}[WV\psi_{j',j}(f')V'W^{*}W'-h](W')^{*}W(V')^{*}\|\\
	&\leq\|V^{*}W^{*}\|\|Wf''W^{*}W'-h\|\|(W')^{*}W(V')^{*}\|\\
	&=\|Wf''W^{*}W'-h\|\\
	&\leq\|Wf''W^{*}W'-WgW^{*}W'\|+\|WgW^{*}W'-h\|\\
	&\leq\|W\|\|f''-g\|\|W^{*}W'\|+\|g'W'-h\|\\
	&\leq\frac{\epsilon}{4}+\frac{\epsilon}{4}\\
	&=\frac{\epsilon}{2},
	\end{aligned}
	$$
	and
	$$
	\begin{aligned}
	\|a-\mu_{j}(f')\|&\leq\|a-\mu_{j}(f)\|+\|\mu_{j}(f)-\mu_{j}(f')\|
	\leq \frac{\epsilon}{4}+\|f-f'\|
	\leq\frac{\epsilon}{2}.
	\end{aligned}
	$$
	Therefore, 
	$$
	\|a-a'\|\leq \|a-\mu_{j}(f')\|+\|\mu_{j}(f')-a'\|\leq\frac{\epsilon}{2}+\frac{\epsilon}{2}=\epsilon,
	$$
	as desired. 
\end{proof}

\subsection{Applications to Dynamical Crossed Products}
\label{sc: cp}

Let $T$ be an infinite compact metric space and let $h\colon T\to T$ be a minimal homeomorphism. In the final portion of this paper, we present two applications of  \autoref{main}, both concerning the dynamical crossed product $A:=\mathrm{C^*}(\mathbb{Z},T,h)$. We show that $A$ has stable rank one (see \autoref{main2}), thereby affirming a conjecture of Archey, Niu, and Phillips (see Conjecture 7.2 of \cite{arcphil}). We also apply a result of Thiel from \cite{thi} to conclude that classification for $A$ is determined by strict comparison (see \autoref{main3}), which establishes the Toms--Winter Conjecture for minimal dynamical crossed products. 

The Toms--Winter Conjecture dates back to 2008 and stipulates that for separable, unital, simple, non-elementary, nuclear $\mathrm{C}^{*}$-algebras three different notions of regularity are equivalent. Below is a precise statement of the conjecture.

\begin{conjecture}[Toms--Winter; \cite{ET}, \cite{WZ}]
	\label{Toms-Winter}
	Let $A$ be a separable, unital, simple, non-elementary, nuclear $\mathrm{C}^{*}$-algebra. The following statements are equivalent.
	\begin{enumerate}
		\item $A$ has finite nuclear dimension.
		\item $A$ is $\mathcal{Z}$-stable; that is, $A\otimes \mathcal{Z}\cong A$.
		\item $A$ has strict comparison of positive elements. 
	\end{enumerate}
\end{conjecture}

At the time, part of this conjecture had already been established by R\o rdam, who, in Theorem 4.5 and Corollary 4.6 of \cite{Ror2}, proved that \textit{(2)} implies \textit{(3)}. Winter showed in Corollary 7.3 of \cite{Win} that \textit{(1)} implies \textit{(2)}. The work of various hands established that \textit{(2)} implies \textit{(1)} in special cases, but very recently, in Theorem A of \cite{CETWW}, this implication was shown to hold in full generality. Therefore, to establish the conjecture for a given $\mathrm{C}^{*}$-algebra, one needs only to check that strict comparison of positive elements yields $\mathcal{Z}$-stability.

Let $\sigma\colon C(T)\to C(T)$ denote the automorphism arising from $h$ given by $\sigma(f):=f\circ h^{-1}$. Let $u$ denote the unitary in the associated crossed product $A$ implementing the $\sigma$ action, i.e., $ufu^{*}=\sigma(f)$ for all $f\in C(T)$. Then $A$ is the $\mathrm{C}^{*}$-algebra generated by $C(T)$ and $u$. Given a closed set $S\subset T$ with non-empty interior, let $A_{S}$ denote the orbit-breaking sub-$\mathrm{C}^{*}$-algebra of $A$ associated to $S$, first introduced by Putnam in \cite{Put}  for Cantor minimal systems and later by Q. Lin and Phillips for more general minimal systems (see \cite{Lin}, \cite{LP}, and \cite{LP2}); that is, $A_{S}$ is the $\mathrm{C}^{*}$-algebra generated by $\{f,ug:f\in C(T), g\in C_{0}(T\setminus S)\}$, where we adopt the shorthand $C_{0}(T\setminus S):=\{g\in C(T): g|_{S}\equiv 0\}$. In \cite{Lin}, \cite{LP}, and \cite{LP2}, Q. Lin and Phillips showed that $A_{S}$ is a recursive subhomogeneous algebra, and in fact a DSH algebra. We outline this below. For a more in-depth discussion, see Theorems 3.1-3.3. of \cite{LP}.

Given $s\in S$, let $\lambda_{S}(s):=\min\{n>0:h^{n}(s)\in S\}$ (the first return time of $s$ to $S$). Since $T$ is compact, it follows that $\sup_{s\in S}\lambda_{S}(s)$ is finite (see also Lemma 2.2 of \cite{LP2}). Thus there exist $1\leq n_{1}^{S}<n_{2}^{S}<\cdots<n_{ l (S)}^{S}$ such that $\{\lambda_{S}(s):s\in S\}=\{n_{i}^{S}:1\leq i\leq  l (S)\}$. For $1\leq i\leq  l (S)$, let $X_{i}^{S}:=\overline{\lambda_{S}^{-1}(n_{i}^{S})}$ and $Y_{i}^{S}:=X_{i}^{S}\setminus\lambda_{S}^{-1}(n_{i}^{S})$. Then, for given $1\leq i \leq  l (S)$ and $y\in Y_{i}^{S}$, there are indices $1\leq i_{1},\ldots,i_{p}< i$ with $y\in X_{i_{1}}^{S}\setminus Y_{i_{1}}^{S}$, such that $n_{i_{1}}^{S}+\cdots+n_{i_{p}}^{S}=n_{i}^{S}$ and such that $h^{k}(y)\in S$ if and only if $k=n_{i_{1}}^{S}+\cdots +n_{i_{j}}^{S}$ for some $1\leq j\leq p$. Note, too, that $h^{n_{i_{1}}^{S}+\cdots +n_{i_{j-1}}^{S}}(y)\in X_{i_{j}}^{S}\setminus Y_{i_{j}}^{S}$ for all $2\leq j\leq p$. Then, $A_{S}$ is isomorphic to a sub-$\mathrm{C}^*$-algebra of $\bigoplus_{i=1}^{ l (S)}C(X_{i}^{S},M_{n_{i}^{S}})$, where an element $(f_{1},\ldots,f_{ l (S)})$ of $\bigoplus_{i=1}^{ l (S)}C(X_{i}^{S},M_{n_{i}^{S}})$ is in $A_{S}$ if and only if, for given $1\leq i\leq  l (S)$ and $y\in Y_{i}^{S}$,
$$
f_{i}(y)=\diag\left(f_{i_{1}}(y),f_{i_{2}}(h^{n_{i_{1}}^{S}}(y)),\ldots,f_{i_{p}}(h^{n_{i_{1}}^{S}+\cdots +n_{i_{p-1}}^{S}}(y))\right),
$$
where $i_{1},\ldots,i_{p}$ are as described above. It follows that $A_{S}$ is a DSH algebra. 

\begin{lemma}
	\label{returning}
	Suppose $R\subset S\subset T$. Let $1\leq i\leq  l (R)$ and $x\in X_{i}^{R}\setminus Y_{i}^{R}=\lambda_{R}^{-1}(n_{i}^{R})$. Then, there are indices $1\leq i_{1},\ldots,i_{q}\leq  l (S)$ such that:
	\begin{enumerate}
		\item $n_{i_{1}}^{S}+\cdots+n_{i_{q}}^{S}=n_{i}^{R}$;
		\item for all $1\leq k\leq n_{i}^{R}$, $h^{k}(x)\in S$ if and only if $k=n_{i_{1}}^{S}+\cdots+n_{i_{j}}^{S}$ for some $1\leq j\leq q$;
		\item $x\in X_{i_{1}}^{S}\setminus Y_{i_{1}}^{S}$ and, for all $2\leq j\leq q$, $h^{n_{i_{1}}^{S}+\cdots+n_{i_{j-1}}^{S}}(x)\in X_{i_{j}}^{S}\setminus Y_{i_{j}}^{S}$.
	\end{enumerate}
\end{lemma}

\begin{proof}
	Since $R\subset S$ and the sets $X_{j}^{S}\setminus Y_{j}^{S}$ for $1\leq j\leq  l (S)$ partition $S$, there is a unique $1\leq i_{1}\leq  l (S)$ such that $x\in X_{i_{1}}^{S}\setminus Y_{i_{1}}^{S}$. Moreover, $n_{i_{1}}^{S}=\lambda_{S}(x)\leq \lambda_{R}(x)=n_{i}^{R}$. If $n_{i_{1}}^{S}=n_{i}^{R}$, then there is nothing to show. Otherwise, there is an $i_{2}$ such that $h^{n_{i_{1}}^{S}}(x)\in X_{i_{2}}^{S}\setminus Y_{i_{2}}^{S}$. Note that $n_{i_{2}}^{S}=\lambda_{S}(h^{n_{i_{1}}^{S}}(x))\leq n_{i}^{R}-n_{i_{1}}^{S}$. If $n_{i_{2}}^{S}=n_{i}^{R}-n_{i_{1}}^{S}$, the desired result follows. Otherwise, we let $i_{3}$ be such that $h^{n_{i_{1}}^{S}+n_{i_{2}}^{S}}(x)\in X_{i_{3}}^{S}\setminus Y_{i_{3}}^{S}$ and proceed as before. Eventually, this process terminates (when $n_{i_{1}}^{S}+\cdots+n_{i_{q}}^{S}=n_{i}^{R}$) and yields indices with the desired properties. This proves the lemma. 
\end{proof}

By \cite{LP2}, Proposition 2.4, there is a unique homomorphism $\gamma_{S}\colon A_{S}\to \bigoplus_{i=1}^{ l (S)}C(X_{i}^{S},M_{n_{i}^{S}})$ with the property that for $f\in C(T)$ and $g\in C_{0}(T\setminus S)$, 
\begin{equation}
	\label{function identification}
	\gamma_{S}(f)_{i}=\diag\left(f\circ h|_{X_{i}^{S}},f\circ h^{2}|_{X_{i}^{S}},\ldots,f\circ h^{n_{i}^{S}}|_{X_{i}^{S}}\right)
\end{equation}
and
\begin{equation}
	\label{ufunction identification}
	\begin{aligned}
		\gamma_{S}(ug)_{i}=
		\begin{pmatrix}
			0 & & & & \\
			g\circ h|_{X_{i}^{S}} & 0 & & & \\
			& g\circ h^{2}|_{X_{i}^{S}} & \ddots & &\\
			& & \ddots & 0 &\\
			& & & g\circ h^{n_{i}^{S}-1}|_{X_{i}^{S}} & 0
		\end{pmatrix}	
	\end{aligned}
\end{equation}
for $1\leq i\leq  l (S)$. 

Now, fix $R\subset S\subset T$. By examining the generating sets, it follows that $A_{S}$ is contained in $A_{R}$. Let $\psi\colon A_{S}\to A_{R}$ denote the inclusion map. 

\begin{lemma}
	\label{diagonal inclusion}
	$\psi$ is a diagonal map (see \autoref{diagonal map}) between DSH algebras. 
	\label{putnam diagonal}
\end{lemma}

\begin{proof}
	Fix $1\leq i\leq  l (R)$ and $x\in X_{i}^{R}\setminus Y_{i}^{R}=\lambda_{R}^{-1}(n_{i}^{R})$. By \autoref{returning}, there are indices $1\leq i_{1},\ldots,i_{q}\leq  l (S)$ such that:
	\begin{enumerate}
		\item $n_{i_{1}}^{S}+\cdots+n_{i_{q}}^{S}=n_{i}^{R}$;
		\item for all $1\leq k\leq n_{i}^{R}$, $h^{k}(x)\in S$ if and only if $k=\beta_{j}$ for some $1\leq j\leq q$, where $\beta_{j}:=n_{i_{1}}^{S}+\cdots+n_{i_{j}}^{S}$;
		\item for all $1\leq j\leq q$, $h^{\beta_{j-1}}(x)\in X_{i_{j}}^{S}\setminus Y_{i_{j}}^{S}$ (here $\beta_{0}:=0$, so that $x\in X_{i_{1}}^{S}\setminus Y_{i_{1}}^{S}$).
	\end{enumerate}	
	Let us show that $x$ decomposes into $h^{\beta_{0}}(x)=x,h^{\beta_{1}}(x),\ldots,h^{\beta_{q-1}}(x)$ under $\psi$. 
	
	Suppose that $f\in C(T)$. Let us begin by verifying that
	\begin{equation}
		\label{diagonal f}
		\begin{aligned}
			\gamma_{R}(\psi(f))_{i}(x)=\diag\left(\gamma_{S}(f)_{i_{1}}(x),\gamma_{S}(f)_{i_{2}}(h^{\beta_{1}}(x)),\ldots,\gamma_{S}(f)_{i_{q}}(h^{\beta_{q-1}}(x))\right).
		\end{aligned}
	\end{equation}
	Fix $1\leq j\leq q$. By \autoref{function identification},  
	$$
	\begin{aligned}
	\gamma_{S}(f)_{i_{j}}(h^{\beta_{j-1}}(x))&=\diag\left(f(h(h^{\beta_{j-1}}(x))),\ldots,f(h^{n_{i_{j}}^{S}}(h^{\beta_{j-1}}(x)))\right)\\
	&=\diag\left(f(h^{\beta_{j-1}+1}(x)),\ldots,f(h^{\beta_{j}}(x))\right).
	\end{aligned}
	$$
	Hence, 
	$$
	\begin{aligned}
	&\diag\left(\gamma_{S}(f)_{i_{1}}(x),\gamma_{S}(f)_{i_{2}}(h^{\beta_{1}}(x)),\ldots,\gamma_{S}(f)_{i_{q}}(h^{\beta_{q-1}}(x))\right)\\
	=&\diag\left(f(h^{\beta_{0}+1}(x)),\ldots,f(h^{\beta_{1}}(x)),\ldots\ldots,f(h^{\beta_{q-1}+1}(x)),\ldots,f(h^{\beta_{q}}(x))\right)\\
	=&\diag\left(f(h(x)),\ldots,f(h^{n_{i}^{R}}(x))\right)\\
	=&\gamma_{R}(f)_{i}(x),
	\end{aligned}
	$$
	which yields \autoref{diagonal f}.
	
	Next, suppose $g\in C_{0}(T\setminus S)$. Let us show that 
	\begin{equation}
		\label{diagonal ug}
		\begin{aligned}
			\gamma_{R}(\psi(ug))_{i}(x)=\diag\left(\gamma_{S}(ug)_{i_{1}}(x),\gamma_{S}(ug)_{i_{2}}(h^{\beta_{1}}(x)),\ldots,\gamma_{S}(ug)_{i_{q}}(h^{\beta_{q-1}}(x))\right).
		\end{aligned}
	\end{equation}
	By \autoref{ufunction identification}, 
	$$
	\begin{aligned}
	\gamma_{R}(ug)_{i}(x)=
	\begin{pmatrix}
	0 & & & & \\
	g(h(x)) & 0 & & & \\
	& g(h^{2}(x)) & \ddots & &\\
	& & \ddots & 0 &\\
	& & & g(h^{n_{i}^{R}-1}(x)) & 0
	\end{pmatrix}	
	\end{aligned}.
	$$
	For $1\leq j\leq q$, we have $h^{\beta_{j}}(x)\in S$ by property (2) above, so that $g(h^{\beta_{j}}(x))=0$. Hence, partitioning $\{1,2,\ldots,n_{i}^{R}\}$ into the sets $\{\beta_{j-1}+1,\ldots,\beta_{j}\}$ for $1\leq j\leq q$, we may view $\gamma_{R}(ug)_{i}(x)$ as a block-diagonal matrix $\diag(B_{1},\ldots,B_{q})$, where
	$$
	\begin{aligned}
	B_{j}&:=\begin{pmatrix}
	0 & & & & \\
	g(h^{\beta_{j-1}+1}(x)) & 0 & & & \\
	& g(h^{\beta_{j-1}+2}(x)) & \ddots & &\\
	& & \ddots & 0 &\\
	& & & g(h^{\beta_{j}-1}(x)) & 0
	\end{pmatrix}
	=\gamma_{S}(ug)_{i_{j}}(h^{\beta_{j-1}}(x)),
	\end{aligned}
	$$
	which yields \autoref{diagonal ug}.
	
	We have shown that $x$ decomposes into $h^{\beta_{0}}(x)=x,h^{\beta_{1}}(x),\ldots,h^{\beta_{q-1}}(x)$ under $\psi$ on the generators of $A_{S}$. Let us now use continuity to prove that this decomposition is maintained for all elements of $A_{S}$. Let $a\in A_{S}$ be arbitrary. By definition, we may write $a=\lim_{n\to\infty}w_{n}$, where for each $n\in\NN$, $w_{n}$ is a word in $C(T)\cup uC_{0}(T\setminus S)\cup C_{0}(T\setminus S)u^{*}$. By \autoref{diagonal f} and \autoref{diagonal ug}, 
	$$
	\begin{aligned}
	\gamma_{R}(\psi(w_{n}))_{i}(x)
	=\diag\left(\gamma_{S}(w_{n})_{i_{1}}(x),\gamma_{S}(w_{n})_{i_{2}}(h^{\beta_{1}}(x)),\ldots,\gamma_{S}(w_{n})_{i_{q}}(h^{\beta_{q-1}}(x))\right)
	\end{aligned}
	$$
	for all $n\in\NN$. Hence, by continuity of $^*$-homomorphisms,
	$$
	\begin{aligned}
	\gamma_{R}(\psi(a))_{i}(x)
	&=\lim_{n\to\infty}\gamma_{R}(\psi(w_{n}))_{i}(x)\\
	&=\lim_{n\to\infty}\diag\left(\gamma_{S}(w_{n})_{i_{1}}(x),\gamma_{S}(w_{n})_{i_{2}}(h^{\beta_{1}}(x)),\ldots,\gamma_{S}(w_{n})_{i_{q}}(h^{\beta_{q-1}}(x))\right)\\
	&=\diag\left(\gamma_{S}(a)_{i_{1}}(x),\gamma_{S}(a)_{i_{2}}(h^{\beta_{1}}(x)),\ldots,\gamma_{S}(a)_{i_{q}}(h^{\beta_{q-1}}(x))\right),
	\end{aligned}
	$$
	which yields the desired diagonal decomposition and completes the proof of \autoref{diagonal inclusion}.
\end{proof}

\begin{theorem}
	\label{Ax}
	Let $T$ be an infinite compact metric space and let $h\colon T\to T$ be a minimal homeomorphism. Given a non-isolated point $x\in T$, the orbit-breaking subalgebra $A_{\{x\}}$ of $A:=\mathrm{C^*}(\mathbb{Z},T,h)$ is a simple inductive limit of DSH algebras with diagonal maps. In particular, $A_{\{x\}}$ has stable rank one.
\end{theorem}

\begin{proof}
	Choose a sequence $S_{1}\supset S_{2}\supset\cdots$ of closed sets with non-empty interior such that $\bigcap_{n=1}^{\infty}S_{n}=\{x\}$. For each $n\in\NN$, let $A_{S_{n}}\subset A$ denote the subalgebra as described above and let $\psi_{n}\colon A_{S_{n}}\to A_{S_{n+1}}$ denote the canonical inclusion. Since $\overline{\bigcup_{n=1}^{\infty}A_{S_{n}}}=A_{\{x\}}$, it follows by \autoref{putnam diagonal} that $A_{\{x\}}$ is an inductive limit of DSH algebras with diagonal maps. Moreover, by Theorem 1.2 of \cite{LP},  $A_{\{x\}}$ is simple (see Proposition 2.5 of \cite{xinP} for a proof). Therefore, by \autoref{main}, $A_{\{x\}}$ has stable rank one.
\end{proof}

\begin{corollary}[cf. \cite{arcphil}, Conjecture 7.2]
	\label{main2}
	Let $T$ be an infinite compact metric space and let $h\colon T\to T$ be a minimal homeomorphism. The dynamical crossed product $A:=\mathrm{C^*}(\mathbb{Z},T,h)$ has stable rank one. 
\end{corollary}

\begin{proof}
	Let $x$ be any non-isolated point in $T$. By \autoref{Ax}, $A_{\{x\}}$ has stable rank one. Since $h$ is minimal and $T$ is infinite, $h^{n}(x)\not=x$ for all $n\in\NN$. Thus, on combining Theorem 7.10 of \cite{phil2} with Theorem 4.6 of \cite{arcphil}, it follows that $A_{\{x\}}$ is a centrally large subalgebra of $A$. But by Theorem 6.3 of \cite{arcphil}, any infinite-dimensional unital simple separable $\mathrm{C}^*$-algebra containing a centrally large subalgebra with stable rank one must itself have stable rank one. 
\end{proof}	

\begin{corollary}
	\label{main3}
	Let $T$ be an infinite compact metric space and let $h\colon T\to T$ be a minimal homeomorphism. The dynamical crossed product $A:=\mathrm{C^*}(\mathbb{Z},T,h)$ is $\mathcal{Z}$-stable if and only if it has strict comparison of positive elements. 
\end{corollary}

\begin{proof}
	Let $x$ be any non-isolated point in $T$. By \autoref{Ax}, $A_{\{x\}}$ has stable rank one. Thus, by Theorem 9.6 of \cite{thi}, the Toms--Winter Conjecture (\autoref{Toms-Winter}) holds for $A_{\{x\}}$. In particular, $A_{\{x\}}$ is $\mathcal{Z}$-stable if and only if it has strict comparison of positive elements. But by Theorem 3.3 and Corollary 3.5 of \cite{ABP}, $A$ is $\mathcal{Z}$-stable if and only if $A_{\{x\}}$ is. Furthermore, by Theorem 6.14 of \cite{phil2}, $A$ has strict comparison if and only if $A_{\{x\}}$ does. Therefore, \autoref{main3} follows. 
\end{proof}

\begin{remark}
	\label{referee}
	Using the same ideas as those in the proof of \autoref{Ax}, one can show that the orbit-breaking simple subalgebras constructed by Deeley, Putnam, and Strung in \cite{DPS} also have stable rank one, despite possibly not being $\mathcal{Z}$-stable. Although, the closed subset of the underlying infinite compact metric space used in their construction need not be a singleton set, it still has the property that it meets every orbit exactly once, and thus, is a simple inductive limit of DSH algebras with diagonal maps.
\end{remark}

\subsection*{Acknowledgement}
The authors would like to thank Zhuang Niu, Chris Phillips, and Andrew Toms for a lot of helpful encouragement, as well as George Elliott and Maria Grazia Viola for all of their support and for proofreading an earlier draft of this paper. We are also very grateful to the referee for providing a lot of prudent comments concerning the structure of this paper and for bringing the observation in \autoref{referee} to our attention.

\end{document}